\documentclass[12pt]{article}
\usepackage[utf8]{inputenc}
\usepackage{amsmath}
\usepackage{amssymb}
\usepackage[margin = 1.5 in, marginpar = 3 cm]{geometry}
\usepackage{graphicx}
\usepackage{amsthm}
\usepackage{mathtools}
\mathtoolsset{showonlyrefs}
\usepackage{hyperref}
\hypersetup{colorlinks,linkcolor={blue},urlcolor=cyan}

\newcommand{\C}{\mathbb{C}}

\newcommand{\F}{\mathbb{F}}

\usepackage{mleftright}
\usepackage{amssymb}
\usepackage{amsmath}
\usepackage{tikz}
\usepackage{theoremref}
\usepackage{wasysym}
\usepackage{harpoon}
\usepackage{float}
\usepackage{caption}
\usepackage{blindtext}
\usepackage{geometry}
\usepackage[title]{appendix}

\def\graphScale{0.29}
\def\largeGraphScale{0.5}
\def\specialVertexColor{white}

\newcommand{\hamSandwich}{
\begin{tikzpicture}[scale=\graphScale]
    \draw (0, 0) -- (1, 0);
    \draw (0, 0) -- (0, 1);
    \draw (1, 0) -- (1, 1);
    \draw (1, 0) -- (0, 1);
    \draw (0, 1) -- (1, 1);
    \draw[fill=black] (0, 0) circle (0.125);
    \draw[fill=black] (0, 1) circle (0.125);
    \draw[fill=black] (1, 0) circle (0.125);
    \draw[fill=black] (1, 1) circle (0.125);
\end{tikzpicture}
}

\newcommand{\trianglePlusIsolated}{
\begin{tikzpicture}[scale=\graphScale]
    \draw (0, 0) -- (1, 0);
    \draw (0, 0) -- (1, 1);
    \draw (1, 0) -- (1, 1);
    \draw[fill=black] (0, 0) circle (0.125);
    \draw[fill=black] (0, 1) circle (0.125);
    \draw[fill=black] (1, 0) circle (0.125);
    \draw[fill=black] (1, 1) circle (0.125);
\end{tikzpicture}
}
\newcommand{\SuzukiThirtyTwoElusiveGraph}{
\begin{tikzpicture}[scale=\graphScale]
    \draw (0, 0) -- (1, 0);
    \draw (0, 0) -- (0, 1);
    \draw (0, 0) -- (1, 1);
    \draw (1, 0) -- (1, 1);
    \draw[fill=black] (0, 0) circle (0.125);
    \draw[fill=black] (0, 1) circle (0.125);
    \draw[fill=black] (1, 0) circle (0.125);
    \draw[fill=black] (1, 1) circle (0.125);
\end{tikzpicture}
}

\newcommand{\PSLhouse}{
\begin{tikzpicture}[scale=\largeGraphScale]
    \draw (0, 0) -- (1, 0);
    \draw (0, 0) -- (1, 1);
    \draw (0, 0) -- (0, 1);
    \draw (1, 0) -- (1, 1);
    \draw (0, 1) -- (1, 1);
    \draw (1.5, 0.5) -- (1, 0);
    \draw (1.5, 0.5) -- (1, 1);
    \draw[fill=black] (0, 0) circle (0.125);
    \draw[fill=black] (0, 1) circle (0.125);
    \draw[fill=black] (1, 0) circle (0.125);
    \draw[fill=black] (1, 1) circle (0.125);
    \draw[fill=\specialVertexColor] (1.5, 0.5) circle (0.125);
\end{tikzpicture}
}
\newcommand{\balloon}{
\begin{tikzpicture}[scale=\largeGraphScale]
    \draw (0, 0) -- (1, 1);
    \draw (0, 0) -- (0, 1);
    \draw (0, 1) -- (1, 1);
    \draw (1.5, 0.5) -- (1, 0);
    \draw (1.5, 0.5) -- (1, 1);
    \draw[fill=black] (0, 0) circle (0.125);
    \draw[fill=black] (0, 1) circle (0.125);
    \draw[fill=black] (1, 0) circle (0.125);
    \draw[fill=black] (1, 1) circle (0.125);
    \draw[fill=\specialVertexColor] (1.5, 0.5) circle (0.125);
\end{tikzpicture}
}
\newcommand{\codart}{
\begin{tikzpicture}[scale=\largeGraphScale]
    \draw (0, 0) -- (1, 1);
    \draw (0, 0) -- (0, 1);
    \draw (0, 1) -- (1, 1);
    \draw (1.5, 0.5) -- (1, 1);
    \draw[fill=black] (0, 0) circle (0.125);
    \draw[fill=black] (0, 1) circle (0.125);
    \draw[fill=black] (1, 0) circle (0.125);
    \draw[fill=black] (1, 1) circle (0.125);
    \draw[fill=\specialVertexColor] (1.5, 0.5) circle (0.125);
\end{tikzpicture}
}
\newcommand{\northStar}{
\begin{tikzpicture}[scale=\largeGraphScale]
    \draw (0, 0) -- (1, 0);
    \draw (0, 0) -- (1, 1);
    \draw (1, 0) -- (1, 1);
    \draw (1.5, 0.5) -- (1, 0);
    \draw (1.5, 0.5) -- (1, 1);
    \draw[fill=black] (0, 0) circle (0.125);
    \draw[fill=black] (0, 1) circle (0.125);
    \draw[fill=black] (1, 0) circle (0.125);
    \draw[fill=black] (1, 1) circle (0.125);
    \draw[fill=\specialVertexColor] (1.5, 0.5) circle (0.125);
\end{tikzpicture}
}
\newcommand{\dart}{
\begin{tikzpicture}[scale=\largeGraphScale]
    \draw (0, 0) -- (1, 0);
    \draw (0, 0) -- (1, 1);
    \draw (1, 0) -- (1, 1);
    \draw (0, 1) -- (1, 1);
    \draw (1.5, 0.5) -- (1, 0);
    \draw (1.5, 0.5) -- (1, 1);
    \draw[fill=black] (0, 0) circle (0.125);
    \draw[fill=black] (0, 1) circle (0.125);
    \draw[fill=black] (1, 0) circle (0.125);
    \draw[fill=black] (1, 1) circle (0.125);
    \draw[fill=\specialVertexColor] (1.5, 0.5) circle (0.125);
\end{tikzpicture}
}
\newcommand{\happyFace}{
\begin{tikzpicture}[scale=\largeGraphScale]
    \draw (1, 0) -- (1, 1);
    \draw (1.5, 0.5) -- (1, 0);
    \draw (1.5, 0.5) -- (1, 1);
    \draw[fill=black] (0, 0) circle (0.125);
    \draw[fill=black] (0, 1) circle (0.125);
    \draw[fill=black] (1, 0) circle (0.125);
    \draw[fill=black] (1, 1) circle (0.125);
    \draw[fill=\specialVertexColor] (1.5, 0.5) circle (0.125);
\end{tikzpicture}
}

\newcommand{\misshapenhouse}{
\begin{tikzpicture}[scale=\largeGraphScale]
    \draw (1, 0) -- (0, 1);
    \draw (1.5, 0.5) -- (1, 0);
    \draw (1.5, 0.5) -- (1, 1);
    \draw (0,1) -- (1,1);
    \draw (0,0) -- (1,0);
    \draw (0,0) -- (0,1);
    \draw[fill=black] (0, 0) circle (0.125);
    \draw[fill=black] (0, 1) circle (0.125);
    \draw[fill=black] (1, 0) circle (0.125);
    \draw[fill=black] (1, 1) circle (0.125);
    \draw[fill=\specialVertexColor] (1.5, 0.5) circle (0.125);
\end{tikzpicture}
}

\newcommand{\trianglestick}{
\begin{tikzpicture}[scale=\largeGraphScale]
    \draw (1, 0) -- (1, 1);
    \draw (1.5, 0.5) -- (1, 0);
    \draw (1.5, 0.5) -- (1, 1);
    \draw (0,1) -- (0,0);
    \draw[fill=black] (0, 0) circle (0.125);
    \draw[fill=black] (0, 1) circle (0.125);
    \draw[fill=black] (1, 0) circle (0.125);
    \draw[fill=black] (1, 1) circle (0.125);
    \draw[fill=\specialVertexColor] (1.5, 0.5) circle (0.125);
\end{tikzpicture}
}

\newcommand{\bull}{
\begin{tikzpicture}[scale=\largeGraphScale]
    \draw (1, 0) -- (1, 1);
    \draw (1.5, 0.5) -- (1, 0);
    \draw (1.5, 0.5) -- (1, 1);
    \draw (1,0) -- (0,0);
    \draw (0,1) -- (1,1);
    \draw[fill=black] (0, 0) circle (0.125);
    \draw[fill=black] (0, 1) circle (0.125);
    \draw[fill=black] (1, 0) circle (0.125);
    \draw[fill=black] (1, 1) circle (0.125);
    \draw[fill=\specialVertexColor] (1.5, 0.5) circle (0.125);
\end{tikzpicture}
}

\newcommand{\bowtiebruh}{
\begin{tikzpicture}[scale=\largeGraphScale]
    \draw (1, 0) -- (1, 1);
    \draw (1.5, 0.5) -- (1, 0);
    \draw (1.5, 0.5) -- (1, 1);
    \draw (1,1) -- (0,0);
    \draw (0,1) -- (1,1);
    \draw (0,0) -- (0,1);
    \draw[fill=black] (0, 0) circle (0.125);
    \draw[fill=black] (0, 1) circle (0.125);
    \draw[fill=black] (1, 0) circle (0.125);
    \draw[fill=black] (1, 1) circle (0.125);
    \draw[fill=\specialVertexColor] (1.5, 0.5) circle (0.125);
\end{tikzpicture}
}

\newcommand{\cricket}{
\begin{tikzpicture}[scale=\largeGraphScale]
    \draw (1, 0) -- (1, 1);
    \draw (1.5, 0.5) -- (1, 0);
    \draw (1.5, 0.5) -- (1, 1);
    \draw (1,0) -- (0,0);
    \draw (0,1) -- (1,0);
    \draw[fill=black] (0, 0) circle (0.125);
    \draw[fill=black] (0, 1) circle (0.125);
    \draw[fill=black] (1, 0) circle (0.125);
    \draw[fill=black] (1, 1) circle (0.125);
    \draw[fill=\specialVertexColor] (1.5, 0.5) circle (0.125);
\end{tikzpicture}
}

\newcommand{\bullcaseone}{
\begin{tikzpicture}[scale=\largeGraphScale]
    \draw (0, 0) -- (0, 1);
    \draw (0, 1) -- (1, 0);
    \draw (1, 0) -- (1, 1);
    \draw (0, 1) -- (1, 1);
    \draw (1, 1) -- (1.5, 0.5);
    \draw[fill=black] (0, 0) circle (0.125);
    \draw[fill=black] (0, 1) circle (0.125);
    \draw[fill=black] (1, 0) circle (0.125);
    \draw[fill=black] (1, 1) circle (0.125);
    \draw[fill=\specialVertexColor] (1.5, 0.5) circle (0.125);
\end{tikzpicture}
}

\newcommand{\housenobrim}{
\begin{tikzpicture}[scale=\largeGraphScale]
    \draw (0, 0) -- (1, 1);
    \draw (0, 0) -- (1, 0);
    \draw (0, 0) -- (0, 1);
    \draw (0, 1) -- (1, 1);
    \draw (0, 1) -- (1, 0);
    \draw (1, 1) -- (1.5, 0.5);
    \draw (1, 0) -- (1.5, 0.5);
    \draw[fill=black] (0, 0) circle (0.125);
    \draw[fill=black] (0, 1) circle (0.125);
    \draw[fill=black] (1, 0) circle (0.125);
    \draw[fill=black] (1, 1) circle (0.125);
    \draw[fill=\specialVertexColor] (1.5, 0.5) circle (0.125);
\end{tikzpicture}
}

\newcommand{\nobutt}{
\begin{tikzpicture}[scale=\largeGraphScale]
    \draw (0, 0) -- (1, 1);
    \draw (0, 0) -- (1, 0);
    \draw (0, 1) -- (1, 1);
    \draw (0, 1) -- (1, 0);
    \draw (1,1) -- (1,0);
    \draw (1, 1) -- (1.5, 0.5);
    \draw (1, 0) -- (1.5, 0.5);
    \draw[fill=black] (0, 0) circle (0.125);
    \draw[fill=black] (0, 1) circle (0.125);
    \draw[fill=black] (1, 0) circle (0.125);
    \draw[fill=black] (1, 1) circle (0.125);
    \draw[fill=\specialVertexColor] (1.5, 0.5) circle (0.125);
\end{tikzpicture}
}

\newcommand{\hamsandwichwitharm}{
\begin{tikzpicture}[scale=\largeGraphScale]
    \draw (0, 0) -- (1, 0);
    \draw (0, 0) -- (0, 1);
    \draw (0, 1) -- (1, 1);
    \draw (1,0) -- (1,1);
    \draw (0, 1) -- (1, 0);
    \draw (1, 1) -- (1.5, 0.5);
    \draw[fill=black] (0, 0) circle (0.125);
    \draw[fill=black] (0, 1) circle (0.125);
    \draw[fill=black] (1, 0) circle (0.125);
    \draw[fill=black] (1, 1) circle (0.125);
    \draw[fill=\specialVertexColor] (1.5, 0.5) circle (0.125);
\end{tikzpicture}
}

\newcommand{\hamSandwichWithOtherArm}{
\begin{tikzpicture}[scale=\largeGraphScale]
    \draw (0, 0) -- (1, 0);
    \draw (0, 0) -- (0, 1);
    \draw (0, 1) -- (1, 1);
    \draw (1,0) -- (1,1);
    \draw (0, 1) -- (1, 0);
    \draw (1, 0) -- (1.5, 0.5);
    \draw[fill=black] (0, 0) circle (0.125);
    \draw[fill=black] (0, 1) circle (0.125);
    \draw[fill=black] (1, 0) circle (0.125);
    \draw[fill=black] (1, 1) circle (0.125);
    \draw[fill=\specialVertexColor] (1.5, 0.5) circle (0.125);
\end{tikzpicture}
}

\newcommand{\emptyhouse}{
\begin{tikzpicture}[scale=\largeGraphScale]
    \draw (0, 0) -- (1, 0);
    \draw (0, 0) -- (0, 1);
    \draw (0, 1) -- (1, 1);
    \draw (1,0) -- (1,1);
    \draw (1, 1) -- (1.5, 0.5);
    \draw (1, 0) -- (1.5, 0.5);
    \draw[fill=black] (0, 0) circle (0.125);
    \draw[fill=black] (0, 1) circle (0.125);
    \draw[fill=black] (1, 0) circle (0.125);
    \draw[fill=black] (1, 1) circle (0.125);
    \draw[fill=\specialVertexColor] (1.5, 0.5) circle (0.125);
\end{tikzpicture}
}

\newcommand{\triangleIsolatedEdge}{
\begin{tikzpicture}[scale=\largeGraphScale]
    \draw (0, 0) -- (1, 0);
    \draw (0, 0) -- (0, 1);
    \draw (1, 0) -- (0, 1);
    \draw (1.5, 0.5) -- (1, 1);
    \draw[fill=black] (0, 0) circle (0.125);
    \draw[fill=black] (0, 1) circle (0.125);
    \draw[fill=black] (1, 0) circle (0.125);
    \draw[fill=black] (1, 1) circle (0.125);
    \draw[fill=\specialVertexColor] (1.5, 0.5) circle (0.125);
\end{tikzpicture}
}
\newcommand{\triangleWithTail}{
\begin{tikzpicture}[scale=\largeGraphScale]
    \draw (0, 0) -- (0, 1);
    \draw (0, 1) -- (1, 1);
    \draw (1,0) -- (1,1);
    \draw (1, 1) -- (1.5, 0.5);
    \draw (1, 0) -- (1.5, 0.5);
    \draw[fill=black] (0, 0) circle (0.125);
    \draw[fill=black] (0, 1) circle (0.125);
    \draw[fill=black] (1, 0) circle (0.125);
    \draw[fill=black] (1, 1) circle (0.125);
    \draw[fill=\specialVertexColor] (1.5, 0.5) circle (0.125);
\end{tikzpicture}
}
\newcommand{\triangleWithShortTail}{
\begin{tikzpicture}[scale=\largeGraphScale]
    \draw (0, 1) -- (1, 1);
    \draw (1,0) -- (1,1);
    \draw (1, 1) -- (1.5, 0.5);
    \draw (1, 0) -- (1.5, 0.5);
    \draw[fill=black] (0, 0) circle (0.125);
    \draw[fill=black] (0, 1) circle (0.125);
    \draw[fill=black] (1, 0) circle (0.125);
    \draw[fill=black] (1, 1) circle (0.125);
    \draw[fill=\specialVertexColor] (1.5, 0.5) circle (0.125);
\end{tikzpicture}
}

\newcommand{\zigZag}{
\begin{tikzpicture}[scale=\largeGraphScale]
    \draw (1, 0) -- (1, 1);
    \draw (1.5, 0.5) -- (1, 0);
    \draw (1,1) -- (0,0);
    \draw (0,1) -- (1,1);
    \draw (0,0) -- (0,1);
    \draw[fill=black] (0, 0) circle (0.125);
    \draw[fill=black] (0, 1) circle (0.125);
    \draw[fill=black] (1, 0) circle (0.125);
    \draw[fill=black] (1, 1) circle (0.125);
    \draw[fill=\specialVertexColor] (1.5, 0.5) circle (0.125);
\end{tikzpicture}
}

\newcommand{\triangleWithTwoTails}{
\begin{tikzpicture}[scale=\largeGraphScale]
    \draw (1, 0) -- (1, 1);
    \draw (1.5, 0.5) -- (1, 1);
    \draw (1,1) -- (0,0);
    \draw (0,1) -- (1,1);
    \draw (0,0) -- (0,1);
    \draw[fill=black] (0, 0) circle (0.125);
    \draw[fill=black] (0, 1) circle (0.125);
    \draw[fill=black] (1, 0) circle (0.125);
    \draw[fill=black] (1, 1) circle (0.125);
    \draw[fill=\specialVertexColor] (1.5, 0.5) circle (0.125);
\end{tikzpicture}
}

\newcommand{\badDart}{
\begin{tikzpicture}[scale=\largeGraphScale]
    \draw (0, 0) -- (0, 1);
    \draw (1, 0) -- (1, 1);
    \draw (1, 0) -- (0, 1);
    \draw (0, 1) -- (1, 1);
    \draw (1.5, 0.5) -- (1, 0);
    \draw (1.5, 0.5) -- (1, 1);
    \draw[fill=black] (0, 0) circle (0.125);
    \draw[fill=black] (0, 1) circle (0.125);
    \draw[fill=black] (1, 0) circle (0.125);
    \draw[fill=black] (1, 1) circle (0.125);
    \draw[fill=\specialVertexColor] (1.5, 0.5) circle (0.125);
\end{tikzpicture}
}
\newcommand{\hamSandwichWhite}{
\begin{tikzpicture}[scale=\graphScale]
    \draw (0, 0) -- (1, 0);
    \draw (0, 0) -- (0, 1);
    \draw (1, 0) -- (1, 1);
    \draw (1, 0) -- (0, 1);
    \draw (0, 1) -- (1, 1);
    \draw[fill=black] (0, 0) circle (0.125);
    \draw[fill=black] (0, 1) circle (0.125);
    \draw[fill=white] (1, 0) circle (0.125);
    \draw[fill=black] (1, 1) circle (0.125);
\end{tikzpicture}
}
\newcommand{\notIsolatedVertexWhite}{
\begin{tikzpicture}[scale=\graphScale]
    \draw (0, 0) -- (1, 0);
    \draw (0, 0) -- (0, 1);
    \draw (1, 0) -- (0, 1);
    \draw[fill=black] (0, 0) circle (0.125);
    \draw[fill=black] (0, 1) circle (0.125);
    \draw[fill=white] (1, 0) circle (0.125);
    \draw[fill=black] (1, 1) circle (0.125);
\end{tikzpicture}
}
\newcommand{\notIsolatedVertexWhiteGrey}{
\begin{tikzpicture}[scale=\graphScale]
    \draw (0, 0) -- (1, 0);
    \draw (0, 0) -- (0, 1);
    \draw (1, 0) -- (0, 1);
    \draw[fill=black] (0, 0) circle (0.125);
    \draw[fill=black] (0, 1) circle (0.125);
    \draw[fill=white] (1, 0) circle (0.125);
    \draw[fill=black] (1, 1) circle (0.125);
\end{tikzpicture}
}

\newcommand{\SuzukiThirtyTwoElusiveGraphWhiteBadb}{
\begin{tikzpicture}[scale=\graphScale]
    \draw (0, 0) -- (1, 0);
    \draw (0, 0) -- (0, 1);
    \draw (0, 0) -- (1, 1);
    \draw (1, 0) -- (1, 1);
    \draw[fill=black] (0, 0) circle (0.125);
    \draw[fill=black] (0, 1) circle (0.125);
    \draw[fill=white] (1, 0) circle (0.125);
    \draw[fill=black] (1, 1) circle (0.125);
\end{tikzpicture}
}
\newcommand{\notIsolatedVertexWhiteBad}{
\begin{tikzpicture}[scale=\graphScale]
    \draw (0, 0) -- (1, 0);
    \draw (0, 0) -- (0, 1);
    \draw (1, 0) -- (0, 1);
    \draw[fill=black] (0, 0) circle (0.125);
    \draw[fill=black] (0, 1) circle (0.125);
    \draw[fill=black] (1, 0) circle (0.125);
    \draw[fill=white] (1, 1) circle (0.125);
\end{tikzpicture}
}

\newcommand{\fletchingWhite}{
\begin{tikzpicture}[scale=\graphScale]
    \draw (0, 0) -- (1, 0);
    \draw (1, 0) -- (1, 1);
    \draw (1, 0) -- (0, 1);
    \draw[fill=black] (0, 0) circle (0.125);
    \draw[fill=black] (0, 1) circle (0.125);
    \draw[fill=white] (1, 0) circle (0.125);
    \draw[fill=black] (1, 1) circle (0.125);
\end{tikzpicture}
}
\newcommand{\szunrecognizei}{
\begin{tikzpicture}[scale=\graphScale]
    \draw (0, 0) -- (1, 0);
    \draw (0, 1) -- (1,1);
    \draw (1, 0) -- (1, 1);
    \draw (1, 0) -- (0, 1);
    \draw[fill=black] (0, 0) circle (0.125);
    \draw[fill=black] (0, 1) circle (0.125);
    \draw[fill=white] (1, 0) circle (0.125);
    \draw[fill=black] (1, 1) circle (0.125);
\end{tikzpicture}
}
\newcommand{\szunrecognizeii}{
\begin{tikzpicture}[scale=\graphScale]
    \draw (0, 0) -- (1, 0);
    \draw (1, 0) -- (1, 1);
    \draw[fill=black] (0, 0) circle (0.125);
    \draw[fill=black] (0, 1) circle (0.125);
    \draw[fill=white] (1, 0) circle (0.125);
    \draw[fill=black] (1, 1) circle (0.125);
\end{tikzpicture}
}
\newcommand{\trianglePlusIsolatedExtra}{
\begin{tikzpicture}[scale=\graphScale]
    \draw (1, 0) -- (0, 1);
    \draw (1, 0) -- (1, 1);
    \draw (0, 1) -- (1, 1);
    \draw[fill=black] (0, 0) circle (0.125);
    \draw[fill=black] (0, 1) circle (0.125);
    \draw[fill=white] (1, 0) circle (0.125);
    \draw[fill=black] (1, 1) circle (0.125);
\end{tikzpicture}
}
\newcommand{\szunrecognizeiii}{
\begin{tikzpicture}[scale=\graphScale]
    \draw (1, 0) -- (0, 1);
    \draw (1, 0) -- (0, 0);
    \draw (0, 1) -- (0, 0);
    \draw (1, 0) -- (1, 1);
    \draw[fill=black] (0, 0) circle (0.125);
    \draw[fill=black] (0, 1) circle (0.125);
    \draw[fill=white] (1, 0) circle (0.125);
    \draw[fill=black] (1, 1) circle (0.125);
\end{tikzpicture}
}
\newcommand{\brokenFrameIsolatediii}{
\begin{tikzpicture}[scale=\graphScale]
    \draw (1, 1) -- (0, 1);
    \draw (0, 1) -- (0, 0);
    \draw[fill=black] (0, 0) circle (0.125);
    \draw[fill=black] (0, 1) circle (0.125);
    \draw[fill=white] (1, 0) circle (0.125);
    \draw[fill=black] (1, 1) circle (0.125);
\end{tikzpicture}
}
\newcommand{\psltwosixteenimpossiblei}{
\begin{tikzpicture}[scale=\graphScale]
    \draw (1, 1) -- (0, 1);
    \draw (1, 1) -- (0, 0);
    \draw (0, 1) -- (0, 0);
    \draw[fill=black] (0, 0) circle (0.125);
    \draw[fill=black] (0, 1) circle (0.125);
    \draw[fill=white] (1, 0) circle (0.125);
    \draw[fill=black] (1, 1) circle (0.125);
\end{tikzpicture}
}
\newcommand{\psltwosixteenimpossibleii}{
\begin{tikzpicture}[scale=\graphScale]
    \draw (1, 1) -- (0, 1);
    \draw (1, 1) -- (0, 0);
    \draw (1, 1) -- (1, 0);
    \draw (0, 1) -- (0, 0);
    \draw[fill=black] (0, 0) circle (0.125);
    \draw[fill=black] (0, 1) circle (0.125);
    \draw[fill=white] (1, 0) circle (0.125);
    \draw[fill=black] (1, 1) circle (0.125);
\end{tikzpicture}
}
\newcommand{\psltwosixteenimpossibleiii}{
\begin{tikzpicture}[scale=\graphScale]
    \draw (1, 1) -- (0, 1);
    \draw (1, 1) -- (0, 0);
    \draw (1, 1) -- (1, 0);
    \draw (0, 0) -- (1, 0);
    \draw (0, 1) -- (0, 0);
    \draw[fill=black] (0, 0) circle (0.125);
    \draw[fill=black] (0, 1) circle (0.125);
    \draw[fill=white] (1, 0) circle (0.125);
    \draw[fill=black] (1, 1) circle (0.125);
\end{tikzpicture}
}

\newcommand{\szthirtytwoTrickyGraph}{
\begin{tikzpicture}[scale=\graphScale]
    \draw (0, 0) -- (1, 0);
    \draw (0, 0) -- (0, 1);
    \draw (0, 1) -- (1, 0);
    \draw (1, 0) -- (1, 1);
    \draw[fill=black] (0, 0) circle (0.125);
    \draw[fill=black] (0, 1) circle (0.125);
    \draw[fill=black] (1, 0) circle (0.125);
    \draw[fill=black] (1, 1) circle (0.125);
\end{tikzpicture}
}
\newcommand{\twoIsolatedpsl}{
\begin{tikzpicture}[scale=\graphScale]
    \draw (0, 0) -- (1, 0);
    \draw (0, 0) -- (0, 1);
    \draw (0, 1) -- (1, 0);
    \draw[fill=black] (0, 0) circle (0.125);
    \draw[fill=black] (0, 1) circle (0.125);
    \draw[fill=black] (1, 0) circle (0.125);
    \draw[fill=black] (1, 1) circle (0.125);
\end{tikzpicture}
}
\makeatletter
\let\@fnsymbol\@arabic
\makeatother

\title{Criteria for Classifying Prime Graphs of $\PSL(2, q)-$Solvable Groups}
\author{Thomas Michael Keller\thanks{Thomas Michael Keller: keller@txstate.edu; Department of Mathematics, Texas State University, 601 University Drive, San Marcos, TX, 78666, USA.}, Zachary Martin\footnote{Zachary Martin: zvmartin@willamette.edu; Department of Mathematics, Willamette University, 3406 North 26th Street, Tacoma, WA, 98407, USA.}, Alexa Renner\footnote{Corresponding Author: Alexa Renner: renneram@rose-hulman.edu; Department of Mathematics, Rose-Hulman Institute of Technology, 5500 Wabash Ave., Terre Haute, IN 47803, USA.}, Gabriel Roca\footnote{Gabriel Roca: gabriel.roca@ucf.edu; Department of Mathematics, University of Central Florida, 4000 Central Florida Blvd, Orlando, FL 32816, USA.}, Eric Yu\footnote{Eric Yu: ericyu25@sas.upenn.edu; Department of Mathematics, University of Pennsylvania, 209 S 33rd St, Philadelphia, PA 19104, USA}}
\date{October 24, 2025}
\DeclareMathOperator{\Aut}{Aut}
\DeclareMathOperator{\pg}{\Gamma}
\DeclareMathOperator{\pgc}{\overline{\Gamma}}
\usepackage{accents}

\DeclareMathOperator{\GL}{GL}
\DeclareMathOperator{\Sz}{Sz}
\DeclareMathOperator{\PSL}{PSL}

\newtheorem{counter}{counter}[section]

\theoremstyle{definition}
\newtheorem{definition}[counter]{Definition}
\newtheorem*{definition*}{Definition}
\newtheorem{fact}[counter]{Fact}
\newtheorem{notation}[counter]{Notation}

\theoremstyle{plain}
\newtheorem{theorem}[counter]{Theorem}
\newtheorem{lemma}[counter]{Lemma}
\newtheorem{corollary}[counter]{Corollary}
\newtheorem{proposition}[counter]{Proposition}
\newtheorem{conjecture}[counter]{Conjecture}

\theoremstyle{remark}
\newtheorem{remark}[counter]{Remark}

\begin{document}
\maketitle
\begin{abstract}
    For a finite group $G$, the prime graph $\Gamma(G)$ (also known as Gruenberg-Kegel graph) is defined to be the graph where the vertices are the primes that divide $|G|$ such that two vertices $p$ and $q$ share an edge if and only if  there is an element of order $pq$ in $G$. The prime graphs of solvable groups have been classified. The prime graphs of groups whose noncyclic composition factors are isomorphic to a single nonabelian simple group $T$ where $|T|$ is divisible by three or four distinct primes have been classified except for the cases where $T = \operatorname{PSL}(2,q)$ for $q\neq 2^5$ and $|\operatorname{PSL}(2,q)|$ is divisible by exactly four primes.
    In this paper, we provide criteria for general classification results for certain classes of $T$, and then use them to classify the prime graphs of some $T$-solvable groups for $T$ a suitably small $\PSL(2, q)$-group. We also provide general results on the prime graphs of $T$-solvable groups where $T$ is a member of the possibly infinite family of groups $\PSL(2, 2^f)$ such that $f\geq 5, f$ is prime, and $|\PSL(2, 2^f)|$ is divisible by exactly four primes. This is the first paper to prove general results about the prime graphs of $T$-solvable groups where $T$ belongs to a large (probably infinite) family of groups.
\end{abstract}

\hfill\break

\hfill\break
\section*{Introduction}
In this paper, we study the prime graphs of finite groups. The prime graph of a finite group $G$ is the graph with vertex set the prime divisors of $|G|$ in which two such divisors $p, q$ are connected by an edge if and only if $G$ contains an element of order $pq$. This graph is also known as the Gruenberg-Kegel graph of $G$. We denote the prime graph of $G$ by $\Gamma(G)$, and the set of prime divisors of $|G|$ by $\pi(G)$. Prime graphs have been studied extensively from the 1970s to the present, and our paper is a contribution towards the complete classification of prime graphs of a large family of finite groups.

Before we can proceed, we must define a few terms. The first is a natural generalization of the idea of a solvable group:
\begin{definition}
    Let $T$ be a nonabelian simple group, and let $G$ be a group. We say that $G$ is $T$-solvable if all of its composition factors are either cyclic or isomorphic to $T$. We say that $G$ is strictly $T$-solvable if it is $T$-solvable and at least one of its composition factors is isomorphic to $T$.
\end{definition}
\begin{definition}
    A $K_n$ group is a finite simple group $G$ such that $|G|$ is divisible by exactly $n$ primes.
\end{definition}
\begin{definition}
    We refer to the class of $T$-solvable groups where $T$ is a $K_4$-group as {\it single $K_4$-solvable groups}.
\end{definition}
In 2015, the authors of \cite{2015REU} produced the following result:
\begin{theorem} \thlabel{solvclass}
    \cite[Theorem 2.10]{2015REU} An unlabeled graph $\Xi$ is isomorphic to the prime graph complement of a solvable group if and only if it is 3-colorable and triangle-free.
\end{theorem}
Since then, there have been several papers classifying the prime graphs of $T$-solvable groups where $T$ is a nonabelian simple group. In 2022, the authors of \cite{2022REU} completely classified the prime graphs of all $T$-solvable groups such that $T$ is a $K_3$ group. The 2025 paper \cite{2023REU} gives complete classifications of the prime graphs of all $T$-solvable groups such that $T$ is a $K_4$-group and $T\neq\operatorname{Sz}(32), T\neq\operatorname{Sz}(8),$ and $T\neq \PSL(2, q)$ for a prime power $q$. In late 2024, in \cite{Suz} we classified the prime graph complements, and thus the prime graphs, of $T$-solvable groups where $T = \operatorname{Sz}(32), T = \operatorname{Sz}(8),$ and $T = \PSL(2, 2^5)$. The classifications for $T = \operatorname{Sz}(8)$ and $T = \PSL(2, 2^5)$ were the first ones for groups in which $\pi(\operatorname{Aut}(T))\neq \pi(T)$. There are many more $K_4$-groups which also have this property, for example, $\PSL(2, 2^f)$ where $f\geq 5$ is prime.

At this point, to finish the classification of the prime graphs of single $K_4$-solvable groups, all that remains to do is to classify the cases where $T = \PSL(2, q)$ for a prime power $q$ such that $\PSL(2, q)$ is a $K_4-$group, $q\neq 2^5$. This is not a trivial task, as this family may be infinite. As of this writing, whether or not this family is infinite is an open question. Additionally, the modular character tables -- which we used extensively in our classifications of $\Sz(8)-, \Sz(32)-,$ and $\PSL(2, 2^5)$-solvable groups in \cite{Suz} -- are only available through \cite{GAP} for a few relatively small members of the family.

In this paper, we provide general classifications of the prime graph complements of $T$-solvable groups for nonabelian simple groups $T$ that meet certain criteria in \ref{sec:PSL216General}, \ref{sec:General}, \ref{sec:generalT211}, and \ref{sec:PSL225General}. To demonstrate the utility of these results and to gain insight into how the prime graphs of $T$-solvable groups for $T$ in various subfamilies of the family of $K_4$-$\PSL(2, q)$ groups for $q$ a prime power can be classified, we show that some members of the family $\PSL(2, q)$ satisfy these criteria. Specifically, we show this for $q = 2^4$ in Section \ref{sec:PSL216new}; $q = 3^3$ and $q = 7^2$ in Section \ref{sec:PSL249}; $q = 11, q = 19;$ and $q = 23$ in Section \ref{sec:PSLPrimesGeneral}; and $q = 5^2$ and $q = 3^4$ in Section \ref{sec:PSL225}; thus classifying the corresponding prime graphs. In addition, we provide general partial results on the prime graphs of $T$-solvable groups where $T$ is a member of the subfamily $\PSL(2, 2^f)$ for $f\geq 5$ prime such that $T$ is a $K_4$-group in Section \ref{sec:GeneralPSL}. This is the first instance where prime graphs of a general, potentially infinite family of simple groups are
studied (using generic character tables), and we obtain a nearly complete classification, just being unable to decide for a few small graphs whether they can occur.\\

None of the criteria needed to show that a group $T$ satisfies a general classification in \ref{sec:PSL216General}, \ref{sec:General}, \ref{sec:generalT211}, or \ref{sec:PSL225General} depend on modular characters, and we anticipate that these results will be useful even when modular character tables are unavailable. These results should reduce the problem of classifying prime graphs of a certain class of $\PSL(2, q)$ groups to a mostly computational problem of computing Sylow subgroups, Schur multipliers, realizations of a few four-vertex graphs, and character tables. Additionally, our results inform the conjecture in Conjecture \ref{conjectureLast}-- if this conjecture is true, then the classification theorems for a large class of $K_4$ $\PSL(2, q)$ groups would follow from it and our results.\\
Finally we remark that even though $\PSL(2, 13)$ is relatively small, we do not classify the prime graph complements of $\PSL(2, 13)-$solvable groups in this paper as that would have required computing fixed point information of the representation $(\mathbb{F}_3)^{36}\rtimes2.\PSL(2, 13)$, which exceeded the available computational power. (However there is work by Alland, Fridman, and Keller forthcoming addressing $\PSL(2, 13)$ without the need for the fixed point information of
such large groups.)
% we prove results that completely classify the prime graph complements (and thus the prime graphs) of $T$-solvable groups where $T$ satisfies certain criteria. We then use these results to quickly classify the prime graphs of some small $\PSL(2,  q)$-solvable groups where $|\PSL(2, q)|$ is divisible by exactly four primes. We also prove some general partial results on the prime graphs of $\PSL(2, 2^f)$-solvable groups where $f\geq 5, f$ is prime, and $|\PSL(2, q)|$ is divisible by exactly four primes. 

We now introduce some terminology and notation that will be used frequently throughout the rest of the paper. Let $G$ be a group, and let $\pi_0$ be a set of primes.
\begin{itemize}
  %  \item Let $T$ be a finite nonabelian simple group. We say that a finite group $G$ is $T$-solvable if the composition factors of $G$ are all either cyclic or isomorphic to $T$. We say that $G$ is strictly $T$-solvable if $G$ is $T$-solvable and at least one composition factor of $G$ is isomorphic to $T$.
 %   \item We write $\pi(G)$ to denote the set of prime divisors of $|G|$.
    \item If $\pi_0$ is a set of primes, We say that $G$ is a $\pi_0$-group if $\pi(G) \subseteq \pi_0$. We say that $G$ is a strict $\pi_0$-group if $\pi(G) = \pi_0$.
    \item We write $\pg(G)$ to denote the prime graph of $G$, the graph whose vertex set is $\pi(G)$ and where an edge between the primes $p$ and $r$ exists if and only if $G$ has an element of order $pr$. We write $p-r$ to denote the edge $\{p,r\}$. We write $\pgc(G)$ to denote the graph-theoretic complement of $\pg(G)$.
    \item If $\pi_0 \subseteq \pi(G)$, we write $\pgc(G)[\pi_0]$ to denote the subgraph of $\pgc(G)$ induced by the vertices $\pi_0$.
    \item We write normal series in ATLAS notation. This means that we will write $G = X_1.X_2. \cdots .X_k$ if there exists a normal series $G = N_k \trianglerighteq N_{k-1} \trianglerighteq \cdots \trianglerighteq N_1 \trianglerighteq N_0 = \{1\}$ such that $X_i \cong N_i/N_{i-1}$ for $i = 1, \dots, k$.
    \item We sometimes use the notation $= $ to signify a variable being defined as some quantity.
\end{itemize}
% \begin{lemma}
%     \cite[Theorem 2.10]{2015REU} An unlabeled simple graph $\Lambda$ is isomorphic to the prime graph complement of a solvable group if and only if it is 3-colorable and triangle-free.
% \end{lemma}
We now introduce the Frobenius Criterion, which we will use throughout this paper:
\begin{definition}
    \cite[Definition 2.2.1]{2023REU} Let $P$ be a $p$-group. $P$ is said to satisfy the \emph{Frobenius Criterion} if $P$ is cyclic, dihedral, Klein-4, or generalized quaternion. Note that every quotient of a generalized quaternion or cyclic group satisfies the Frobenius Criterion.
\end{definition}
We now present an orientation of $\pgc(G)$ that will be helpful in many of the classification proofs. First, we present relevant terminology quoted from \cite{2022REU}:
\begin{definition}
    A group $G = QP$ is \emph{Frobenius of type (p, q)} if it is a Frobenius group where the Frobenius complement $P$ is a $p$-group and the Frobenius kernel $Q$ is a $q$-group.
\end{definition}
\begin{definition}
    A group $G = P_1QP_2$ is called \emph{2-Frobenius of type (p, q, p)} if it is a 2-Frobenius group where the subgroup $P_1Q$ is Frobenius of type $(q, p)$ and the quotient group $QP_2$ is Frobenius of type $(p, q)$.
\end{definition}
We present the following definition as it appears in \cite[Definition 2.3.3]{2023REU}. It was first defined in \cite{2015REU}.
\begin{definition}\thlabel{frobDigraph}
    \cite[Definition 2.3.3]{2023REU} Let $G$ be a solvable group. The \emph{Frobenius Digraph} of $G$, denoted $\overrightarrow{\Gamma}(G)$, is the orientation of $\pgc(G)$ where $p\rightarrow q$ if the Hall $\{p, q\}$-subgroup of $G$ is Frobenius of type $(p, q)$ or 2-Frobenius of type $(p, q, p)$.
\end{definition}
We include the definition of a rooted graph and a rooted graph isomorphism:
\begin{definition}
    A rooted graph is a graph in which one of the vertices is distinguished to be the root. Two rooted graphs $A$ and $B$ are isomorphic if there is a graph isomorphism which identifies the root of $A$ with the root of $B$.
\end{definition}
Throughout this paper, the definition of a ``graph realizable by a $T$-solvable group" will depend on $T$, so we include the appropriate definitions in the relevant sections.

As in the \cite{Suz} paper, we classify the prime graph complements of $T$-solvable groups in place of the prime graphs. Throughout this paper, all groups are assumed to be finite unless otherwise stated.
% \begin{definition}
%     Let $G$ be a group and $\mathcal{T}$ a set of distinct nonabelian simple groups. Then $G$ is $\mathcal{T}$ -solvable if each of its composition factors are cyclic or in $\mathcal{T}$. Furthermore, $G$ is said to be strictly $\mathcal{T}$-solvable if $G$ is $\mathcal{T}$-solvable and has at least one composition factor from $\mathcal{T}$.
\section*{Preliminaries}
As we continue the work of \cite{2015REU}, \cite{2022REU}, \cite{2023REU}, and \cite{Suz} in this paper. We use several results from these works, so for the reader's convenience, we dedicate this section to stating a few of the results which we use most often.

The following result allows us to take subgroups of a certain form. It is useful in proving that there do not exist edges of the form $r-p$ for $r\in \pi(G)\setminus\pi(T)$ and some $p\in\pi(T)$:
\begin{lemma}\thlabel{lemma212inPaper}
    (\cite[Lemma 2.1.2]{2023REU}): Let $T$ be a nonabelian simple group satisfying $\pi(T ) = \pi(\Aut(T ))$, and let $G$ be a strictly $T$-solvable group. Then $G$ has a subgroup $K \cong N.T$ with $N$ solvable and $\pi(G) = \pi(K)$.
\end{lemma}
The following lemma follows from the proof of \cite[Proposition 2.2.2]{2023REU}:
\begin{proposition}\thlabel{frobprop}
   (\cite[Proposition 2.2.2]{2023REU}): Let $T$ be a nonabelian simple group, and suppose $G\cong N.T$ for some solvable $N$. Fix $r\in\pi(T)$. If the Sylow $r$-subgroups of $T$ do not satisfy the Frobenius criterion, then for all $p\in\pi(N)$ we have $r-p\not\in\pgc(G)$.
\end{proposition}

We state a corollary which allows us to eliminate edges between $p\in\pi(G)\setminus\pi(T)$ and $r\in\pi(T)$ using ordinary characters:
\begin{corollary}\thlabel{corollaryLemma}
    (\cite[Corollary 2.2.6]{2023REU}): Let $r \in \pi(T )$ be an odd prime 
    % such that $(r, |M(T )|) = 1$
    which is coprime to the Schur multiplier of $T$. 
    Suppose that in every complex irreducible representation of a perfect central extension of $T$, some element of order $r$ has fixed points. Then $r - p \notin \pgc(G) \text{ for all } p \in \pi(K) \setminus \pi(T )$.

\end{corollary}
The following lemma is a corrected and slightly generalized version of \cite[Lemma 2.3.6]{2023REU}, although the proof is almost identical to that of the original. We include it for completeness:
\begin{lemma}\thlabel{236Improved}
   Let $T$ be a nonabelian simple group, $F$ a strictly $T$-solvable group, $\Xi$ a graph, and $X$ a set of $|\pi(F)|$ vertices of $\Xi$. Suppose $\Xi\setminus X$ is triangle-free and $\Xi$ has a $3$-coloring $\{\mathcal{O}, \mathcal{D}, \mathcal{I}\}$ such that vertices in $N(X)\setminus X$ are all colored $\mathcal{I}$. Further suppose that:
   \begin{enumerate}
       \item There exists a graph isomorphism $\varphi$ from the subgraph of $\Xi$ induced by $X$ to $\pgc(F)$.
       \item Given $v\in \Xi\setminus X$, there exists a complex irreducible representation $V$ of $F$ such that for all $x\in X$, $x-v\in \Xi$ if and only if elements of order $\varphi(x)$ of $F$ act without fixed points in $V$.
   \end{enumerate}
   Then there exists a group $G$ such that $\Xi\cong \pgc(G)$, and $G$ is $T$-solvable of the form $J\rtimes(F\times K)$ for suitable solvable groups $J$ and $K$.
\end{lemma}
\begin{proof}
    We define a partial orientation of $\Xi$ as follows. In $\Xi\setminus X$, direct edges according to color: $\mathcal{O}\to\mathcal{D}$, $\mathcal{O}\to\mathcal{I}$, and $\mathcal{D}\to \mathcal{I}$. For all edges $u-v\in\Xi$ where $u\in X$ and $v\in\Xi\setminus X$, define the orientation as $u\to v$. Now let $n_o = |\mathcal{O}|, n_i = |\mathcal{I}|, n_d = |\mathcal{D}|$.

    Choose $n_0$ distinct primes $p_1, \cdots, p_{n_o}\not\in \pi(F)$. Define $p = \prod_{j=0}^{n_o}p_j$. Using Dirichlet's theorem on arithmetic progressions, pick a set of distinct primes $q_1, \dots, q_{n_d}$ other than those in $\pi(F)$ such that $q_i\equiv 1 (\operatorname{mod } p_i)$ for all $i$. Identify each vertex in $\mathcal{O}$ with one of the $p_i$ and identify each vertex in $\mathcal{D}$ with one of the $q_i$. Define groups
    \begin{align*}
        P = C_{p_1}\times\cdots\times C_{p_{n_o}}\text{ and } Q = C_{q_1}\times\cdots\times C_{q_{n_d}}.
    \end{align*}
    For all indices $i, j$ if $p_j\to q_i$ is an edge in $\overrightarrow{\Xi}$, then let $C_{p_j}$ act Frobeniusly on $C_{q_i}$. This is possible because $q_i\equiv 1(\text{mod } p_i)$. Otherwise, if $p_j$ and $q_i$ are not adjacent in $\Xi$, let $C_{p_j}$ act trivially on $C_{q_j}$. This defines a group action of $P$ on $Q$, so we obtain the induced semidirect product $K = Q\rtimes P$. Note that $K$ is solvable.

    Now let $v_1,\cdots, v_{n_i}$ be the vertices in $\mathcal{I}$. For each $k\in\{1, \cdots, n_i\}$, let $N^1(v_k), N^2(v_k)$ denote the set of primes in $\Xi\setminus X$ with in-distance 1 and 2 to $v_k$, respectively. If $N^1(v_k)$ is nonempty, let $B_k$ be the Hall $(N^1(v_k)\cup N^2(v_k))$-subgroup of $K$. By our definition of $K$, $\operatorname{Fit}(B_k)$ is a Hall $N^1(v_k)$-subgroup of $K$. If $N^1(v_k)$ is empty, set $B_k = 1$. Now we divide into two cases for each $k$.

    If $v_k$ is not adjacent to any vertex in $X$, consider the trivial complex representation of $F$ and pick a prime $r_k$ such that $|F\times B_k|\mid (r_k - 1)$. By \cite[Lemma 3.5]{236} there exists a modular representation $R_k$ of $F\times B_k$ over a finite field of characteristic $r_k$ such that $\operatorname{Fit}(B_k)$ acts Frobeniusly on $R_k$ and $F$ acts trivially on $R_k$.

    If $v_k$ is adjacent to some vertex in $X$, let $V$ be the associated irreducible representation granted by the hypothesis. Applying Dirichlet's theorem on arithmetic progressions, take a prime $r_k$ such that $|F\times B_k|\mid (r_k - 1)$. According to \cite[Lemma 3.5]{236}, there exists a representation $R_k$ of $F\times B_k$ over a finite field of characteristic $r_k$ such that $\operatorname{Fit}(B_k)$ acts Frobeniusly on $R_k$, and elements of $F$ act without fixed points on $R_k$ if and only if they act without fixed points in $V$. In other words, for all $x\in X$, we have $x-p\in\pgc(G)$ if and only if order $\varphi(x)$ elements of $F$ act without fixed points on $R_k$.

    Let $J = R_1\times\cdots\times R_{n_i}$. Note that by Dirichlet's theorem on arithmetic progressions, we can require each of the $r_k$ to be distinct from the $p_j$ and $q_i$ and additionally be mutually distinct. Thus, we have defined an action of each $F\times B_k$ on $R_k$. This induces an action of $F\times K$ on $J$, which induces the $T$-solvable semidirect product $G = J\rtimes(F\times K)$. By this construction, $\pgc(G)\cong \Xi$.
\end{proof}

Finally, we re-state the result which relates Brauer characters to prime graph complements:
\begin{theorem}\thlabel{GeneralizedModuleExtensionsCited}
        (\cite[Theorem 3.7]{Suz}): Let $T$ be a finite simple group, and let $p \in \pi(T)$. For each $\chi \in \operatorname{IBr}_p(T)$, let $B_\chi$ be the set of edges
        $\{p-q \mid  \exists \ g\in T \textnormal{ s.t. } \operatorname{o}(g) = q \textnormal{ and } \frac{1}{\operatorname{o}(g)}\sum_{x \in \langle g \rangle}\chi(x) > 0\}$. Then given some graph $\Lambda$, we have that $\Lambda$ is realizable as the prime graph complement of a group of the form $N.T$ where $N$ is a $p$-group if and only if there is some subset $Y\subseteq \operatorname{IBr}_p(T)$ such that $\Lambda = \pgc(T)\setminus\mleft(\bigcup_{\chi\in Y} B_\chi\mright)$.
\end{theorem} 

\section{Lemmas for the Classifications of Prime Graph Complements of $\PSL(2, q)$-Solvable Groups}\label{sec:UsefulLemmas}
This section contains results which we use often in the classifications prime graph complements of $\PSL(2, q)$-solvable groups. The lemmas are fairly general and can be applied in many situations where $T$ is a finite nonabelian simple group. The first lemma generalizes \cite[Lemma 3.1.3]{2023REU}.

% \begin{lemma} \thlabel{isolate2psl211}\thlabel{noworry2}
%     Let $T$ be a nonabelian simple group and let $G$ be a strictly $T$-solvable group. Suppose $T$ satisfies all of the following:
%     \begin{itemize}
%         \item $\pi(T) = \pi(\Aut(T))$.
%         \item The Sylow 2-subgroups of $T$ satisfy the Frobenius criterion but are not cyclic or generalized quaternion.
%         \item There exists a perfect central extension $2.T$.
%     \end{itemize}
%     Then if $2-p \in \pgc(G)$ for some $p \in \pi(G) \setminus \pi(T)$, both of the following hold: 
%     \begin{enumerate}
%         \item There exists some subgroup $K\leq G$ isomorphic to $L.2.T$ with $\pi(K) = \pi(G)$ where $L$ is a solvable group of odd order and $2.T$ is a perfect central extension.
%         \item $2-q \notin \pgc(G)$ for all $q \in \pi(T)$.
%     \end{enumerate}
% \end{lemma}
\begin{lemma} \thlabel{isolate2psl211}\thlabel{noworry2}
    Let $T$ be a nonabelian simple group and let $G$ be a strictly $T$-solvable group. Suppose $T$ satisfies all of the following:
    \begin{itemize}
        \item $\pi(T) = \pi(\Aut(T))$.
        \item There exists a perfect central extension $2.T$.
    \end{itemize}
    Then if $2-p \in \pgc(G)$ for some $p \in \pi(G) \setminus \pi(T)$, both of the following hold: 
    \begin{enumerate}
        \item There exists some subgroup $K\leq G$ isomorphic to $L.2.T$ with $\pi(K) = \pi(G)$ where $L$ is a solvable group of odd order and $2.T$ is a perfect central extension.
        \item $2-q \notin \pgc(G)$ for all $q \in \pi(T)$.
    \end{enumerate}
\end{lemma}
\begin{proof}
    This proof follows similarly to \cite[Lemma 3.1.3]{2023REU}. By \thref{lemma212inPaper}, we may assume $G$ is of the form $N.T$ where $N$ is solvable. Because $2-p\in\pgc(G)$ for $p\in\pi(G)\setminus\pi(T)$, the Sylow 2-subgroups of $T$ must satisfy the Frobenius criterion by \thref{lemma212inPaper} and \thref{frobprop}. By the odd-order theorem, $T$ must have even order. The Sylow 2-subgroups of $T$ cannot be cyclic by Cayley's normal 2-complement theorem; and by the Brauer-Suzuki theorem, the Sylow 2-subgroups of $T$ cannot be generalized quaternion groups. 
    
     Let $Q$ be a Sylow 2-subgroup of $G$, let $H$ be a Hall $\{2,p\}$-subgroup of $NQ$ containing $Q$, and let $P$ be a Sylow $p$-subgroup of $H$. By \cite{Williams}, we have that $H$ is Frobenius of type $(2,p)$, or 2-Frobenius of type $(2,p, 2)$. $H$ has a normal series 
    \begin{equation}
        H \trianglerighteq H \cap N \trianglerighteq \{1\}.
    \end{equation}
    We observe that $H \cap N$ is a Hall $\{2,p\}$-subgroup of $N$, and $H/(H\cap N)$ is isomorphic to a Sylow 2-subgroup of $T$ and is thus not cyclic. The topmost Frobenius complement of $H$ is an extension of $H/(H\cap N)$, meaning it must be a non-cyclic 2-group. By \cite[Lemma 2.1]{2015REU}, the topmost Frobenius complement of a 2-Frobenius group is cyclic, so $H$ cannot be 2-Frobenius. Thus, $H \cong P.Q$, with $Q$ acting Frobeniusly on $P$. Since $Q$ is a non-cyclic Frobenius complement of prime power order, it must be a generalized quaternion group \cite[Corollary 6.17]{IsaacsFiniteGroups}. Let $Q_0 = Q \cap N$. We must have that $Q_0$ is nontrivial, or else $Q\cong Q/Q_0 \in \operatorname{Syl}_2(T)$ would be generalized quaternion, a contradiction. Additionally, we must have that $[Q:Q_0] \geq 4$, or else $Q/Q_0 \in \operatorname{Syl}_2(T)$ would be cyclic, a contradiction. Since all normal subgroups of generalized quaternion groups of index at least 4 are cyclic (see e.g. \cite[Corollary 4.5]{conrad2}, so $Q_0$ must be cyclic.
%Corollary 4.5 CONRAD

    By Cayley's normal 2-complement theorem we may write $N = L.Q_0$ where $L$ is a normal Hall $2'$-subgroup of $N$. However, $L$ is characteristic in $N$ which implies it is normal in $G$. Recall that $N/L\cong Q_0$ is abelian, thus $(G/L)/(N/L)\cong T$ acts on $Q_0$ by conjugation. But $\Aut(Q_0)$ has order a power of 2, meaning the action has nontrivial kernel. Since $T$ is simple, we conclude that it must act trivially on $Q_0$. Thus, all of $G/L\cong Q_0.T$ acts trivially on $Q_0$ by conjugation, meaning $Q_0$ is contained in the center of $G/L$. In particular, it is contained in the center of $Q$. We know that $Q$ is generalized quaternion, so the center has order 2. Since $Q_0$ must be nontrivial it follows that $Q_0 \cong C_2$. It is well known that generalized quaternion groups do not split over $C_2$, therefore $Q_0.T$ does not split over $Q_0$. Any nonsplit extension of a simple group by $C_2$ is a perfect central extension, meaning $Q_0.T \cong 2.T$. Thus we see that $G \cong L.2.T$, so (1) is satisfied.

%Note 2\leq \Phi(G)-- if there existed a maximal Y subgroup of G not containing 2, Y is a subgroup of 2Y
    
    For (2), notice that an element of order 2 is in the center of $Q_0.T$, meaning all elements in $T$ commute with an element of order 2, thus there must be elements of order $2q$ for every $q \in \pi(T)$.
\end{proof}
The following lemma generalizes \thref{frobprop}, and its proof is very similar to that of \thref{frobprop}:
\begin{lemma}\thlabel{generalize222}
    Let $T$ be a nonabelian simple group and $G\cong N.S$ for $N$ solvable and $S$ a group with a section isomorphic to $T$ such that $\pi(S) = \pi(T)$. Suppose the Sylow $r$-groups of $S$ do not satisfy the Frobenius Criterion for some $r\in\pi(T)$. Then, for all $p\in\pi(N), p-r\not\in\pgc(G)$.
\end{lemma}
\begin{proof}
    By \cite[Lemma 5.8]{Suz}, $G$ has a normal series $1 = N_0\unlhd N_1\unlhd\cdots\unlhd N_k\unlhd G$ such that $N_k = N$ and $\frac{N_i}{N_{i-1}}$ is elementary abelian; we may assume that $\frac{N_i}{N_{i-1}}$ is a $p_i$-group where $p_i$ is a orine for all $i$. Let $j$ be the least $i$ such that $p\nmid |N/N_i|$. Let $V = \frac{N_j}{N_{j-1}}$, $W = \frac{N}{N_j}$, and notice that $V$ is a nontrivial elementary abelian $p$-group and $W$ is a $p'$-group. Consider the section $L\cong V.W.S$. Then $L/V\cong W.S$ acts on $V$ by conjugation as $V$ is abelian. 
    
    Let $M$ be a Sylow $r$-group of $L/V$ and $R$ be a Sylow $r$-group of $S$, and notice $M/(M\cap W)\cong R$. By hypothesis $R$ does not satisfy the Frobenius criterion, so $M$ is not cyclic or generalized quaternion. Thus, $M$ cannot act Frobeniusly on $V$ by \cite[Corollary 6.17]{IsaacsFiniteGroups}. Hence $pr\not\in\pgc(G)$, as desired.
    
    %Thus, there exist nontrivial elements $s\in S$ with order a power of $r$ and $v\in V$ with order a power of $p$ such that $s^{-1}xs = x$. Then, $sV$ fixes $x$ and $o(sV) = o(s)$. Adjusting the exponents of $sV$ and $x$ as needed, $p-r\not\in\pgc(L)$ by \cite[Corollary 1.5]{Suz}. $\pgc(G)[\pi(L)]\subseteq\pgc(L)$, so $p-r\not\in\pgc(G)$.
\end{proof}
The next result builds on \thref{236Improved}.

\begin{lemma}\thlabel{NewPGCFromOld}
    Let $T$ be a nonabelian simple group, $F$ a strictly $T$-solvable group, $\Xi$ a graph, and $X$ a set of $|\pi(F)|$ vertices of $\Xi$. Suppose $\Xi\setminus X$ is triangle-free and has a 3-coloring $\{\mathcal{O}, \mathcal{D}, \mathcal{I}\}$ such that vertices in $N(X)\setminus X$ are all colored $\mathcal{I}$. Further suppose that there exists some $E = F\ltimes L$ for some solvable $L$ such that $\pi(L) \subseteq \pi(T)$ and:
    \begin{itemize}
        \item There exists a graph isomorphism $\phi$ from $\Xi[X]$ to $\pgc(E)$
        \item Given $v\in\Xi\setminus X$ there exists a complex irreducible representation $V$ of $F$ such that for all $x\in X$, $x-v\in \Xi$ if and only if elements of order $\varphi(x)$ of $F$ act without fixed points in $V$
        \item $N[\pi(L)]\subseteq\pi(T)$ in all prime graph complements of strictly $T$-solvable groups.
    \end{itemize}
    Then there exists a solvable group $N$ such that $G = F\ltimes(N\times L)$ and $\pgc(G)\cong\Xi$.
\end{lemma}
\begin{proof}
    Let $\Lambda$ be the graph obtained from $\Xi$ by possibly adding edges among vertices in $X$ such that $\Lambda[X]\cong\pgc(F)$. In particular, 
    $E(\Xi)\subseteq E(\Lambda)$. By our assumptions and \thref{236Improved}, there exists an $F$-solvable group $H$ such that $\pgc(H)\cong\Lambda$ and $\pgc(H)[\pi(F)]\cong \pgc(F)$. Then, $H = (K\times F)\ltimes J$ for solvable groups $K$ and $J$ by the construction in \thref{236Improved}. $|K|$ is coprime to $|F|$ by construction. $J$ is defined as the direct product of $R_k$s, where each $R_k$ is a module corresponding to a representation of $F\times B_k$, for $B_k$ defined in the proof of \thref{236Improved}, over a field of characteristic $r_k$ where $r_k$ is a prime such that $|F\times B_k|$ divides $r_k - 1$. Thus, $r_k$ and $|F|$ are coprime, and we can also see that each $R_k$ has order a power of $r_k$ (each $R_k$ is a module over a field, and thus a vector space). Then, $\pi(J)\cap\pi(F) = \varnothing$. As such, $(\pi(J)\cup\pi(K))\cap\pi(F) = \varnothing$. We have $H/J\cong F\times K$, so there exists a $N\unlhd H$ such that $N/J\cong K$. Thus, $H/N\cong F$. As such, we have  $|N| = |K||J|$, which is coprime to $|F|$. By the Schur-Zassenhaus theorem, this means that $H = F\ltimes N$, and $N$ must be solvable because $H$ is $T$-solvable and $|N|$ is coprime to $|T|$. Consider $G = F\ltimes(N\times L)$, where $F$ acts on $N$ as in $H$ and $F$ acts on $L$ as in $E$. For all $r\in \pi(G)\setminus\pi(T)$ and $q\in\pi(L)$, $q-r\not\in\pgc(G)$ and by our assumptions, $\pgc(G)[\pi(F)]$ is $\pgc(E)$. Notice $\pgc(H)[\pi(N)] = \pgc(G)[\pi(N)]$ and that $p-r\in \pgc(G)$ for $p\in\pi(F), r\in\pi(G)\setminus\pi(F)$ if and only if $p-r\in\pgc(H)$ for $p\in\pi(F), r\in\pi(H)\setminus\pi(F)$. This follows from the fact that $N[\pi(L)]\subseteq\pi(T)$ in all prime graph complements of strictly $T$-solvable groups. Thus, labeling the vertices of $\Xi$ by labeling $x\in X$ with the $p\in \pi(E)$ that corresponds to it under the isomorphism between $\Xi[X]$ and $\pgc(E)$ and labeling $y\in\Xi\setminus X$ with the $p\in\pi(G)\setminus\pi(E) = \pi(H)\setminus\pi(F)$ which corresponds to $y$ in the isomorphism $\Lambda\cong \pgc(H)$, we see $\Xi\cong \pgc(G)$.
\end{proof}

\begin{lemma}\thlabel{ConjugacyClassLanding}
    Let $p$ be a prime. Let $G$ be a group such that its Sylow $p$-subgroups are cyclic of order $p$. Let $U$ be such a Sylow subgroup. Suppose that there are $n$ conjugacy classes in $G$ of elements of order $p$. Then there are $\frac{p - 1}{n}$ elements in each conjugacy class of $U$ of elements of order $p$. 
\end{lemma}
\begin{proof}
    Note that $x, y\in U$ are conjugate in $G$ if and only if they are conjugate in $N_G(U)$. Also note that $N_G(U)/C_G(U)$ acts Frobeniusly on $U$. Since all Sylow $p$-subgroups of $G$ are conjugate, it follows that $n$ equals the number of orbits of $N_G(U)/C_G(U)$ on the nontrivial elements of $U$, and each such orbit has $|N_G(U)/C_G(U)|$ elements. Thus $n=(p-1)/|N_G(U)/C_G(U)|$ and the proof is complete.
    %As such, it suffices to study the action of $N_G(U)/C_G(U)$ on $U$, which we may consider since $N_G(U)/C_G(U)\cong S\leq \Aut(U)\cong C_{p-1}$, so $N_G(U)/C_G(U)$ acts via automorphisms on $U$. The action of $C_{p-1}$ on $U$ is Frobenius because each element of $C_{p-1}$ permutes the $p - 1$ non-identity elements of $U$, so the action of $N_G(U)/C_G(U)$ on $U$ must be a Frobenius action. In order to find the number of elements of $U$ in each conjugacy class of elements of order $p$, we must find the size of the orbit which contains $x$ of $N_G(U)/C_G(U)$ acting on $U$. Because the action of $N_G(U)/C_G(U)$ is a Frobenius action, the size of any orbit of a non-identity element must be $|N_G(U)/C_G(U)|$. Then, the non-identity elements of $U$ must be evenly distributed between the different conjugacy classes of elements of order $p$. Thus, given $n$ conjugacy classes of elements of order $p$, for any non-identity element $x$ of $U$, there will be $\frac{p - 1}{n}$ elements of $U$ in each conjugacy class of elements of order $p$.
\end{proof}

\section{$T$-solvable groups with $T$ similar to $\PSL(2, 2^4)$}\label{sec:PSL216}
In Section \ref{sec:PSL216General} we will prove a general result which classifies the prime graph complements of $T$-solvable groups where $T$ is a $K_4$ group that satisfies certain criteria. In Section \ref{sec:PSL216new} we will then show that $T = \PSL(2, 2^4)$ satisfies these criteria and apply the result from Section \ref{sec:PSL216General} to classify the prime graph complements of $\PSL(2, 2^4)$-solvable groups.

\subsection{General $T$}\label{sec:PSL216General}
Here we are interested in $T$-solvable groups for non-abelian simple groups $T$ whose prime graph complement looks as follows. 
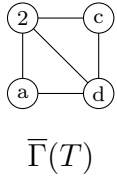
\begin{figure}[H]
    \centering
    \begin{minipage}{\textwidth}
        \centering
        \begin{tikzpicture}
            % Define coordinates for the vertices
            \coordinate (A) at (0, 1);   % Vertex 2
            \coordinate (B) at (1, 1);   % Vertex 7
            \coordinate (C) at (1, 0);   % Vertex 13
            \coordinate (D) at (0, 0);   % Vertex 5
            % Draw the edges of the complete graph
            \draw (B) -- (C);
            \draw (C) -- (D);
            \draw (A) -- (D);
            \draw (A) -- (B);
            \draw (A) -- (C);
            % Draw the vertices (circles only)
            \foreach \point in {A, B, C, D}
                \draw[draw, fill=white] (\point) circle [radius=0.2cm];
            % Draw the text on top of the circles
            \foreach \point/\name in {(A)/2, (B)/c, (C)/d, (D)/a}
                \node at \point {\scriptsize\name};
        \end{tikzpicture}
        \caption*{$\pgc(T)$}
    \end{minipage}%
\end{figure}
We first introduce some notation which we just need in this section.
\begin{notation} \thlabel{rootedGraphIsomorphismPSL216}
    Given a group $G$ with $\pi(G) = \pi(\Aut(T))$ and a rooted graph $\Lambda$ on four vertices, it is understood that when we say $\pgc(G) \cong \Lambda$, we require the isomorphism to map the vertex $d$ to the root.
\end{notation}
\begin{notation}\thlabel{realizableDefPSL216}
    Given a set of rooted graphs $\mathcal{H}$ on four vertices, we say that $\mathcal{H}$ is \emph{realizable} if for each $\Lambda \in \mathcal{H}$, there exists a $T$-solvable $\pi(T)$-group $G$ such that $\pgc(G) \cong \Lambda$ in the sense of \thref{rootedGraphIsomorphismPSL216}.
\end{notation}
Our goal is to classify the prime graph complements of $T$-solvable groups where $T$ is a nonabelian simple group such that $\pgc(T)$ is the above graph, $\pi(T) = \pi(\Aut(T))$, and each of the following criteria is satisfied:
\begin{itemize}
    \item[1.] The Schur multiplier of $T$ is 1
    \item[2.] The fixed point information for $T$ matches \thref{psl16fpinfo} under some bijection between $\pi(T)$ and $\{2, a, c, d\}$ such that 2 is mapped to itself.
    \item[3.] The (rooted) graph $\szunrecognizei$ (where the white vertex indicates the root) is realizable (in the sense of Definition \ref{rootedGraphIsomorphismPSL216}) by $T\ltimes P$ where $P$ is a $p$-group for some $p\in\{2, a, c\}$.
    \item[4.] The graph $\SuzukiThirtyTwoElusiveGraphWhiteBadb$ is not realizable by a $T$-solvable group. 
    \item[5.] The Sylow $2$-subgroups of $T$ do not satisfy the Frobenius criterion.
\end{itemize}

We will refer to these criteria as Criteria 1-5 in the remainder of Section \ref{sec:PSL216}.\\
Also, for the remainder of Section \ref{sec:PSL216General}, let $T$ be a nonabelian simple group with the above five properties.

\begin{fact}\thlabel{psl16fpinfo}
    \begin{tabular}[h]{|c|c|}
    \hline
        Fixed Point Information &  \\
        \hline
        $T$ & $[[2, a, c], [2,a,c,d]]$ \\
        \hline
    \end{tabular}\\
    Each list in the list corresponds to some irreducible, complex  character of $T$ and contains all primes $p$ for which there exist elements of order $p$ in $T$ which have fixed points in the representation associated with that character
    $T$such that elements $T$. There are different lists within the list because there are irreducible characters with different fixed point information. The data was obtained with GAP as in, for instance, \cite{2023REU}. 
\end{fact}

\begin{theorem}\thlabel{StructureofPSL16SolvableGraphs}
    Let $G$ be a strictly $T$-solvable group. Then $\pgc(G)$ satisfies both of the following:
    \begin{enumerate}
    \item There are no edges $r-p$ for $r\in \pi(T)\setminus\{d\}$, $p\in \pi(G) \setminus (T)$, and $\pgc(G)$ has a three-coloring such that $N(\pi(T))\setminus\pi(T)$ shares one color.
        \item All triangles of $\pgc(G)$ are contained in $\pgc(G)[\pi(T)]$.
    \end{enumerate}
\end{theorem}
\begin{proof}
    To prove (1), first recall that by Criterion 1, the Schur multiplier of $T$ is 1. Then, we may apply \thref{corollaryLemma} to \thref{psl16fpinfo} and (1) follows. Thus, there are no edges of the form $r-p$ for $r\in \pi(T)\setminus\{d, 2\}$ and $p\in \pi(G) \setminus (\pi(T))$. Note that the Sylow $2$-subgroups of $T$ do not satisfy the Frobenius criterion by Criterion 5. Then by \cite[Proposition 2.2.2]{2023REU} for any $p\in\pi(G)\setminus\pi(T)$, $2-p\notin\pgc(G)$. Because $\pgc(G)[\pi(T)]\subseteq\hamSandwich$, $\pgc(G)[\pi(T)]$ is three-colorable, so we may apply \cite[Lemma 2.3.5]{2023REU}. By \cite[Lemma 2.3.5]{2023REU}, $\pgc(G)$ has a three-coloring for which all vertices adjacent to $d$ not in $\pi(T)$ have the same color. Therefore, all vertices in $N(\pi(T))\setminus\pi(T)$ share one color under this coloring, so (1) is satisfied. (2) follows from \cite[Lemma 2.3.5]{2023REU}.
\end{proof}
\begin{corollary}\thlabel{trianglefree3colors}
    Let $G$ be a $T$-solvable group. If $\pgc(G)[\pi(T)]$ is triangle-free, $\pgc(G)$ is triangle-free and three-colorable.
\end{corollary}
\begin{proof}
    We consider two cases: The case where $G$ is not strictly $T$-solvable and the case where $G$ is strictly $T$-solvable. If $G$ is not strictly $T$-solvable, it is solvable, and the result follows from \cite{2015REU}. If $G$ is strictly $T$-solvable, note that $\pgc(G)$ is three-colorable and that $\pgc(G)\setminus\pi(T)$ is triangle-free by \thref{StructureofPSL16SolvableGraphs}. By \thref{StructureofPSL16SolvableGraphs} (2), all triangles in $\pgc(G)$ are contained in $\pgc(G)[\pi(T)]$, so $\pgc(G)$ is triangle-free and three-colorable. 
\end{proof}

\begin{fact}\thlabel{factPSL16}
    The graphs listed below can be realized via $T$-solvable groups where $d$ is the white vertex.
    \begin{itemize}
        \item Any triangle free and 3-colorable graph can be realized via some solvable group by methods in \cite{2015REU}.
        \item The rooted graph \hamSandwichWhite can be realized by $T$.
        \item  \notIsolatedVertexWhite can be realized via $C_a\times T$.
        \item $\szunrecognizei$ can be realized by $T\ltimes P$ for some $p\in\{2, a, c\}$ by Criterion 3.
    \end{itemize}
\end{fact}

\begin{theorem}\thlabel{generalClassificationPSL216}
    A graph $\Xi$ is isomorphic to the prime graph complement of some $T$-solvable group if and only if one of the following is true: 
\begin{enumerate}
    \item $\Xi$ is triangle-free and 3-colorable.
    \item $\Xi$ contains a subset $X = \{w, x, y, z\} \subseteq V(\Xi)$ such that for all $p-q\in \Xi$ with $p\in X$ and $q\in\Xi\setminus X$, we have that $p = z$. Moreover, in some 3-coloring of $\Xi$ the closed neighborhood $N(X) \setminus X$ shares one color; and $\Xi[X] =\hamSandwichWhite$, \szunrecognizei, or \notIsolatedVertexWhite where $z$ is the white vertex. 
    %Furthermore, all triangles of $\Xi$ are contained in $\Xi[X]$.
\end{enumerate}
\end{theorem}
\begin{proof}
Let $G$ be a $T$-solvable group. If $G$ is solvable, $\pgc(G)$ satisfies (1) by \cite{2015REU}. Now consider the case where $G$ is strictly $T$-solvable. By \thref{StructureofPSL16SolvableGraphs}, $\pgc(G)$ is 3-colorable. There are two cases: Either $\pgc(G)[\pi(T)]$ is triangle-free or $\pgc(G)[\pi(T)]$ is not triangle-free. First consider the case where $\pgc(G)[\pi(T)]$ is triangle-free. By \thref{trianglefree3colors}, (1) is satisfied. If $\pgc(G)[\pi(T)]$ is not triangle-free, let $X = \pi(T)$ and let $z = d$. By the fact that $\pgc(G)[\pi(T)]\subseteq\pgc(T)$ and by Criterion 4, $\pgc(G)[\pi(T)]$ is not isomorphic to \psltwosixteenimpossiblei, \psltwosixteenimpossibleii, \SuzukiThirtyTwoElusiveGraphWhiteBadb, or \psltwosixteenimpossibleiii. Thus, $\pgc(G)[\pi(T)]$ realizes a graph listed in (2). By \thref{StructureofPSL16SolvableGraphs}, all edges between $p\in X$ and $q\in\pi(G)\setminus\pi(T)$ have $p = z$, so $N(z) = N(X)$. Also by \thref{StructureofPSL16SolvableGraphs}, we have that there is a three-coloring of $\pgc(G)$ such that $N(X)\setminus X$ shares one color. Thus, (2) is satisfied.

We now turn to the backwards direction. Suppose we have some graph $\Xi$ such that $\Xi$ satisfies (1) or (2). If $\Xi$ satisfies (1), there exists some solvable group $G$ such that $\pgc(G)\cong \Xi$ by \cite{2015REU}. If $\Xi$ satisfies (2), we split into cases. If $\Xi[X]\cong\hamSandwichWhite$, there exists a $T$-solvable group $G$ such that $\pgc(G)\cong\Xi$ by Criterion 2 and \thref{236Improved}. If $\Xi[X]\cong\notIsolatedVertexWhite$, notice there is a graph isomorphism from $C_a\times T$ to $\Xi[X]$, given by assigning $2\to w, a\to x, c\to y$, and $d\to z$. By \thref{NewPGCFromOld}, \thref{StructureofPSL16SolvableGraphs}, and Criterion 2, there exists a $T$-solvable group $G$ such that $\pgc(G)\cong \Xi$.
Now suppose $\Xi[X]\cong \szunrecognizei$. Then by Criterion 3, $\pgc(T\ltimes P)\cong \szunrecognizei$ where $d$ is the white vertex and $P$ a suitable $p$-group with $p\in\{2, a, c\}$ . Let $E = T\ltimes P$. By Criterion 2 and \thref{StructureofPSL16SolvableGraphs}, we may apply \thref{NewPGCFromOld} to get a group $G$ such that $\pgc(G)\cong \Xi$.

% We have, for any $T$-solvable group $H$, for all $p\in\pi(H)\setminus\pi(\PSL(2, 2^4))$ the edge $3-p\not\in\pgc(H)$. 
% let $\Xi'$ be a graph such that $V(\Xi') = V(\Xi)$ and the edge set is the same except that there are exactly 2 edges from the isolated vertex to the vertex $z$ and some other vertex. By criteria 2 and \thref{236Improved}, there exists a $T$-solvable group $H$ such that $\pgc(H)\cong\Xi'$ and $\pgc(H)[\pi(\PSL(2, 2^4))]\cong \pgc(\PSL(2, 2^4))$. Then let $G = C_3\times H$. 

% Because there are no edges of the form $3-p$ for $p\in\pi(H)\setminus\pi(\PSL(2, 2^4))$ by \thref{StructureofPSL16SolvableGraphs}, $\pgc(G)\cong\Xi$.

\end{proof}
\subsection{$\PSL(2, 2^4)$}\label{sec:PSL216new}
\begin{figure}[H]
    \centering
    \begin{minipage}{.3\textwidth}
        \centering
        \begin{tikzpicture}
            % Define coordinates for the vertices
            \coordinate (A) at (0, 1);   % Vertex 2
            \coordinate (B) at (1, 1);   % Vertex 7
            \coordinate (C) at (1, 0);   % Vertex 13
            \coordinate (D) at (0, 0);   % Vertex 5
            % Draw the edges of the complete graph
            \draw (B) -- (C);
            \draw (C) -- (D);
            \draw (A) -- (D);
            \draw (A) -- (B);
            \draw (A) -- (C);
            % Draw the vertices (circles only)
            \foreach \point in {A, B, C, D}
                \draw[draw, fill=white] (\point) circle [radius=0.2cm];
            % Draw the text on top of the circles
            \foreach \point/\name in {(A)/2, (B)/5, (C)/17, (D)/3}
                \node at \point {\scriptsize\name};
        \end{tikzpicture}
        \caption*{$\pgc\left(\PSL(2,2^4)\right)$}
    \end{minipage}%
    \begin{minipage}{.3\textwidth}
        \centering
        \begin{tikzpicture}
            % Define coordinates for the vertices
            \coordinate (A) at (0, 1);   % Vertex 2
            \coordinate (B) at (1, 1);   % Vertex 7
            \coordinate (C) at (1, 0);   % Vertex 13
            \coordinate (D) at (0, 0);   % Vertex 5
            % Draw the edges of the complete graph
            \draw (B) -- (C);
            \draw (C) -- (D);
            \draw (A) -- (C);
            % Draw the vertices (circles only)
            \foreach \point in {A, B, C, D}
                \draw[draw, fill=white] (\point) circle [radius=0.2cm];
            % Draw the text on top of the circles
            \foreach \point/\name in {(A)/2, (B)/5, (C)/17, (D)/3}
                \node at \point {\scriptsize\name};
        \end{tikzpicture}
        \caption*{$\pgc\left(\PSL(2,2^4):C_2\right)$}
    \end{minipage}%
   \begin{minipage}{.3\textwidth}
        \centering
        \begin{tikzpicture}
            % Define coordinates for the vertices
            \coordinate (A) at (0, 1);   % Vertex 2
            \coordinate (B) at (1, 1);   % Vertex 7
            \coordinate (C) at (1, 0);   % Vertex 13
            \coordinate (D) at (0, 0);   % Vertex 5
            % Draw the edges of the complete graph
            \draw (B) -- (C);
            \draw (C) -- (D);
            \draw (A) -- (C);
            % Draw the vertices (circles only)
            \foreach \point in {A, B, C, D}
                \draw[draw, fill=white] (\point) circle [radius=0.2cm];
            % Draw the text on top of the circles
            \foreach \point/\name in {(A)/2, (B)/5, (C)/17, (D)/3}
                \node at \point {\scriptsize\name};
        \end{tikzpicture}
        \caption*{$\pgc\left(\Aut\left(\PSL(2,2^4)\right)\right)$}
    \end{minipage}%
\end{figure}
In this subsection, we show that the group $\PSL(2, 2^4)$ meets the criteria given in \ref{sec:PSL216General}. The group $\PSL(2, 2^4)$ has order $4080 = 2^4 \cdot3\cdot5\cdot17$ and its automorphism group has order $16320 = 2^6\cdot3\cdot5\cdot17$. The above figures were computed with \cite{GAP} and \thref{psl16fact} was found via \cite{GAP}. \thref{subgroupsOfAutPSL216} was found via \cite{ATLAS}

\begin{fact}\thlabel{subgroupsOfAutPSL216}
     $\PSL(2, 2^4), \PSL(2, 2^4):C_2, $ and $\Aut\left(\PSL(2, 2^4)\right)$ are the only subgroups of $\Aut(\PSL(2, 2^4))$ containing $\PSL(2, 2^4)$.
\end{fact}
\begin{fact}\thlabel{psl16fact}
    The Schur multiplier of $\PSL(2,2^4)$ is 1. The Sylow $2$-subgroups of $\PSL(2, 2^4)$ are isomorphic to $(C_2)^4$ and do not satisfy the Frobenius Criterion (\cite[Definition 2.2.1]{2023REU}). Also, $\pi\left(\PSL(2,2^4)\right)=\pi\left(\Aut\left(\PSL(2,2^4)\right)\right)$.
\end{fact}
\begin{fact}\thlabel{psl16fpinfoactual}
    \begin{tabular}[h]{|c|c|}
    \hline
        Fixed Point Information &  \\
        \hline
        $\PSL(2,2^4)$ & $[[2,3,5], [2,3,5,17]]$ \\
        \hline
    \end{tabular}\\
    The list refers to all prime element orders where elements of given prime order in the list have fixed points in some representation, found through \cite{GAP}. There are different lists within the list because there are different irreducible representations have different fixed points.
\end{fact}
\begin{lemma}\thlabel{psl16brauer}
    Let $p \in \pi(\PSL(2, 2^4))$, let $N$ be a nontrivial $p$-group, and suppose $G \cong N.\PSL(2, 2^4)$. Then we have the following restrictions on $\pgc(G)$, where the white vertex denotes 17 and the other vertices correspond to the labeling at the beginning of this subsection:
    \begin{itemize}
        \item If $p  = 2$ then $\pgc(G)$ realizes one of \hamSandwichWhite, \szunrecognizei, \fletchingWhite, or \szunrecognizeii.
        \item If $p = 3$ then $\pgc(G)$ is  \szunrecognizei or \trianglePlusIsolatedExtra.
         \item If $p = 5$ then $\pgc(G)$ is \szunrecognizeiii or \notIsolatedVertexWhite.
        \item If $p= 17$ then $\pgc(G)$ is \brokenFrameIsolatediii.
    \end{itemize}
\end{lemma}
\begin{proof}
    This follows from the Brauer tables of $\PSL(2, 2^4)$ and \thref{GeneralizedModuleExtensionsCited}. We examine the fixed point information of the representations in the tables below, where ``Yes" in column $x$ means ``An element of order $x$ has fixed points in the representation" and ``No" means ``An element of order $x$ does not have fixed points in the representation". The tables below are computed using \cite[Lemma 6.2]{BrauerTableFP} and \cite{GAP}, but by \thref{ConjugacyClassLanding}, they could also be computed manually.
    \begin{center} 
    \mbox{
    \begin{tabular}{||c c c c||} 
     \hline
     $\chi_i \in \operatorname{IBr}_{2}(\PSL(2,2^{4}))$ & 3 & 5 & 17 \\ [0.5ex] 
     \hline\hline
     $\chi_1$ & Yes & Yes & Yes \\ 
     \hline
     $\chi_2$\text{ through }$\chi_5$& No & No & No \\ 
     \hline
     $\chi_6$\text{ through }$\chi_{7}$& Yes & Yes & No \\ 
     \hline
     $\chi_8$\text{ through }$\chi_{11}$& Yes & No & No \\ 
     \hline
     $\chi_{12}$\text{ through }$\chi_{16}$& Yes & Yes & No \\ [1ex]
     \hline
    \end{tabular}
    \quad
    \begin{tabular}{||c c c c||} 
     \hline
     $\chi_i \in \operatorname{IBr}_{3}(\PSL(2,2^4))$ & 2 & 5 & 17\\ [0.5ex] 
     \hline\hline
     $\chi_1$ \text{ through } $\chi_{9}$ & Yes & Yes & Yes\\
     \hline
      $\chi_{10}$ & Yes & Yes & No\\
     \hline
      $\chi_{11}$\text{ through }$\chi_{12}$ & Yes & Yes & Yes\\[1ex]
     \hline
    \end{tabular}
    }
    \end{center}
    \begin{center}
    \mbox{
    \begin{tabular}{||c c c c||} 
     \hline
     $\chi_i \in \operatorname{IBr}_{5}(\PSL(2,2^4))$ & 2 & 3 & 17 \\ [0.5ex] 
     \hline\hline
     $\chi_1$ through $\chi_{9}$ & Yes & Yes & Yes \\ 
     \hline
     $\chi_{10}$ & Yes & Yes & No\\
     \hline
     $\chi_{11}$ & Yes & Yes & Yes\\   [1ex] 
     \hline
    \end{tabular}
    \quad
    \begin{tabular}{||c c c c||} 
     \hline
     $\chi_i \in \operatorname{IBr}_{17}(\PSL(2, 2^4))$ & 2 & 3 & 5 \\ [0.5ex] 
     \hline\hline
     $\chi_1$ \text{ through }$\chi_{9}$ & Yes & Yes & Yes \\ [1ex] 
     \hline
    \end{tabular}
    }\vspace{0.2cm}
    \end{center}
\end{proof}
\begin{lemma}\thlabel{Suzuki32PartialRuleOut}
    The graph $\SuzukiThirtyTwoElusiveGraphWhiteBadb$ where 17 is the white vertex is not realizable by a $\PSL(2, 2^4)$-solvable group.
\end{lemma}
\begin{proof}
    For contradiction, suppose otherwise. By \cite[Lemma 4.4]{Suz}, it suffices to show that for a $\PSL(2, 2^4)$-solvable group $H$ such that $\pi(H) = \pi\left(\PSL(2, 2^4)\right)$, $\pgc(H)$ cannot realize $\SuzukiThirtyTwoElusiveGraphWhiteBadb$. By \cite[Lemma 2.1.1]{2023REU}, there is some solvable group $N$ and some $S$ such that $\operatorname{Inn}\left(\PSL(2, 2^4)\right)\leq S\leq \Aut\left(\PSL(2, 2^4)\right)$ and $H\cong N.S$. Because $\left|\operatorname{Out}\left(\PSL(2, 2^4)\right)\right| = 4$, there are 3 possibilities by \thref{subgroupsOfAutPSL216}: $S\cong \PSL(2, 2^4), S\cong \PSL(2, 2^4).C_2, S\cong\Aut\left(\PSL(2, 2^4)\right)$. We may consider the $K\leq H$ such that $K/N\cong \PSL(2, 2^4)$. The vertex labeled 2 must be degree 3 in $\pgc(H)$ because $\pgc(H)\subseteq\pgc\left(\PSL(2, 2^4)\right)$ and the vertex labeled 17 is not degree 3, so $17-2, 3-2, 5-2\in\pgc(H)$. Thus, $3$, 5, and 17 cannot divide $|N|$ by \cite[Lemma 5.11]{Suz}. We must then have that $N$ is a 2-group such that $\pgc(K)\cong\hamSandwichWhite$, as it is the only possible graph produced by a 2-group that is a supergraph of \SuzukiThirtyTwoElusiveGraphWhiteBadb by \thref{psl16brauer}. If $G\cong K$, we reach a contradiction to \thref{psl16brauer} and are done. If $G\cong K.C_2$, note that $G\cong N.\PSL(2, 2^4).C_2$. Because $\pgc(G) \cong \SuzukiThirtyTwoElusiveGraphWhiteBadb$, the edges $5-3, 17-p\not\in\pgc(G)$ for some $p\in\{2, 3\}$, but also note that the edge $3-2\not\in\pgc\left(\PSL(2, 2^4).C_2\right)$, which is a supergraph of $\pgc(G)$. Thus, $2-3, 17-p, 3-5\not\in\pgc(G)$ for some $p\in\{2, 3\}$, so $\pgc(G)\not\cong\SuzukiThirtyTwoElusiveGraphWhiteBadb$. If $G\cong N.\Aut\left(\PSL(2, 2^4)\right)$, note that $H\cong N.\PSL(2, 2^4).C_2$ is a subgroup of $G$, so $\pgc(G)\subseteq\pgc(H)$ and we reach a contradiction. 
\end{proof}
We are now ready to verify the criteria given in \ref{sec:PSL216General}. First make the assignments $a = 3, c = 5, d = 17$. By \thref{psl16fact} and the graph in the introduction to this subsection, $\pgc\left(\PSL(2, 2^4)\right)$ matches the graph given in the previous section, $\pi(T) = \pi\left(\Aut(T)\right)$, and Criteria 1 and 5 are satisfied. By \thref{psl16fpinfoactual}, Criterion 2 is satisfied; by \thref{psl16brauer} and \cite[Theorem 2.2]{Suz}, Criterion 3 is satisfied; and by \thref{Suzuki32PartialRuleOut}, Criterion 4 is satisfied. Thus, the classification result for the prime graph complements of $\PSL(2, 2^4)$-solvable groups is given by \thref{generalClassificationPSL216}.
\section{$T$-solvable groups with $T$ similar to $\PSL(2, 3^3)$ and $\PSL(2, 7^2)$}\label{sec:PSL249}
In this section, we classify the prime graph complements of $T$-solvable groups where $T$ satisfies the criteria given in Section \ref{sec:General}. We then apply the result to classify the prime graph complements of $\PSL(2, 3^3)$- and $\PSL(2, 7^2)$-solvable groups in Sections \ref{sec:PSL227New} and \ref{sec:PSL249new}, respectively.
\subsection{General $T$}\label{sec:General}
\begin{figure}[H]
    \centering
    \begin{minipage}{\textwidth}
        \centering
        \begin{tikzpicture}
            % Define coordinates for the vertices
            \coordinate (A) at (0, 1);   % Vertex 2
            \coordinate (B) at (1, 1);   % Vertex 7
            \coordinate (C) at (1, 0);   % Vertex 13
            \coordinate (D) at (0, 0);   % Vertex 5
            % Draw the edges of the complete graph
            \draw (B) -- (C);
            \draw (C) -- (D);
            \draw (A) -- (D);
            \draw (A) -- (B);
            \draw (A) -- (C);
            % Draw the vertices (circles only)
            \foreach \point in {A, B, C, D}
                \draw[draw, fill=white] (\point) circle [radius=0.2cm];
            % Draw the text on top of the circles
            \foreach \point/\name in {(A)/a, (B)/2, (C)/c, (D)/d}
                \node at \point {\scriptsize\name};
        \end{tikzpicture}
        \caption*{$\pgc(T)$}
    \end{minipage}%
\end{figure}
We establish soome notation specific for Section \ref{sec:PSL249}.
\begin{notation} \thlabel{rootedGraphIsomorphismPSL227}
    Throughout this section, given a group $G$ with $\pi(G) = \pi(\Aut(T))$ and a rooted graph $\Lambda$ on four vertices, it is understood that when we say $\pgc(G) \cong \Lambda$, we require the isomorphism to map the vertex $2$ to the root.
\end{notation}
\begin{notation}\thlabel{realizableDefPSL227}
    Given a set of rooted graphs $\mathcal{H}$ on four vertices, we say that $\mathcal{H}$ is \emph{realizable} if for each $\Lambda \in \mathcal{H}$, there exists a $T$-solvable $\pi(T)$-group $G$ such that $\pgc(G) \cong \Lambda$ in the sense of \thref{rootedGraphIsomorphismPSL227}.
\end{notation}
In this section, we classify the prime graphs of $T$-solvable groups where $T$ is a nonabelian simple group with $\pi(T) = \pi\left(\Aut(T)\right) = \{2, a, c, d\}$ such that $T$ has the prime graph listed above, and such that $T$ satisfies the following criteria:
\begin{itemize}
    \item[1.] The Schur multiplier of $T$ is 2.
    \item[2.] The fixed point information for $T$ matches \thref{psl249fpinfo1}.
    % \item[3.] The Brauer characters for $T$ satisfy \thref{psl249Brauer} where a ''yes" appears in the table of $\operatorname{IBr}_p(T)$ when elements of the given order in the column have fixed points on elements of order $p$ in the representations corresponding to the row. For each $p\in\pi(T)$, every $\chi\in\operatorname{IBr}_p(T)$ must fall into one of the rows, and every row in the tables must have the same fixed point information as some $\chi\in\operatorname{IBr}_p(T)$.
    \item[3.] There exists an $R$ with $\operatorname{Inn}(T)\leq R\leq \Aut(T)$ such that $\pgc(R)$ realizes \SuzukiThirtyTwoElusiveGraph.
\end{itemize}
\begin{fact}\thlabel{psl249fpinfo1}
    \begin{tabular}[h]{|c|c|}
    \hline
        Fixed Point Information &  \\
        \hline
        $T$ & $[ [2,a, c, d]]$ \\
        \hline
        $2.T$ & $[[ 2, a, c, d], [ a, c, d] ]$ \\
        \hline
    \end{tabular}\\
    The first column of the table above contains all perfect central extensions of $T$. For each group $G$ in the first column, the second column contains a list of lists $L$. A list $\ell$ appears in $L$ if and only if there exists an irreducible complex representation $G \to \GL(n, \C)$ for which $\ell$ is the set of primes $p$ for which there exists an element of order $p$ in $G$ which fixes some vector in $\C^n \setminus \{0\}$.
\end{fact}

\begin{theorem}\thlabel{StructureofPSL249SolvableGraphs}
    Let $G$ be a strictly $T$-solvable group. Then $\pgc(G)$ satisfies both of the following:
    \begin{enumerate}
        \item Either $\{2, a,c, d\}$ is a union of connected components, or $\{a, c, d\}$ is a union of connected components.
        \item $\pgc(G) \setminus \{a, c, d\}$ is triangle-free and 3-colorable.
    \end{enumerate}
\end{theorem}
\begin{proof}
    We prove (1) first. By \cite[Lemma 2.1.2 and Corollary 2.2.6]{2023REU} and Criterion 2 for $T$, we have that there are no edges of the form $r-p$ for $r \in \pi(T) \setminus \{2\}$ and $p \in \pi(G) \setminus \pi(T)$. If there is no edge $2-p$ for $p \in \pi(G) \setminus \pi(T)$, then $\{2, a, c, d\}$ is a union of connected components, and we are done. So assume such an edge exists. By \thref{isolate2psl211}, $2-q \notin \pgc(G)$ for all $q \in \{a, c, d\}$, meaning $\{a, c, d\}$ is a union of connected components.

    For (2), by \thref{lemma212inPaper}, there exists a subgroup $K \leq G$ and a solvable group $N$ with $K \cong N.T$ such that $\pi(K) = \pi(G)$. There exists a $J \leq K$ such that $J/N$ is isomorphic to a Sylow $2$-subgroup of $T$. Then, $J$ is solvable, so $\pgc(J)$ is triangle-free and 3-colorable by \cite{2015REU}. Thus, (2) is satisfied.

    % To prove (1), first apply \thref{corollaryLemma} to \thref{psl249fpinfo}. Then, there are no edges of the form $r-p$ for $r\in \pi(\PSL(2, 7^2))\setminus\{2\}$ and $p\in \pi(G) \setminus (\pi(\PSL(2, 7^2)))$. There are two cases: There is an edge between some $r\in \pi(\PSL(2, 7^2))\setminus\{2\}$ and $p\in \pi(G) \setminus (\pi(\PSL(2, 7^2)))$ and the case where there is no such edge. For the first case, this implies that for any $r\in \pi(\PSL(2, 7^2))\setminus\{2\}$ and $p\in \pi(G) \setminus (\pi(\PSL(2, 7^2)))$, there is no edge $r-p\in\pgc(G)$. If there is no edge of the form $2-p\in\pgc(G)$ where $p\in\pi(G)\setminus\pi(\PSL(2, 7^2))$, then the first case of (1) holds, so consider the case where an edge of that form exists in $\pgc(G)$. Note that the Sylow 2-subgroups of $\PSL(2, 7^2)$ are isomorphic to dihedral groups, $\pi(\PSL(2, 7^2)) = \pi(\Aut(\PSL(2, 7^2)))$, and there exists some perfect central extension $2.\PSL(2, 7^2)$, so we can apply \thref{isolate2psl211}. Then, the second case of (1) holds. 
    
    %For case 2, by \cite[Lemma 5.2]{Suz}, there exists a subgroup $J\leq G$ and a solvable group $S$ with $J\cong S.\PSL(2, 7^2)$ such that $\pi(K) = \pi(G)$. Then, there exists a $K\leq J$ such that $K/S$ is isomorphic to a Sylow subgroup of $\PSL(2, 7^2)$. Then, $K$ is solvable, so $\pgc(K)$ is triangle-free and three-colorable by \cite{2015REU}. Thus, $\pgc(G)[\pi(K)]\subseteq\pgc(K)$ is triangle-free and three-colorable. Then, (2) is satisfied.

\end{proof}

% \begin{corollary}\thlabel{AlmostHallSubgroupsPSL249}
%     Let $G$ be a $\PSL(2,7^2)$-solvable group with exactly one copy of $\PSL(2,7^2)$ in its composition series, and let $\pi_0$ be a set of primes such that $\pi_0 \supseteq \pi(T)$.
%     Then there exists a Hall $\pi_0$-subgroup $H \leq G$ such that $H$ is $\PSL(2,7^2)$-solvable, $A_5$-solvable, or $\PSL(3, 2)$-solvable and $\pgc(H)$ is isomorphic to the subgraph of $\pgc(G)$ induced by $\pi(H)$.
% \end{corollary}

% \begin{proof}
%     This corollary follows quickly from \cite[Lemma 4.4]{Suz} however with the addition that there are exactly two nonsolvable subgroups of $\PSL(2, 7^2)$. Because there is exactly one copy of $\PSL(2,7^2)$ in the composition series of $G$, it cannot be $\{A_5, \PSL(2, 7^2)\}$-solvable by an application of the proof of  \cite[Lemma 4.1]{Suz}.
%     \end{proof}
% \begin{lemma}\thlabel{psl249brauer}
%     Let $p \in \pi(T)$, let $N$ be a nontrivial $p$-group, and let $G \cong N.T$. Then we have the following restrictions on $\pgc(G)$:
%     \begin{itemize}
%         \item If $p \in \{b, d \}$ then $\pgc(G)$ is isomorphic to \trianglePlusIsolated.
%         \item If $p = c$ then $\pgc(G)$ is isomorphic to \twoarms
%         \item if $p= a$ then $\pgc(G)$ is isomorphic to \twoarms or \fletching.
%     \end{itemize}
% \end{lemma}
% \begin{proof}
%     This follows from the Brauer tables of $T$ (condition 3) and \thref{GeneralizedModuleExtensionsCited}.
% \end{proof}
We say that a graph $\Xi$ on four vertices is realized by a $T$-solvable group if there exists a group $G$ with $\pi(G) = \pi(T)$ such that $\pgc(G)$ is isomorphic to $\Xi$.
\begin{proposition}\thlabel{factPSL249}
    All subgraphs of \hamSandwich are realizable by a $T$-solvable group.
\end{proposition}
\begin{proof}
    Any triangle free and 3-colorable graph can be realized by some solvable group via methods in \cite{2015REU}. There are three remaining graphs to check:
    \begin{itemize}
        \item Any triangle free and 3-colorable graph can be realized via some solvable group by methods in \cite{2015REU}.
        \item The graph \SuzukiThirtyTwoElusiveGraph is realized by the subgroup of $\Aut(T)$ guaranteed by Criterion 3.
        \item The graph \hamSandwich is realized by $T$.
        \item The graph \trianglePlusIsolated can be realized via the central extension $2.T$, or by $C_d\times T$.
    \end{itemize}
\end{proof}

\begin{theorem}\thlabel{generalThm}
Given a graph $\Xi$, we have that $\Xi$ is isomorphic to the prime graph complement of some $T$-solvable group if and only if one of the following is true: 
\begin{enumerate}
    \item $\Xi$ is triangle-free and 3-colorable.
    \item $\Xi$ is 3-colorable and contains a subset $X = \{w, x, y, z\} \subseteq V(\Xi)$ with $\Xi[X] \cong \notIsolatedVertexWhiteBad$ where $w$ is the white vertex; for all $p\in \pi(G)\setminus \pi(T)$ such that $p-q\in\pgc(G)$ and $q\in\pi(T)$, we have that $q = w$; and that in some 3-coloring of $\Xi$ the closed neighborhood $N(X) \setminus X$ shares one color. Also, the only triangle of $\Xi$ is $\Xi[x, y, z]$. 
    \item $\Xi$ is 3-colorable and contains a subset $X = \{w, x, y, z\} \subseteq V(\Xi)$ such that $N(X) = X$, and $\Xi[X]\cong\trianglePlusIsolated$, $\Xi[X]\cong\SuzukiThirtyTwoElusiveGraph$, or $\Xi[X]\cong\hamSandwich$. Furthermore, all triangles are contained in $\Xi[X]$.
\end{enumerate}
\end{theorem}
\begin{proof}
We will first prove the forwards direction. Let $G$ be a $T$-solvable group. If $G$ is not strictly $T$-solvable, $G$ must be solvable. Then, $\pgc(G)$ satisfies (1) by \cite{2015REU} so consider the case where $G$ is strictly $T$-solvable. If there exists an edge in $\pgc(T)$, then by \cite[Lemma 2.1.1]{2023REU}, there exists a solvable normal subgroup $N\unlhd G$ such that, for some $S$ with $\operatorname{Inn}(T)\leq S\leq \Aut(T)$, $G\cong N.S$. We now consider two cases: The case in which there is some $p\in\pi(G)\setminus\pi(T)$ such that $2-p\in\pgc(G)$ and the case where no such $p$ exists.

We first consider the case where there is some $p\in\pi(G)\setminus\pi(T)$ such that $2-p\in\pgc(G)$. By \thref{noworry2}, $\pgc(G)[\pi(T))]\subseteq\notIsolatedVertexWhiteBad$ where $2$ is the white vertex. We can see that $\pgc(G)[\pi(T)]$ is 3-colorable because $\pgc(G)[\pi(T)]\subseteq\hamSandwich$, which is 3-colorable. By \cite[Lemma 2.3.5]{2023REU}, $\pgc(G)$ is 3-colorable and all triangles are contained in $\pgc(G)[\pi(T)]$. As such, if there is no triangle in $\pgc(G)[\pi(T)]$, $\pgc(G)$ satisfies (1).

% Let $K$ be a subgroup of $G$ such that $K/N$ is isomorphic to a Sylow 2-subgroup of $\PSL(2, 7^2)$. Then, $\pgc(G)[(\pi(G)\setminus\pi(N))\cup\{2\}]$ is triangle-free and three-colorable by \cite{2015REU}. As such, if there is no triangle in $\pgc(G)[\pi(\PSL(2, 7^2)))]$, the graph is triangle-free and 3-colorable, so it satisfies (1).

If there is a triangle, label the vertices $d, c, a$ with $x, y, z$; label $2$ with $w$; and let $X = \{w, x, y, z\}$. Because $\pgc(G)[\pi(T)]$ is isomorphic to a subgraph of $\notIsolatedVertexWhiteBad$ and there exists a triangle in $\pgc(G)[\pi(T)]$, we have $\pgc(G)[\pi(T)]\cong\notIsolatedVertexWhiteBad$. By \thref{StructureofPSL249SolvableGraphs} (1), $x-p\in\pgc(G), p\in\pi(T)$ implies $p = w$. Let $R\leq S$ such that $R/N$ is a Sylow $2$-subgroup of $S$. Then $R$ is solvable and $\pgc(R)$ is 3-colorable by \cite{2015REU}, so $\pgc(G)[\pi(N)\cup\{b\}]$ is 3-colorable. Thus, by the fact that $\pgc(G)[\pi(T)]\cong\notIsolatedVertexWhiteBad$, $\pgc(G)[\pi(T)]$ is 3-colorable. Then, by \cite[Lemma 2.3.5]{2023REU} and \thref{StructureofPSL249SolvableGraphs}, we have that there is a three-coloring of $\pgc(G)$ such that all vertices of $N(X) \setminus X$ share one color, and all triangles are contained in $\pgc(G)[\pi(T)]$. Thus, $\pgc(G)\setminus X$ is triangle free, so (2) is satisfied.

We next consider the case where there is no $p\in\pi(G)\setminus\pi(T)$ such that $2-p\in\pgc(G)$. By the proof of \thref{StructureofPSL249SolvableGraphs}, $\{2, a, c, d\}$ is the union of connected components. In particular, we have $N(\pi(T)) = \pi(T)$. By \cite[Lemma 2.3.5]{2023REU} and the fact that $\pgc(G)[\pi(T)]\subseteq\pi(T)$, which is a 3-colorable graph, $\pgc(G)$ satisfies (1) if $\pgc(G)[\pi(T)]$ contains no triangles so we may assume that $\pgc(G)[\pi(T)]$ contains a triangle. Let $X = \pi(T)$. By \thref{factPSL249}, we have $\pgc(G)[\pi(T)]$ is isomorphic to one of $\trianglePlusIsolated, \SuzukiThirtyTwoElusiveGraph$, or $\hamSandwich$. $\pgc(G)[\pi(T)]$ is 3-colorable, so by \cite[Lemma 2.3.5]{2023REU}, $\pgc(G)$ is 3-colorable and all triangles are contained in $\pgc(G)[\pi(T)]$. By the proof of \thref{StructureofPSL249SolvableGraphs}, $N(X) = X$. Thus, $\pgc(G)$ satisfies (3). 

% There exists a $R\leq G$ such that $\pi(R) = \pi(\PSL(2, 7^2))$ with $R\cong Y.\PSL(2, 7^2)$ 
% If there does not exist an edge in $\pgc(\PSL(2, 7^2))$, then by \thref{StructureofPSL249SolvableGraphs}, $\pgc(G)$ is triangle-free and three-colorable.

We will now prove the backwards direction. To do this, we will split into cases. Let $\Xi$ be a graph that satisfies one of the criteria listed in the statement. If the graph $\Xi$ satisfies (1), then there exists a solvable group $G$ such that $\pgc(G)\cong\Xi$ by \cite{2015REU}. 

Now suppose that $\Xi$ satisfies (2). Color $\Xi$ with the colors $\mathcal{O}, \mathcal{I}, \mathcal{D}$ such that all of $N(X)\setminus X$ is colored $\mathcal{I}$, $w$ is colored $\mathcal{O}$, $x$ is colored $\mathcal{O}$, $y$ is colored $\mathcal{I}$, and $z$ is colored $\mathcal{D}$. Assign $w$ to $2, x$ to $d$, $y$ to $a$, and $z$ to $c$ to obtain a graph isomorphism from $\pgc(C_2.T)$ to $\Xi[X]$ (from \thref{factPSL249}). We know that there are irreducible representations of $C_2.T$ such that elements of orders $a, c,$ and $d$ have fixed points while elements of order $2$ have no fixed points on the representation by \thref{psl249fpinfo1}. Then, by \thref{236Improved}, $\Xi$ is the prime graph complement of a $T$-solvable group.

% Given $v\in\Xi\setminus X$, there exists a complex irreducible representation $V$ of $C_2.\PSL(2, 7^2)$ such that for all $x\in X$, $x-v\in\Xi$ if and only if elements of order $\varphi(x)$ of $C_2.\PSL(2, 7^2)$ act without fixed points on $V$ by \thref{psl249fpinfo}. Then, by \cite[Theorem 2.3.6]{2023REU}, $\Xi$ is the prime graph complement of a $\PSL(2, 7^2)$-solvable group.

Finally, if the graph $\Xi$ satisfies (3), there exists a $T$-solvable group $G_0$ such that $\pgc(G_0)\cong \Xi[X]$ by \thref{factPSL249}. By the methods in \cite[Theorem 2.8]{2015REU}, there exists a solvable group $N$ such that $(|G_0|, |N|) = 1$ and $\pgc(N)\cong\Xi\setminus X$. Then, define $G = G_0\times N$. We have $\pgc(G)\cong \Xi$.
\end{proof}
\subsection{$\PSL(2, 3^3)$}\label{sec:PSL227New}
\begin{figure}[H]
    \centering
    \begin{minipage}{.5\textwidth}
        \centering
        \begin{tikzpicture}
            % Define coordinates for the vertices
            \coordinate (A) at (0, 1);   % Vertex 2
            \coordinate (B) at (1, 1);   % Vertex 7
            \coordinate (C) at (1, 0);   % Vertex 13
            \coordinate (D) at (0, 0);   % Vertex 5
            % Draw the edges of the complete graph
            \draw (B) -- (C);
            \draw (C) -- (D);
            \draw (A) -- (D);
            \draw (A) -- (B);
            \draw (A) -- (C);
            % Draw the vertices (circles only)
            \foreach \point in {A, B, C, D}
                \draw[draw, fill=white] (\point) circle [radius=0.2cm];
            % Draw the text on top of the circles
            \foreach \point/\name in {(A)/3, (B)/2, (C)/13, (D)/7}
                \node at \point {\scriptsize\name};
        \end{tikzpicture}
        \caption*{$\pgc(\PSL(2,3^3))$}
    \end{minipage}%
    \begin{minipage}{.5\textwidth}
        \centering
        \begin{tikzpicture}
            % Define coordinates for the vertices
            \coordinate (A) at (0, 1);   % Vertex 2
            \coordinate (B) at (1, 1);   % Vertex 7
            \coordinate (C) at (1, 0);   % Vertex 13
            \coordinate (D) at (0, 0);   % Vertex 5
            % Draw the edges of the complete graph
            \draw (C) -- (D);
            \draw (A) -- (D);
            \draw (A) -- (C);
            % Draw the vertices (circles only)
            \foreach \point in {A, B, C, D}
                \draw[draw, fill=white] (\point) circle [radius=0.2cm];
            % Draw the text on top of the circles
            \foreach \point/\name in {(A)/3, (B)/2, (C)/13, (D)/7}
                \node at \point {\scriptsize\name};
        \end{tikzpicture}
        \caption*{$\pgc(\Aut(\PSL(2,3^3)))$}
    \end{minipage}%
\end{figure}
We classify the prime graph complements of $\PSL(2, 3^3)$-solvable groups in this section. The order of the group $\PSL(2, 3^3)$ is $9828 = 2^2 \cdot3^3\cdot7\cdot13$ and its automorphism group has order $58968 = 2^3 \cdot3^4\cdot7\cdot13$ by \cite{GAP}. Let $a = 3, c = 7, d = 13$. Notice $\pi\left(\PSL(2, 3^3)\right) = \pi\left(\Aut\left(\PSL(2, 3^3)\right)\right)$. The Schur multiplier of $\PSL(2, 3^3)$ is 2. The fixed point information for $\PSL(2, 3^3)$ matches the table in \thref{psl227fpinfo} listed below.

\begin{fact}\thlabel{psl227fpinfo}
    \begin{tabular}[h]{|c|c|}
    \hline
        Fixed Point Information &  \\
        \hline
        $\PSL(2,3^3)$ & $[ [ 2, 3, 7, 13 ] ]$ \\
        \hline
        $2.\PSL(2,3^3)$ & $[ [ 2, 3, 7, 13 ], [ 3, 7, 13 ] ]$ \\
        \hline
    \end{tabular}\\
    The list refers to all element orders where elements have fixed points in some representation. Then the multiple lists refers to the fact that different irreducible representations have different fixed points.
\end{fact}
Notice $\pgc\left(\operatorname{PGL}(2, 3^3)\right)$ realizes $\SuzukiThirtyTwoElusiveGraph$. By \cite[Page 48]{FSGBook}, $\operatorname{PGL}(2, 3^3)$ is isomorphic to a $\PSL(2, 3^3)$-solvable subgroup of $\Aut\left(\PSL(2, 3^3)\right)$.
% \begin{lemma}\thlabel{BrauerPSL227}
%     There is a $\PSL(2,27)$-solvable group which realizes the graph \SuzukiThirtyTwoElusiveGraph.
% \end{lemma}
% \begin{proof}
%     Similarly to \cite[Corollary 2.3]{Suz}, the Brauer character $\chi_{11}$ of the Brauer table of $\PSL(2,27)$ mod 2 shows there exists a finite $\PSL(2,27)$-solvable group $G$ such that $\pgc(G) \cong \SuzukiThirtyTwoElusiveGraph$.
% \end{proof}
Thus, by the above facts, $\PSL(2, 3^3)$ satisfies the criteria given in Section \ref{sec:General}. Then, \thref{generalThm} with $T = \PSL(2, 3^3)$ completely classifies the prime graph complements of $\PSL(2, 3^3)$-solvable groups.
\subsection{$\PSL(2, 7^2)$}\label{sec:PSL249new}
\begin{figure}[H]
    \centering
    \begin{minipage}{.5\textwidth}
        \centering
        \begin{tikzpicture}
            % Define coordinates for the vertices
            \coordinate (A) at (0, 1);   % Vertex 2
            \coordinate (B) at (1, 1);   % Vertex 7
            \coordinate (C) at (1, 0);   % Vertex 13
            \coordinate (D) at (0, 0);   % Vertex 5
            % Draw the edges of the complete graph
            \draw (B) -- (C);
            \draw (C) -- (D);
            \draw (A) -- (D);
            \draw (A) -- (B);
            \draw (A) -- (C);
            % Draw the vertices (circles only)
            \foreach \point in {A, B, C, D}
                \draw[draw, fill=white] (\point) circle [radius=0.2cm];
            % Draw the text on top of the circles
            \foreach \point/\name in {(A)/7, (B)/2, (C)/5, (D)/3}
                \node at \point {\scriptsize\name};
        \end{tikzpicture}
        \caption*{$\pgc(\PSL(2,7^2))$}
    \end{minipage}%
    \begin{minipage}{.5\textwidth}
        \centering
        \begin{tikzpicture}
            % Define coordinates for the vertices
            \coordinate (A) at (0, 1);   % Vertex 2
            \coordinate (B) at (1, 1);   % Vertex 7
            \coordinate (C) at (1, 0);   % Vertex 13
            \coordinate (D) at (0, 0);
            % Draw the edges of the complete graph
            \draw (C) -- (D);
            \draw (A) -- (D);
            \draw (A) -- (C);
            % Draw the vertices (circles only)
            \foreach \point in {A, B, C, D}
                \draw[draw, fill=white] (\point) circle [radius=0.2cm];
            % Draw the text on top of the circles
            \foreach \point/\name in {(A)/7, (B)/2, (C)/5, (D)/3}
                \node at \point {\scriptsize\name};
        \end{tikzpicture}
        \caption*{$\pgc(\Aut(\PSL(2,7^2)))$}
    \end{minipage}
\end{figure}

In this section, we classify the prime graphs of $\PSL(2, 7^2)$-solvable groups. The group $\PSL(2, 7^2)$ has order $58800 = 2^4 \cdot3\cdot5^2\cdot 7^2$ and its automorphism group has order $235200 = 2^6\cdot3\cdot5^2\cdot7^2$ by \cite{GAP}. Assign $a = 7, c = 5$, and $d = 3$. The Schur multiplier of $\PSL(2,7^2)$ is 2, and the fixed point information of $\PSL(2, 7^2)$ is given in \thref{psl249fpinfosubsection}.
\begin{fact}\thlabel{psl249fpinfosubsection}
    \begin{tabular}[h]{|c|c|}
    \hline
        Fixed Point Information &  \\
        \hline
        $\PSL(2,7^2)$ & $[ [2,3,5,7]]$ \\
        \hline
        $2.\PSL(2,7^2)$ & $[[ 2, 3, 5, 7 ], [ 3, 5, 7 ] ]$ \\
        \hline
    \end{tabular}\\
    The first column of the table above contains all perfect central extensions of $\PSL(2, 7^2)$. For each group $G$ in the first column, the second column contains a list of lists $L$. A list $\ell$ appears in $L$ if and only if there exists a representation $G \to \Aut(\C^n)$ for which $\ell$ is the set of primes $p$ for which there exists an element of order $p$ in $G$ which fixes some vector in $\C^n \setminus \{0\}$.
    % The list refers to all prime element orders where elements of that order have fixed points in some representation. Then the multiple lists refers to the fact that different irreducible representations have different fixed points.
\end{fact}
The graph \SuzukiThirtyTwoElusiveGraph is realized by $\operatorname{PGL}(2, 7^2)$, which is isomorphic to a subgroup of $\Aut\left(\PSL(2, 7^2)\right)$ that contains $\operatorname{Inn}\left(\PSL(2, 7^2)\right)$ by \cite[Page 48]{FSGBook}.

% of order 117600 found using our GAP code repository. The group's structure description is $\PSL(2, 7^2):C_2$. 
%%OR SHOULD WE USE
%The graph \SuzukiThirtyTwoElusiveGraph can be realized by a subgroup of $\Aut(\PSL(2, 7^2))$ that contains $\operatorname{Inn}(\PSL(2, 7^2))$ of order 117600 found using our GAP code repository. The group's structure description is $\PSL(2, 7^2):C_2$, and it is accessible via the command \verb|PrimitiveGroup( 1225, 16 )| in \cite{GAP}.
Thus, $\PSL(2, 7^2)$ satisfies the criteria of \ref{sec:General}, so \thref{generalThm} with $T = \PSL(2, 7^2)$ completely classifies the prime graph complements of $\PSL(2, 7^2)$-solvable groups.
\section{$T$-solvable groups with $T$ similar to $\PSL(2, 11)$, $\PSL(2, 19)$, and $\PSL(2, 23)$}\label{sec:PSLPrimesGeneral}
In this section, we classify the prime graph complements of $T$-solvable groups where $T$ is as in Section \ref{sec:generalT211} and then use the main result of Section \ref{sec:generalT211} to classify the prime graph complements of $\PSL(2, 11)-, \PSL(2, 19)-$, and $\PSL(2, 23)$-solvable groups in Sections \ref{sec:PSL211new}, \ref{sec:PSL219}, and \ref{sec:PSL223new} respectively.
\subsection{General $T$}\label{sec:generalT211}
\begin{figure}[H]
    \centering
    \begin{minipage}{\textwidth}
        \centering
        \begin{tikzpicture}
            % Define coordinates for the vertices
            \coordinate (A) at (0, 1);   % Vertex 2
            \coordinate (B) at (1, 1);   % Vertex 7
            \coordinate (C) at (1, 0);   % Vertex 13
            \coordinate (D) at (0, 0);   % Vertex 5
            % Draw the edges of the complete graph
            \draw (B) -- (C);
            \draw (C) -- (D);
            \draw (A) -- (D);
            \draw (A) -- (B);
            \draw (A) -- (C);
            % Draw the vertices (circles only)
            \foreach \point in {A, B, C, D}
                \draw[draw, fill=white] (\point) circle [radius=0.2cm];
            % Draw the text on top of the circles
            \foreach \point/\name in {(A)/c, (B)/2, (C)/d, (D)/a}
                \node at \point {\scriptsize\name};
        \end{tikzpicture}
        \caption*{$\pgc(T)$}
    \end{minipage}%
\end{figure}
As in previous sections, we first introduce some important notation to be used throughout Section \ref{sec:PSLPrimesGeneral}.
\begin{notation} \thlabel{rootedGraphIsomorphismPSL211}
    Given a group $G$ with $\pi(G) = \pi(\Aut(T))$ and a rooted graph $\Lambda$ on four vertices, it is understood that when we say $\pgc(G) \cong \Lambda$, we require the isomorphism to map the vertex $d$ to the root.
\end{notation}
\begin{notation}\thlabel{realizableDefPSL211}
    Given a set of rooted graphs $\mathcal{H}$ on four vertices, we say that $\mathcal{H}$ is \emph{realizable} if for each $\Lambda \in \mathcal{H}$, there exists a $T$-solvable $\pi(T)$-group $G$ such that $\pgc(G) \cong \Lambda$ in the sense of \thref{rootedGraphIsomorphismPSL211}.
\end{notation}
In this section, we classify the prime graph complements of $T$-solvable groups such that $T$ is a nonabelian simple group such that $\pgc(T)$ is as above, $\pi(T) = \pi(\Aut(T))$, and that satisfies the criteria listed below:
\begin{itemize}
    \item[1.] The Schur multiplier of $T$ is 2.
    \item[2.] The fixed point information for $T$ is as in \thref{psl223fpinfogeneral}.
    \item[3.] The graph \szunrecognizei is realizable by $T\ltimes P$ where $d$ is the white vertex and $P$ is a $p$-group, $p\in\{a, c\}$.
    \item[4.] The graph \SuzukiThirtyTwoElusiveGraphWhiteBadb where $d$ is the white vertex is not realizable by a $T$-solvable group.
\end{itemize}

\begin{fact}\thlabel{psl223fpinfogeneral}
    \begin{tabular}[h]{|c|c|}
    \hline
        Fixed Point Information &  \\
        \hline
        $T$ & $[[2, a, c], [2, a,  c, d]]$ \\
        \hline 
        $2.T$ & $[[2, a, c], [2, a, c, d], [a, c], [a, c, d]]$ \\
        \hline
    \end{tabular}\\
    Each list refers to all prime element orders where elements of an order in a sublist have fixed points in some irreducible complex representation. The multiple lists refer to the fact that different irreducible representations have different fixed points. 
\end{fact}

\begin{lemma}\thlabel{311DisconnectPSL223}
    Let $G$ be a $T$-solvable group. Then there are no edges $c-p$ or $a-p$ in $\pgc(G)$ for any $p\in \pi(G)\setminus \pi(T)$.
\end{lemma}
\begin{proof}
    By \thref{psl223fpinfogeneral}, we find that the primes $a$ and $c$ satisfy \thref{corollaryLemma}.
\end{proof}
\begin{lemma}
    For any $T$-solvable group $G$, if $2-p\in\pgc(G)$ for any $p\in\pi(G)\setminus\pi(T)$, then $2-q\not\in\pgc(G)$ for any $q\in\pi(T)$.
\end{lemma}
\begin{proof}
    $T$ has a perfect central extension $2.T$. The lemma follows from applying \thref{isolate2psl211}.
\end{proof}
\begin{lemma}\thlabel{psl2113color}
    For any $T$-solvable group, $G$, one of the two conditions applies:
    \begin{itemize}
        \item $\pgc(G)$ is triangle free and 3-colorable.
        \item There exists a 3-coloring of $\pgc(G)$ such that all neighbors of $2$ and $d$ not in $\pi(T)$ are the same color.
    \end{itemize}
\end{lemma}
\begin{proof}
    We consider two cases: The case where $G$ is solvable and the case where $G$ is strictly $T$-solvable. In the first case, the first item is satisfied by \cite{2015REU}. In the second case, we consider two subcases: The case where $2-p\not\in\pgc(G)$ for all $p\in\pi(G)\setminus\pi(T)$ and the case where $2-p\in\pgc(G)$ for some $p\in\pi(G)\setminus\pi(T)$. First, note that if there exists any $K\leq G$ with $\pi(K) = \pi(G)$ such that $\pgc(K)$ satisfies the second property, $\pgc(G)$ must also satisfy the second property as $\pgc(G)\subseteq\pgc(K)$.
    
    In the first subcase, notice $T$ is a nonabelian simple group, $\pi(T) = \pi(\Aut(T))$, $G$ is strictly $T$-solvable, and $\pgc(T)$ is 3-colorable. The second condition follows from \cite[Lemma 2.3.5]{2023REU}.
    
    Now we consider the case where $2-p \in \pgc(G)$ for some $p \in \pi(G) \setminus \pi(T)$. Because $\pi(T) = \pi(\Aut(T))$ and there exists a perfect central extension $2.T$ of $T$ by \thref{psl223fpinfogeneral}, we may apply \thref{noworry2} to get $K\leq G, \pi(K) = \pi(G)$ with $K\cong L.2.T$ where $2.T$ is a perfect central extension of $T$ and $L$ is a solvable group of odd order. Let $N = L.2$. Now we will label the vertices in $N(\{2, d\})$ by considering the Frobenius Digraph (\cite[Definition 2.3.3]{2023REU}) of $N$. Color the vertices of $\pgc(K)[\pi(K)\setminus\pi(T)]$ as follows. By \cite[Lemma 2.3.4]{2023REU}, for any $p, q\in\pi(K)\setminus\pi(T)$ with $p-q, d-p\in\pgc(K)$ or $2-p\in\pgc(K)$, $q\to p\in\overrightarrow{\Gamma}(N)$. Thus, color all such vertices with $\mathcal{I}$ (we cannot have $p-q, d-p\in\pgc(K)$ and $2-q\in\pgc(K)$, or else we would get $q\to p, p\to q\in\overrightarrow{\Gamma}(N)$ a contradiction to the definition of Frobenius digraph). Color the remaining vertices of $\pi(K)\setminus\pi(T)$ as follows. If a vertex has nonzero in-degree but zero out-degree in $\overrightarrow{\Gamma}(N)$ color it $\mathcal{I}$. If a vertex has nonzero in-degree and nonzero out-degree color it $\mathcal{D}$. If a vertex has zero in-degree and nonzero out-degree color it $\mathcal{O}$. Label any isolated vertices with the color $\mathcal{O}$. Notice that all vertices in $\pi(G)\setminus\pi(T)$ adjacent to $p\in\pi(T)$ have color $\mathcal{I}$. We finish our coloring of $\pgc(K)$ by coloring $d$ with $\mathcal{O}$, and the remaining three vertices as such: $2$ as $\mathcal{D}$, $a$ as $\mathcal{I}$, and $c$ as $\mathcal{D}$. If this were not a valid 3-coloring, we would have two neighbors $p, q$ with the same color, which would mean that they would both have to be colored $\mathcal{D}$. This means that there would be a directed 3-path between primes $p, q, r, s$ in Frobenius digraph of $N$ with $p, q, r, s\not\in\pi(T)$ by \cite[Lemma 2.3.4]{2023REU}. By the fact that $N$ is solvable, this contradicts \cite[Corollary 2.7]{2015REU}, so this 3-coloring of $\pgc(K)$ must be valid. Notice that all vertices in $\pi(K)\setminus\pi(T)$ adjacent to $d$ or $2$ are colored $\mathcal{I}$. Thus, in both subcases of the case where $G$ is strictly $T$-solvable, the second condition is satisfied.
\end{proof}
\begin{lemma}\thlabel{TriangleJailII}
    Let $G$ be a $T$-solvable group. Then all triangles in $\pgc(G)$ are contained in $\pgc(G)[\pi(T)]$.
\end{lemma}
\begin{proof}
    For contradiction, suppose otherwise. Then there exists a triangle in $\pgc(G)$ which is not contained in $\pgc(G)[\pi(T)]$. This means $\pgc(G)$ is not triangle free, so by the main result of \cite{2015REU}, $G$ must be strictly $T$-solvable. Then, there exists $K\leq G$ such that $K\cong N.T$ for solvable $N$ with $\pi(K) = \pi(G)$ by \thref{lemma212inPaper}. Because $\pgc(G)\subseteq\pgc(K)$, it suffices to show the result for $K$. Let $A$ be a triangle of $\pgc(K)$ not contained in $\pgc(K)[\pi(T)]$. By \cite[Lemma 5.7]{Suz} and the fact that $A$ is not contained in $\pgc(K)[\pi(T)]$, $|V(A)\cap \pi(N)| = 1$. Call the single vertex in $V(A)\cap\pi(N)$ $p$. By \thref{311DisconnectPSL223}, we must have $2-p\in\pgc(K), d-p\in\pgc(K)$, and $2-d\in\pgc(K)$. Because $2-p\in\pgc(K)$ for $p\not\in\pi(T), 2-d\notin\pgc(K)$ by \thref{noworry2}, a contradiction. 
\end{proof}
\begin{lemma}\thlabel{WhichGraphsCanExistPSL223}
    Let $G$ be a $T$-solvable group such that $\pgc(G)$ contains a triangle. Then $\pgc(G)[\pi(T)]$ is isomorphic to one of the following, where the white vertex is $d$:
    \begin{itemize}
        \item $\pgc(G)[\pi(T)] \cong \hamSandwichWhite$, \szunrecognizei, or \notIsolatedVertexWhite.
    \end{itemize}
\end{lemma}
\begin{proof}
 If $\pgc(G)$ contains a triangle, by \thref{TriangleJailII}, that triangle is contained in $\pi(T)$.  We must have $\pgc(G)[\pi(T)]\subseteq\pgc(T)$, so we list the subgraphs of $\pgc(T)$ which contain a triangle: \hamSandwichWhite, \szunrecognizei, \notIsolatedVertexWhite, \SuzukiThirtyTwoElusiveGraphWhiteBadb. \SuzukiThirtyTwoElusiveGraphWhiteBadb is not realizable by a $T$-solvable group by Criterion 4 at the beginning of this section, so by the proof of \cite[Theorem 4.4]{Suz} $\pgc(G)[\pi(T)]$ is isomorphic to one of \hamSandwichWhite, \szunrecognizei, or \notIsolatedVertexWhite. 
\end{proof}

\begin{theorem}\thlabel{TPSL211General}
     A graph $\Xi$ is isomorphic to the prime graph complement of some $T$-solvable group if and only if one of the following is true: 
\begin{enumerate}
    \item $\Xi$ is triangle-free and 3-colorable.
    \item $\Xi$ contains a subset $X = \{w, x, y, z\} \subseteq V(\Xi)$ with white vertex $z$ such that $N(X)\setminus X$ share one color in some 3-coloring of $\Xi$ and all triangles of $\Xi$ are contained in $\Xi[X]$. In addition, one of the following hold:
    \begin{itemize}
        \item[a.] All edges in $\Xi$ between $p\in X$ and $q\in\Xi\setminus X$ have $p = z$ and $\Xi[X]\cong \hamSandwichWhite$, \szunrecognizei, or \notIsolatedVertexWhite.
        \item[b.] All edges in $\Xi$ between $p\in X$ and $q\in\Xi\setminus X$ have $p = z$ or $p = x$ and $\Xi[X]\cong$\notIsolatedVertexWhiteGrey where $x$ is the isolated vertex.
    \end{itemize}
\end{enumerate}
\end{theorem}
\begin{proof}
    We first prove the forward direction. Let $G$ be a $T$-solvable group. If $\pgc(G)$ is triangle-free and 3-colorable, (1) is satisfied so consider the case where $\pgc(G)$ is not triangle-free and 3-colorable. We know that $\pgc(G)$ is 3-colorable by \thref{psl2113color}, so $\pgc(G)$ must contain a triangle. All triangles of $\pgc(G)$ are contained in $\pgc(G)[\pi(T)]$ by \thref{TriangleJailII}. Assign $X$ to $\pi(T)$ by assigning $w = a, x = 2, y = c, z = d$. Applying \thref{psl2113color}, there exists a 3-coloring of $\Xi = \pgc(G)$ such that $N(X)\setminus X$ shares one color. By \thref{psl223fpinfo}, there exist no edges of the form $c-p$ or $a-p$ in $\pgc(G)$ for $p\in\pi(G)\setminus\pi(T)$. By \thref{WhichGraphsCanExistPSL223}, $\pgc(G)[\pi(T)]\cong \hamSandwichWhite$, \szunrecognizei, or \notIsolatedVertexWhite.  If $\pgc(G)[\pi(T)]\cong \hamSandwichWhite$ or \szunrecognizei, all edges in $\pgc(G)$ between $p\in X$ and $q\in\pi(G)\setminus X$ have $p = d$, or else we would reach a contradiction to \thref{noworry2}. Thus, this case satisfies (2a).
    
     If $\pgc(G)[\pi(T)]\cong$\notIsolatedVertexWhite, we have 2 cases: The case where $b$ has no neighbors and the case where $b$ has neighbors.  The first case satisfies (2a) and the second case satisfies (2b). Thus, the forward direction is proved.
     
    We now turn to the backward direction. Let $\Xi$ be a graph that satisfies (1). Then, by \cite{2015REU}, there exists a solvable group $G$ such that $\pgc(G)\cong\Xi$, so $G$ is $T$-solvable. Now consider the case where $\Xi$ satisfies (2a).  If $\Xi[X]\cong\hamSandwichWhite$, there exists a $T$-solvable group $G$ such that $\pgc(G)\cong\Xi$ by \thref{psl223fpinfogeneral} and \thref{236Improved}. If $\Xi[X]\cong\notIsolatedVertexWhite$, let $E = C_c\times T$. For every $T$-solvable group $H$, for $p\in\pi(H)\setminus\pi(T)$, the edge $c-p\not\in\pgc(H)$ by \thref{311DisconnectPSL223}. Then, by \thref{psl223fpinfogeneral} we may apply \thref{NewPGCFromOld} to $E$ to get a $T$-solvable group $G$ such that $\pgc(G)\cong \Xi$. Now suppose $\Xi[X]\cong \szunrecognizei$. Note $\Xi\setminus X$ is triangle-free. There exists a strictly $T$-solvable group $M$ of the form $M\cong T\ltimes P$ for $P$ a $p$-group, $p\in\{a, c\}$ such that $\pgc(M)\cong \szunrecognizei$ where $d$ is the white vertex by Criterion 4. By \thref{psl223fpinfogeneral}, \thref{311DisconnectPSL223}, and our assumptions on $\Xi$; we may apply \thref{NewPGCFromOld} to find a $T$-solvable group $G$ such that $\pgc(G)\cong \Xi$. Now let $\Xi$ be a graph that satisfies (2b). Let the group $E$ be $2.T$ and apply \thref{236Improved} with \thref{psl223fpinfogeneral}. Thus, there exists a $T$-solvable group $G$ such that $\pgc(G)\cong\Xi$.
\end{proof}
\subsection{$\PSL(2,11)$}\label{sec:PSL211new}
\begin{figure}[H]
    \centering
    \begin{minipage}{.5\textwidth}
        \centering
        \begin{tikzpicture}
            % Define coordinates for the vertices
            \coordinate (A) at (0, 1);   % Vertex 2
            \coordinate (B) at (1, 1);   % Vertex 7
            \coordinate (C) at (1, 0);   % Vertex 13
            \coordinate (D) at (0, 0);   % Vertex 5
            % Draw the edges of the complete graph
            \draw (B) -- (C);
            \draw (C) -- (D);
            \draw (A) -- (D);
            \draw (A) -- (B);
            \draw (A) -- (C);
            % Draw the vertices (circles only)
            \foreach \point in {A, B, C, D}
                \draw[draw, fill=white] (\point) circle [radius=0.2cm];
            % Draw the text on top of the circles
            \foreach \point/\name in {(A)/5, (B)/2, (C)/11, (D)/3}
                \node at \point {\scriptsize\name};
        \end{tikzpicture}
        \caption*{$\pgc(\PSL(2,11))$}
    \end{minipage}%
    \begin{minipage}{.5\textwidth}
        \centering
        \begin{tikzpicture}
            % Define coordinates for the vertices
            \coordinate (A) at (0, 1);   % Vertex 2
            \coordinate (B) at (1, 1);   % Vertex 7
            \coordinate (C) at (1, 0);   % Vertex 13
            \coordinate (D) at (0, 0);   % Vertex 5
            % Draw the edges of the complete graph
            \draw (B) -- (C);
            \draw (C) -- (D);
            \draw (A) -- (D);
            \draw (A) -- (C);
            % Draw the vertices (circles only)
            \foreach \point in {A, B, C, D}
                \draw[draw, fill=white] (\point) circle [radius=0.2cm];
            % Draw the text on top of the circles
            \foreach \point/\name in {(A)/5, (B)/2, (C)/11, (D)/3}
                \node at \point {\scriptsize\name};
        \end{tikzpicture}
        \caption*{$\pgc(\Aut(\PSL(2,11)))$}
    \end{minipage}%
\end{figure}

% \begin{fact}\thlabel{psl211fact}
%     $\PSL(2,11)$ contains a subgroup isomorphic to $A_5$. The Shur multiplier of $\PSL(2,11)$ is 2. There exists solvable subgroups of order 55, 10, and 6 within $\PSL(2,11)$. Also, $\pi(\PSL(2,11))=\pi(\Aut(\PSL(2,11)))$.
% \end{fact}
We classify the prime graph complements of $\PSL(2, 11)$-solvable groups in this section. The group $\PSL(2, 11)$ has order $660 = 2^2 \cdot3\cdot5\cdot11$ and its automorphism group has order $1320 = 2^3\cdot3\cdot5\cdot11$. The Schur multiplier of $\PSL(2,11)$ is 2 and $\pi(\PSL(2,11))=\pi(\Aut(\PSL(2,11)))$ by \cite{ATLAS} and \cite{GAP}. Let $d = 11, a = 5, c = 3$. Then \thref{psl211fpinfo} matches Criterion 2 given in \ref{sec:generalT211}.

\begin{fact}\thlabel{psl211fpinfo}
    \begin{tabular}[h]{|c|c|}
    \hline
        Fixed Point Information &  \\
        \hline
        $\PSL(2,11)$ & $[[2,3,5], [2,3,5,11]]$ \\
        \hline
        $2.\PSL(2,11)$ & $[[2,3,5], [2,3,5,11], [3,5], [3,5,11]]$ \\
        \hline
    \end{tabular}\\
    The list refers to all element orders where elements have fixed points in some representation. Then the multiple lists refers to the fact that different irreducible representations have different fixed points. This information was found via \cite{GAP}.
\end{fact}

\begin{lemma}\thlabel{psl211brauer}
    Let $p \in \pi(\PSL(2,11))$, let $N$ be a nontrivial $p$-group, and let $G \cong N.\PSL(2,11)$. Then we have the following restrictions on $\pgc(G)$, where 11 is the white vertex and the other vertices are as in the diagram at the beginning of this subsection:
     \begin{itemize}
        \item If $p  = 2$ then $\pgc(G)$ is \notIsolatedVertexWhite or \szunrecognizeiii.
        \item If $p = 3$ then $\pgc(G)$ is \szunrecognizei or \trianglePlusIsolatedExtra.
         \item If $p = 5$ then $\pgc(G)$ is \szunrecognizeii or \fletchingWhite.
        \item If $p= 11$ then $\pgc(G)$ is \brokenFrameIsolatediii.
    \end{itemize}
\end{lemma}
\begin{proof}
    The proof follows from \thref{GeneralizedModuleExtensionsCited} and the Brauer tables of $\PSL(2, 11)$. The tables below were computed using the Brauer tables of $\PSL(2, 11)$ accessed through GAP \cite{GAP} and \cite[Lemma 6.2]{BrauerTableFP}. They could also  be computed manually using \thref{ConjugacyClassLanding}.

    \begin{center} 
    \mbox{
    \begin{tabular}{||c c c c||} 
     \hline
     $\chi_i \in \operatorname{IBr}_{2}(\PSL(2,11))$ & 3 & 5 & 11 \\ [0.5ex] 
     \hline\hline
     $\chi_1$ & Yes & Yes & Yes \\ 
     \hline
     $\chi_2$ through $\chi_4$  & Yes & Yes & No\\
     \hline
     $\chi_5$ through $\chi_6$ & Yes & Yes & Yes\\
     \hline
    \end{tabular}
    \quad
    \begin{tabular}{||c c c c||} 
     \hline
     $\chi_i \in \operatorname{IBr}_{3}(\PSL(2,11))$ & 2 & 5 & 11\\ [0.5ex] 
     \hline\hline
     $\chi_1$ & Yes & Yes & Yes\\[1ex]
     \hline
     $\chi_2$ through $\chi_4$ & Yes & Yes & No\\[1ex]
     \hline
     $\chi_5$ through $\chi_6$ & Yes & Yes & Yes\\[1ex]
     \hline
    \end{tabular}
    }
    \end{center}
    \begin{center}
    \mbox{
    \begin{tabular}{||c c c c||} 
     \hline
     $\chi_i \in \operatorname{IBr}_{5}(\PSL(2,11))$ & 2 & 3 & 11 \\ [0.5ex] 
     \hline\hline
     $\chi_1$ & Yes & Yes & Yes \\ [1ex] 
     \hline
     $\chi_2$ through $\chi_{5}$ & Yes & Yes & No \\ [1ex] 
     \hline
     $\chi_6$ & Yes & Yes & Yes \\ [1ex] 
     \hline
    \end{tabular}
    \quad
    \begin{tabular}{||c c c c||} 
     \hline
     $\chi_i \in \operatorname{IBr}_{11}(\PSL(2,11))$ & 2 & 3 & 5\\ [0.5ex] 
     \hline\hline
     $\chi_1$ through $\chi_6$ & Yes & Yes & Yes\\[1ex] 
     \hline
    \end{tabular}
    }\vspace{0.2cm}
    \end{center}
\end{proof}
\begin{corollary}\thlabel{cri3satisfiedPSL211}
    There exists a group $T\ltimes P$ with $\pgc(T\ltimes P) = \szunrecognizei$ such that $P$ is a 3-group.
\end{corollary}
\begin{proof}
    This follows from \cite[Theorem 2.2]{Suz} applied to the representation corresponding to $\chi_4\in\operatorname{IBr}_3(\PSL(2, 11))$ (see \thref{psl211brauer}).
\end{proof}
\begin{lemma}\thlabel{nokitespsl211}
    There is no $\PSL(2,11)$-solvable group that realizes \SuzukiThirtyTwoElusiveGraphWhiteBadb where 11 corresponds to the white vertex.
\end{lemma}
\begin{proof}
    For contradiction, suppose there exists a $\PSL(2, 11)$-solvable group $G$ such that \\
    $\pgc(G)[\pi(\PSL(2, 11))]$ is \SuzukiThirtyTwoElusiveGraphWhiteBadb. Because $\pgc(G)[\pi(\PSL(2, 11))]\subseteq\pgc(\PSL(2, 11))$, the degree three vertex must be 5. Then, $G$ must be strictly $\PSL(2, 11)$-solvable, or we would reach a contradiction to the main result of \cite{2015REU}. By \cite[Theorem 4.4]{Suz}, there exists a $H\leq G$ with $\pi(H) = \pi(\PSL(2, 11))$ and $\pgc(H) = \pgc(G)[\pi(H)]$. Then $H\cong N.S$ where $N$ is solvable by \cite[Lemma 2.1.1]{2023REU} and $S$ is such that $\operatorname{Inn}(\PSL(2, 11))\leq S\leq \Aut(\PSL(2, 11))$. First suppose $S\cong\operatorname{Inn}(\PSL(2, 11))\cong\PSL(2, 11)$.  Take a chief series $1 = N_1\unlhd N_2\unlhd\cdots\unlhd N_n\unlhd H$ of $H$ such that $H/N_n\cong\PSL(2, 11)$, guaranteed by \cite[Lemma 5.8]{Suz}. Consider the group $R = H/N_{n-1}$, which is isomorphic to a group $K_p.\PSL(2, 11)$ for some $p$-group with $p\in\pi(\PSL(2, 11))$, also by \cite[Lemma 5.8]{Suz}. If $N$ is nontrivial, $K_p$ must be nontrivial. By \thref{GeneralizedModuleExtensionsCited} and \thref{psl211brauer}, notice that every $q$-group extension for $q \in \{2,3,5,11\}$ removes at least one edge of the form $5-p$ for $p \in \{2,3,11\}$. Therefore, $K_p$ must be trivial, so $H = \PSL(2, 11)$, which contradicts $\pgc(H)\cong\pgc(G)\cong\SuzukiThirtyTwoElusiveGraphWhiteBadb$. Then, we must have $\operatorname{Inn}(\PSL(2, 11)) < S\leq \Aut(\PSL(2, 11))$. Because $[\operatorname{Inn}(\PSL(2, 11)):\Aut(\PSL(2, 11))] = 2$, this means $S \cong \Aut(\PSL(2, 11))$, so $\pgc(G)\cong \pgc(H)\subseteq\pgc(\Aut(\PSL(2, 11)))$. $5$ is a degree 3 vertex in $\pgc(G)$ but not in  $\pgc(\Aut(\PSL(2, 11)))$, so we reach a contradiction in all cases.
\end{proof}

By the prime graph complements given at the beginning of this section, the fact that $\pi(\PSL(2, 11)) = \pi(\Aut(\PSL(2, 11)))$, \thref{psl211fpinfo}, \thref{cri3satisfiedPSL211}, and \thref{nokitespsl211}, the criteria listed at the beginning of \ref{sec:generalT211} are satisfied, so \thref{TPSL211General} with $T = \PSL(2, 11)$ gives a complete classification of the prime graph complements of $\PSL(2, 11)$-solvable groups.
\subsection{$\PSL(2, 19)$}\label{sec:PSL219}
\begin{figure}[H]
    \centering
    \begin{minipage}{.5\textwidth}
        \centering
        \begin{tikzpicture}
            % Define coordinates for the vertices
            \coordinate (A) at (0, 1);   % Vertex 2
            \coordinate (B) at (1, 1);   % Vertex 7
            \coordinate (C) at (1, 0);   % Vertex 13
            \coordinate (D) at (0, 0);   % Vertex 5
            % Draw the edges of the complete graph
            \draw (B) -- (C);
            \draw (C) -- (D);
            \draw (A) -- (D);
            \draw (A) -- (B);
            \draw (A) -- (C);
            % Draw the vertices (circles only)
            \foreach \point in {A, B, C, D}
                \draw[draw, fill=white] (\point) circle [radius=0.2cm];
            % Draw the text on top of the circles
            \foreach \point/\name in {(A)/3, (B)/2, (C)/19, (D)/5}
                \node at \point {\scriptsize\name};
        \end{tikzpicture}
        \caption*{$\pgc(\PSL(2,19))$}
    \end{minipage}%
    \begin{minipage}{.5\textwidth}
        \centering
        \begin{tikzpicture}
            % Define coordinates for the vertices
            \coordinate (A) at (0, 1);   % Vertex 2
            \coordinate (B) at (1, 1);   % Vertex 7
            \coordinate (C) at (1, 0);   % Vertex 13
            \coordinate (D) at (0, 0);   % Vertex 5
            % Draw the edges of the complete graph
            \draw (B) -- (C);
            \draw (C) -- (D);
            \draw (A) -- (D);
            \draw (A) -- (C);
            % Draw the vertices (circles only)
            \foreach \point in {A, B, C, D}
                \draw[draw, fill=white] (\point) circle [radius=0.2cm];
            % Draw the text on top of the circles
            \foreach \point/\name in {(A)/3, (B)/2, (C)/19, (D)/5}
                \node at \point {\scriptsize\name};
        \end{tikzpicture}
        \caption*{$\pgc(\Aut(\PSL(2,19)))$}
    \end{minipage}%
\end{figure}
In this section, we classify the prime graph complements of $\PSL(2, 19)$-solvable groups. The group $\PSL(2, 19)$ has order $3420 = 2^2\cdot3^3\cdot5\cdot19$ and its automorphism group has order $6840 = 2^3\cdot3^3\cdot5\cdot19$. The Schur multiplier of $\PSL(2,19)$ is 2 and $\pi(\PSL(2,19))=\pi(\Aut(\PSL(2,19)))$. Let $d = 19, a = 5, c = 3$. Then \thref{psl219fpinfo} matches Criterion 2 given in \ref{sec:generalT211}.
%generalT211

\begin{fact}\thlabel{psl219fpinfo}
    \begin{tabular}[h]{|c|c|}
    \hline
        Fixed Point Information &  \\
        \hline
        $\PSL(2,19)$ & $[[2,3,5], [2,3,5,19]]$ \\
        \hline 
        $2.\PSL(2,19)$ & $[[2,3,5], [2,3,5,19], [3,5], [3,5,19]]$ \\
        \hline
    \end{tabular}\\
    Each list refers to all prime element orders where elements of an order in a sublist have fixed points in some representation. The multiple lists refer to the fact that different irreducible representations have different fixed points. This list was obtained using \cite{GAP}.
\end{fact}
\begin{lemma}\thlabel{psl219Brauer}
    Let $p \in \pi(\PSL(2, 19))$, let $N$ be a nontrivial $p$-group, and let $G \cong N.\PSL(2, 19)$. Then we have the following restrictions on $\pgc(G)$, where the white vertex denotes 19 and the other vertices correspond to their positions in the figures at the start of this section:
    \begin{itemize}
        \item If $p  = 2$ then $\pgc(G)$ is \notIsolatedVertexWhite or \szunrecognizeiii.
        \item If $p = 3$ then $\pgc(G)$ is \fletchingWhite or \szunrecognizeii.
         \item If $p = 5$ then $\pgc(G)$ is \szunrecognizei or \trianglePlusIsolatedExtra.
        \item If $p= 19$ then $\pgc(G)$ is \brokenFrameIsolatediii.
    \end{itemize}
\end{lemma}
\begin{proof}
    Follows from the Brauer tables of $\PSL(2, 19)$ and \thref{GeneralizedModuleExtensionsCited}. The tables below are computed using \cite[Lemma 6.2]{BrauerTableFP} and \cite{GAP}, but could also be computed manually using \thref{ConjugacyClassLanding}.
    \begin{center} 
    \mbox{
    \begin{tabular}{||c c c c||} 
     \hline
     $\chi_i \in \operatorname{IBr}_{2}(\PSL(2,19))$ & 3 & 5 & 19 \\ [0.5ex] 
     \hline\hline
     $\chi_1$ & Yes & Yes & Yes \\ 
     \hline
     $\chi_2$\text{ through }$\chi_{5}$& Yes & Yes & No \\ 
     \hline
     $\chi_{6}$\text{ through }$\chi_{9}$& Yes & Yes & Yes \\ [1ex]
     \hline
    \end{tabular}
    \quad
    \begin{tabular}{||c c c c||} 
     \hline
     $\chi_i \in \operatorname{IBr}_{3}(\PSL(2,19))$ & 2 & 5 & 19\\ [0.5ex] 
     \hline\hline
     $\chi_1$ & Yes & Yes & Yes\\
     \hline
      $\chi_{2}$\text{ through } $\chi_7$ & Yes & Yes & No\\
     \hline
      $\chi_{8}$ & Yes & Yes & Yes\\[1ex]
     \hline
    \end{tabular}
    }
    \end{center}
    \begin{center}
    \mbox{
    \begin{tabular}{||c c c c||} 
     \hline
     $\chi_i \in \operatorname{IBr}_{5}(\PSL(2,19))$ & 2 & 3 & 19 \\ [0.5ex] 
     \hline\hline
     $\chi_1$& Yes & Yes & Yes \\ 
     \hline
     $\chi_{2}$ through $\chi_4$ & Yes & Yes & No\\
     \hline
     $\chi_{5}$ through $\chi_8$ & Yes & Yes & Yes\\   [1ex] 
     \hline
    \end{tabular}
    \quad
    \begin{tabular}{||c c c c||} 
     \hline
     $\chi_i \in \operatorname{IBr}_{19}(\PSL(2, 19))$ & 2 & 3 & 5 \\ [0.5ex] 
     \hline\hline
     $\chi_1$ \text{ through }$\chi_{10}$ & Yes & Yes & Yes \\ [1ex] 
     \hline
    \end{tabular}
    }\vspace{0.2cm}
    \end{center}
\end{proof}
We now prove that Criterion 3 listed in \ref{sec:generalT211} is satisfied.
\begin{corollary}\thlabel{criterion3PSL219}
   There exists a group $T\ltimes P$ with $\pgc(T\ltimes P)\cong\szunrecognizei$ where $P$ is a 5-group and $T = \PSL(2, 19)$.
\end{corollary}
\begin{proof}
Apply \cite[Theorem 2.2]{Suz} to $\chi_2\in\operatorname{IBr}_5(\PSL(2, 19))$.
\end{proof}

\begin{lemma}\thlabel{SuzukiBadCantExist219}
    There exists no $\PSL(2, 19)$-solvable group $G$ realizing \SuzukiThirtyTwoElusiveGraphWhiteBadb where 19 corresponds to the white vertex.
\end{lemma}
\begin{proof}
    For contradiction, suppose otherwise. Then there exists a $\PSL(2, 19)$-solvable group $G$ such that $\pgc(G)[\pi(\PSL(2, 19))]$ is \SuzukiThirtyTwoElusiveGraphWhiteBadb. If $G$ is solvable we reach an immediate contradiction to the main result of \cite{2015REU}, so assume $G$ is not solvable. By \cite[Lemma 4.4]{Suz}, there exists a $\PSL(2, 19)$-solvable group $H$ such that $\pi(H) = \pi(\PSL(2, 19))$ and $\pgc(H)\cong\pgc(G)$. Thus, we only need to consider the case where $\pgc(H)\cong \SuzukiThirtyTwoElusiveGraphWhiteBadb$. Note that 3 must be the degree 3 vertex for $\pgc(H)$ to be a subgraph of $\pgc(\PSL(2, 19))$. By \cite[Lemma 2.1.1]{2023REU}, there exists a solvable group $N\unlhd H$ such that $H/N$ is isomorphic to a subgroup of $\Aut(\PSL(2, 19))$ that is either $\operatorname{Inn}(\PSL(2, 19))$ or $\Aut(\PSL(2, 19))$ because $|\operatorname{Out}(\PSL(2, 19))| = 2$. In either case, there exists a subgroup $K\unlhd H$ such that $K\cong N.\PSL(2, 19)$. Take a chief series of $K$: $1 = N_1\unlhd N_2\unlhd\cdots\unlhd N_n\unlhd G$ such that $G/N_n\cong\PSL(2, 19)$, guaranteed by \cite[Lemma 5.8]{Suz}. Consider the group $R = G/N_{n-1}$, which is isomorphic to a group $K_p.\PSL(2, 19)$ for some $p$-group with $p\in\pi(\PSL(2, 19))$. Notice that if $K_p = 1, K = 1$. We consider two cases: The case where $K_p \neq 1$ and the case where $K_p = 1$. In the first case, if $p$ is 2, we must be missing the 2-3 edge in $\pgc(K_p.\PSL(2, 19))$ by \thref{psl223Brauer}. This is a contradiction because $\pgc(K_p.\PSL(2, 19))$ is a supergraph of $\pgc(K)$, which is in turn a supergraph of $\pgc(H)\cong\pgc(G)$. If $p$ is 3, we are missing the $3-19$ edge in $\pgc(K_p.\PSL(2, 19))$ by \thref{psl223Brauer} and reach a contradiction in a similar manner. If $p$ is 5, $\pgc(K_p.\PSL(2, 19))$ is missing the $3-5, 2-5 $ edge by \thref{psl223Brauer}, and if $p = 19$ then $\pgc(K_p.\PSL(2, 19))$ is missing the $2-19, 3-19,$ and $5-19$ edges, so in both all subcases of case 1 we reach a contradiction. In case 2, notice $H \cong \PSL(2, 19)$ or $H\cong\Aut(\PSL(2, 19))$ as $|\operatorname{Out}(\PSL(2, 19))| = 2$, which is a contradiction as neither $\PSL(2, 19)$ nor $\Aut(\PSL(2, 19))$ realizes \SuzukiThirtyTwoElusiveGraphWhiteBadb.
\end{proof}
Thus, by \thref{criterion3PSL219}, \thref{SuzukiBadCantExist219}, \thref{psl219fpinfo}, the fact that the Schur multiplier of $\PSL(2, 19)$ is 2, the prime graphs of $\PSL(2, 19)$ and $\Aut(\PSL(2, 19))$ given at the start of this section, and the fact that $\pi(\PSL(2, 19)) = \pi(\Aut(\PSL(2, 19)))$, \thref{TPSL211General} with $T = \PSL(2, 19)$ completely classifies the prime graph complements of $\PSL(2, 19)$-solvable groups.

\subsection{$\PSL(2, 23)$}\label{sec:PSL223new}
\begin{figure}[H]
    \centering
    \begin{minipage}{.5\textwidth}
        \centering
        \begin{tikzpicture}
            % Define coordinates for the vertices
            \coordinate (A) at (0, 1);   % Vertex 2
            \coordinate (B) at (1, 1);   % Vertex 7
            \coordinate (C) at (1, 0);   % Vertex 13
            \coordinate (D) at (0, 0);   % Vertex 5
            % Draw the edges of the complete graph
            \draw (B) -- (C);
            \draw (C) -- (D);
            \draw (A) -- (D);
            \draw (A) -- (B);
            \draw (A) -- (C);
            % Draw the vertices (circles only)
            \foreach \point in {A, B, C, D}
                \draw[draw, fill=white] (\point) circle [radius=0.2cm];
            % Draw the text on top of the circles
            \foreach \point/\name in {(A)/11, (B)/2, (C)/23, (D)/3}
                \node at \point {\scriptsize\name};
        \end{tikzpicture}
        \caption*{$\pgc(\PSL(2,23))$}
    \end{minipage}%
    \begin{minipage}{.5\textwidth}
        \centering
        \begin{tikzpicture}
            % Define coordinates for the vertices
            \coordinate (A) at (0, 1);   % Vertex 2
            \coordinate (B) at (1, 1);   % Vertex 7
            \coordinate (C) at (1, 0);   % Vertex 13
            \coordinate (D) at (0, 0);   % Vertex 5
            % Draw the edges of the complete graph
            \draw (B) -- (C);
            \draw (C) -- (D);
            \draw (A) -- (D);
            \draw (A) -- (C);
            % Draw the vertices (circles only)
            \foreach \point in {A, B, C, D}
                \draw[draw, fill=white] (\point) circle [radius=0.2cm];
            % Draw the text on top of the circles
            \foreach \point/\name in {(A)/11, (B)/2, (C)/23, (D)/3}
                \node at \point {\scriptsize\name};
        \end{tikzpicture}
        \caption*{$\pgc(\Aut(\PSL(2,23)))$}
    \end{minipage}%
\end{figure}
We classify the prime graph complements of $\PSL(2, 23)$-solvable groups in this section. The group $\PSL(2, 23)$ has order $6072 = 2^3 \cdot3\cdot11\cdot23$ and its automorphism group has order $12144 = 2^4\cdot3\cdot11\cdot23$. The Schur multiplier of $\PSL(2,23)$ is 2 and $\pi(\PSL(2,23))=\pi(\Aut(\PSL(2,23)))$ by \cite{ATLAS} and \cite{GAP}. Let $d = 23, a = 11,$ and $c = 3$. Then \thref{psl223fpinfo} matches Criterion 2 given in \ref{sec:generalT211}.
%generalT211

\begin{fact}\thlabel{psl223fpinfo}
    \begin{tabular}[h]{|c|c|}
    \hline
        Fixed Point Information &  \\
        \hline
        $\PSL(2,23)$ & $[[2,3,11], [2,3,11, 23]]$ \\
        \hline 
        $2.\PSL(2,23)$ & $[[2,3,11], [2,3,11,23], [3,11], [3,11,23]]$ \\
        \hline
    \end{tabular}\\
    Each list refers to all prime element orders where elements of an order in a sublist have fixed points in some representation. The multiple lists refer to the fact that different irreducible representations have different fixed points. This list was obtained using \cite{GAP}.
\end{fact}
\begin{lemma}\thlabel{psl223Brauer}
    Let $p \in \pi(\PSL(2, 23))$, let $N$ be a nontrivial $p$-group, and let $G \cong N.\PSL(2, 23)$. Then we have the following restrictions on $\pgc(G)$, where the white vertex denotes 23 and the other vertices are labeled as they are in the diagrams at the beginning of this section:
    \begin{itemize}
        \item If $p  = 2$ then $\pgc(G)$ is \notIsolatedVertexWhite or \szunrecognizeiii.
        \item If $p = 3$ then $\pgc(G)$ is \szunrecognizei or \trianglePlusIsolatedExtra.
         \item If $p = 11$ then $\pgc(G)$ is \szunrecognizeii or \fletchingWhite.
        \item If $p= 23$ then $\pgc(G)$ is \brokenFrameIsolatediii.
    \end{itemize}
\end{lemma}
\begin{proof}
    Follows from the Brauer tables of $\PSL(2, 23)$ and \thref{GeneralizedModuleExtensionsCited}. The tables below are computed using \cite[Lemma 6.2]{BrauerTableFP} and \cite{GAP}, but could also be computed manually using \thref{ConjugacyClassLanding}.
    \begin{center} 
    \mbox{
    \begin{tabular}{||c c c c||} 
     \hline
     $\chi_i \in \operatorname{IBr}_{2}(\PSL(2,23))$ & 3 & 11 & 23 \\ [0.5ex] 
     \hline\hline
     $\chi_1$ & Yes & Yes & Yes \\ 
     \hline
     $\chi_2$\text{ through }$\chi_{4}$& Yes & Yes & No \\ 
     \hline
     $\chi_{5}$\text{ through }$\chi_{9}$& Yes & Yes & Yes \\ [1ex]
     \hline
    \end{tabular}
    \quad
    \begin{tabular}{||c c c c||} 
     \hline
     $\chi_i \in \operatorname{IBr}_{3}(\PSL(2,23))$ & 2 & 11 & 23\\ [0.5ex] 
     \hline\hline
     $\chi_1$ & Yes & Yes & Yes\\
     \hline
      $\chi_{2}$\text{ through } $\chi_5$ & Yes & Yes & No\\
     \hline
      $\chi_{6}$\text{ through }$\chi_{10}$ & Yes & Yes & Yes\\[1ex]
     \hline
    \end{tabular}
    }
    \end{center}
    \begin{center}
    \mbox{
    \begin{tabular}{||c c c c||} 
     \hline
     $\chi_i \in \operatorname{IBr}_{11}(\PSL(2,23))$ & 2 & 3 & 23 \\ [0.5ex] 
     \hline\hline
     $\chi_1$& Yes & Yes & Yes \\ 
     \hline
     $\chi_{2}$ through $\chi_8$ & Yes & Yes & No\\
     \hline
     $\chi_{9}$ & Yes & Yes & Yes\\   [1ex] 
     \hline
    \end{tabular}
    \quad
    \begin{tabular}{||c c c c||} 
     \hline
     $\chi_i \in \operatorname{IBr}_{23}(\PSL(2, 23))$ & 2 & 3 & 11 \\ [0.5ex] 
     \hline\hline
     $\chi_1$ \text{ through }$\chi_{12}$ & Yes & Yes & Yes \\ [1ex] 
     \hline
    \end{tabular}
    }\vspace{0.2cm}
    \end{center}
\end{proof}
We now prove that Criterion 3 listed in \ref{sec:generalT211} is satisfied.
\begin{corollary}\thlabel{criterion3PSL223}
   There exists a group $T\ltimes P$ with $\pgc(T\ltimes P)\cong\szunrecognizei$ where $P$ is a 3-group.
\end{corollary}
\begin{proof}
Apply \cite[Theorem 2.2]{Suz} to $\chi_2\in\operatorname{IBr}_3(\PSL(2, 23))$.
\end{proof}

\begin{lemma}\thlabel{SuzukiBadCantExist223}
    There exists no $\PSL(2, 23)$-solvable group $G$ which realizes \SuzukiThirtyTwoElusiveGraphWhiteBadb where 23 corresponds to the white vertex.
\end{lemma}
\begin{proof}
    For contradiction, suppose otherwise. Then there exists a $\PSL(2, 23)$-solvable group $G$ such that $\pgc(G)[\pi(\PSL(2, 23))]$ is \SuzukiThirtyTwoElusiveGraphWhiteBadb. If $G$ is solvable we reach an immediate contradiction to the main result of \cite{2015REU}, so assume $G$ is not solvable. By \cite[Lemma 4.4]{Suz}, there exists a $\PSL(2, 23)$-solvable group $H$ such that $\pi(H) = \pi(\PSL(2, 23))$ and $\pgc(H)\cong\pgc(G)$. Thus, we only must consider the case where $\pgc(H)\cong \SuzukiThirtyTwoElusiveGraphWhiteBadb$. Note that 11 must be the degree 3 vertex for $\pgc(H)$ to be a subgraph of $\pgc(\PSL(2, 23))$. By \cite[Lemma 2.1.1]{2023REU}, there exists a solvable group $N\unlhd H$ such that $H/N$ is isomorphic to a subgroup of $\Aut(\PSL(2, 23))$ that is either $\operatorname{Inn}(\PSL(2, 23))$ or $\Aut(\PSL(2, 23))$ because $|\operatorname{Out}(\PSL(2, 23))| = 2$. In either case, there exists a subgroup $K\unlhd H$ such that $K\cong N.\PSL(2, 23)$. Take a chief series of $K$ $1 = N_1\unlhd N_2\unlhd\cdots\unlhd N_n\unlhd G$ such that $G/N_n\cong\PSL(2, 23)$, guaranteed by \cite[Lemma 5.8]{Suz}. Consider the group $R = K/N_{n-1}$, which is isomorphic to a group $K_p.\PSL(2, 23)$ for some $p$-group with $p\in\pi(\PSL(2, 23))$. Notice that if $K_p = 1, N = 1$, so consider two cases: The case where $K_p\neq 1$ and the case where $K_p = 1$. In the first case, if $p$ is 2,  $2-11\notin\pgc(K_p.\PSL(2, 23))$ by \thref{psl223Brauer}. This is a contradiction because $\pgc(K_p.\PSL(2, 23))$ is a supergraph of $\pgc(K)$, which is in turn a supergraph of $\pgc(H)\cong\pgc(G)$. If $p$ is 3, we are missing the $3-11$ edge in $\pgc(K_p.\PSL(2, 23))$ by \thref{psl223Brauer} and reach a contradiction in a similar manner. If $p$ is 11, $\pgc(K_p.\PSL(2, 23))$ is missing the $3-11$ and $2-11$ edges by \thref{psl223Brauer}, and if $p = 23$ then $\pgc(K_p.\PSL(2, 23))$ is missing the $23-11$ edge, so in all subcases of the first case we reach a contradiction. In the second case, notice $H \cong \PSL(2, 23)$ or $H\cong\Aut(\PSL(2, 23))$ as $|\operatorname{Out}(\PSL(2, 23))| = 2$, which is a contradiction as neither $\PSL(2, 23)$ nor $\Aut(\PSL(2, 23))$ realizes \SuzukiThirtyTwoElusiveGraphWhiteBadb.
\end{proof}
By \thref{criterion3PSL223}, \thref{SuzukiBadCantExist223}, \thref{psl223fpinfo}, the fact that the Schur multiplier of $\PSL(2, 23)$ is 2, the prime graphs of $\PSL(2, 23)$ and $\Aut(\PSL(2, 23))$ given at the start of this section, and the fact that $\pi(\PSL(2, 23)) = \pi(\Aut(\PSL(2, 23)))$, \thref{TPSL211General} with $T = \PSL(2, 23)$ completely classifies the prime graph complements of $\PSL(2, 23)$-solvable groups.

\section{$T$-solvable groups with $T$ similar to $\PSL(2, 5^2)$ and $\PSL(2, 3^4)$}\label{sec:PSL225}
In this section, we classify the prime graph complements of $T$-solvable groups for $T$ as in \ref{sec:PSL225General}. We then use the main result of \ref{sec:PSL225General} to classify the prime graph complements of $\PSL(2, 5^2)-$ and $\PSL(2, 3^4)$-solvable groups in subsections \ref{sec:PSL225new} and \ref{sec:PSL281new} respectively.
\subsection{General $T$}\label{sec:PSL225General}
\begin{figure}[H]
    \centering
    \begin{minipage}{.5\textwidth}
        \centering
        \begin{tikzpicture}
            % Define coordinates for the vertices
            \coordinate (A) at (0, 1);   % Vertex 2
            \coordinate (B) at (1, 1);   % Vertex 7
            \coordinate (C) at (1, 0);   % Vertex 13
            \coordinate (D) at (0, 0);   % Vertex 5
            % Draw the edges of the complete graph
            \draw (B) -- (C);
            \draw (C) -- (D);
            \draw (A) -- (D);
            \draw (A) -- (B);
            \draw (A) -- (C);
            % Draw the vertices (circles only)
            \foreach \point in {A, B, C, D}
                \draw[draw, fill=white] (\point) circle [radius=0.2cm];
            % Draw the text on top of the circles
            \foreach \point/\name in {(A)/a, (B)/2, (C)/d, (D)/c}
                \node at \point {\scriptsize\name};
        \end{tikzpicture}
        \caption*{$\pgc(T)$}
    \end{minipage}%
    \begin{minipage}{.5\textwidth}
        \centering
        \begin{tikzpicture}
            % Define coordinates for the vertices
            \coordinate (A) at (0, 1);   % Vertex 2
            \coordinate (B) at (1, 1);   % Vertex 7
            \coordinate (C) at (1, 0);   % Vertex 13
            \coordinate (D) at (0, 0);
            % Draw the edges of the complete graph
            \draw (C) -- (D);
            \draw (A) -- (D);
            \draw (A) -- (C);
            % Draw the vertices (circles only)
            \foreach \point in {A, B, C, D}
                \draw[draw, fill=white] (\point) circle [radius=0.2cm];
            % Draw the text on top of the circles
            \foreach \point/\name in {(A)/a, (B)/2, (C)/d, (D)/c}
                \node at \point {\scriptsize\name};
        \end{tikzpicture}
        \caption*{$\pgc(\Aut(T))$}
    \end{minipage}
\end{figure}
% We include the following definition:
% \begin{definition}\thlabel{realizableDefPSL225}
%     Given a set of rooted graphs $\mathcal{H}\subseteq \mathcal{F}$, we say that $\mathcal{H}$ is \emph{realizable} if for each $\Lambda \in \mathcal{H}$, there exists a $T$-solvable $\pi(T)$-group $G$ such that $\pgc(G) \cong \Lambda$.
% \end{definition}
In this section, we prove a classification result for all $T$-solvable groups where $T$ is a finite nonabelian simple group with the associated prime graph complements given above such that $\pi(T) = \pi(\Aut(T))$ and $T$ satisfies the criteria given below: 
\begin{enumerate}
    \item The Schur multiplier of $T$ is 2
    \item The fixed point information for $T$ matches \thref{fpInfoGenPSL225}
    \item The graph \szthirtytwoTrickyGraph is realized by $T\ltimes P$ for some $p$-group $P$ with $p\in\pi(T)$
    \item The Sylow-2 subgroups of $T$ are dihedral.
    \item The Sylow-$a$ subgroups of $T$ do not satisfy the Frobenius criterion
\end{enumerate}
\begin{remark}
     In the above criteria, we only require that \szthirtytwoTrickyGraph be realized by $T\ltimes P$ for some $p$-group $P$ with $p\in\pi(T)$, but in all the cases for which we use the general classification theorem at the end of this section \szthirtytwoTrickyGraph is realized by $T\ltimes P$ for some 2-group $P$.
\end{remark}
Fix some $T$ that has these properties for the remainder of the section. We do not make use of rooted graphs here and instead prove our classification result via an approach that looks at the triangles the induced subgraph corresponding to $\pgc(T)$ contains and the neighborhood of the vertex set of that subgraph, as this approach is more straightforward than the rooted graph approach given the fixed point information in \thref{fpInfoGenPSL225}.
\begin{fact}\thlabel{fpInfoGenPSL225}
 \begin{center}
    \begin{tabular}[h]{|c|c|}
    \hline
        Fixed Point Information &  \\
        \hline
        $2.T$ & $[ [ 2, a, c, d ], [ a, c ], [ a, c, d ] ]$ \\
        \hline
        $T$ & $[ [2, a, c, d]]$ \\
        \hline
    \end{tabular}
     \end{center}
\end{fact}
\begin{proposition}\thlabel{structurepslsubgroups}
    $T\cong \PSL(2, q)$ for some odd $q\geq 5$ and the Sylow-2 subgroups of $2.T$ are generalized quaternion.
\end{proposition}
\begin{proof}
    It follows from \cite[Theorem 2]{psla7} and the fact that the Schur multiplier of $A_7$ is not 2 that $T\cong \PSL(2, q)$ for some odd $q\geq 5$. By \cite[Page 206]{IsaacsFiniteGroups} the Sylow 2-subgroups of $2.T\cong \operatorname{SL}(2, q)$ for odd $q$ are generalized quaternion.
\end{proof}
\begin{proposition}\thlabel{psl225possible}
    All subgraphs of \hamSandwich are realizable by a $T$-solvable group.
\end{proposition}
\begin{proof}
    Any triangle free and 3-colorable graph can be realized by some solvable group via methods in \cite{2015REU}. There are three remaining graphs to check:
    \begin{itemize}
        \item The complete graph minus one edge \hamSandwich is realized by $T$.
        \item The triangle with one edge \szthirtytwoTrickyGraph is realized by Criterion 3.
        \item The triangle plus an isolated vertex \twoIsolatedpsl is realized by $T \times C_2$.
    \end{itemize}
\end{proof}
\begin{lemma}\thlabel{35disconnectpsl225}
    Let $G$ be a strictly $T$-solvable group, then the edges $a-p$ and $c-p \not \in \pgc(G)$ for any $p\in \pi(G)\setminus \pi(T)$.
\end{lemma}
\begin{proof}
    By \thref{fpInfoGenPSL225}, $c$ satisfies \thref{corollaryLemma}. Furthermore, the Sylow $a$-subgroups of $T$ do not satisfy the Frobenius Criterion by Criterion 4, thus satisfying \cite[Proposition 2.2.2]{2023REU}.
\end{proof}

\begin{corollary}\thlabel{psl225135edge}
    For any $T$-solvable group, $G$, where $K \cong N.T$ is the subgroup granted by \thref{lemma212inPaper}, if $d$ divides $|N|$ then $a-d \notin \pgc(K), \pgc(G)$.
\end{corollary}
\begin{proof}
    $N$ is solvable, so by \thref{generalize222} and Criterion 5, $a-d\not\in\pgc(K)$. $K\leq G$, so $\pgc(G)\subseteq\pgc(K)$.

    % For contradiction, assume we have some $T$- solvable group, $G$, where $K \cong N.T$ is the subgroup mentioned in the supposition, such that $d$ divides $|N|$ and $a-d \in \pgc(G)$. We take a solvable subgroup of $K$, $M = N.A$ where $A$ is a Sylow-$a$ subgroup of $T$, then we take a Hall $\{a, d\}$-subgroup $H$ of $M$. Let $R$ be a Sylow $a$-subgroup of $M$, and notice that $d-a\in\pgc(G)[\{d, a\}]\subseteq\pgc(H)$. Notice that $RN \subseteq N.A$, furthermore the order of $RN$ is at least $|R||Y|$ where $Y$ is a Hall $\{a\}'$-subgroup of $N$, so it is at least $|N.A|$, and thus $N.A\subseteq RN$. Thus, $RN = N.A$, so $A\cong RN/N\cong R/(R\cap N)$. Notice that by the definition of a group satisfying the Frobenius criterion, a group satisfying the Frobenius criterion cannot have sections that do not satisfy the Frobenius criterion, so we have reached a contradiction.

    % Notice that this subgroup must be a Frobenius group or 2-Frobenius group because the edge $d-a \in \pgc(M)\subseteq\pgc(G)$. In the case where it is a Frobenius group,  the Sylow $a$-subgroups of $M$ must satisfy the Frobenius criterion, which because $a\neq b$ implies the Sylow $a$-subgroups of $M$ must be cyclic. This is a contradiction as a copy of $A$ is a section of any Sylow $a$-subgroup of $M$, and $A$ is not cyclic. In the case where it is a 2-Frobenius group 
    
    % However, this is a contradiction because this means $A$ is a Frobenius complement, which is a contradiction because it does not satisfy the Frobenius Criterion.

    % [FILL IN CITATION]
\end{proof}

\begin{lemma} \thlabel{psl225triangles}
    Let $G$ be a $T$-solvable group. The only two triangles that can exist in $\pgc(G)$ are $\{2,a,d\}$ and $\{a, c,d\}$.
\end{lemma}
\begin{proof}
     Let $G$ be a $T$-solvable group and let $K \cong N.T$ be the group given by \thref{lemma212inPaper}.
    % Assume that $\pgc(G)$ contains a triangle that has either exactly 0 primes in $\pi(T)$ or exactly 1 prime in $\pi(G)$. 
    By \cite[Corollary 5.7]{Suz}, every triangle of $\pgc(G)$ must have at least two vertices in $\pi(T)$. For contradiction, suppose there existed a triangle in $\pgc(G)$ with precisely 2 vertices in $\pi(T)$. From \thref{35disconnectpsl225} two of the vertices must be $d$ and $2$, so the edge $d-2$ must be contained in the triangle. However, by \thref{noworry2}, if $2-p\in\pgc(G)$, $d-2\not\in\pgc(G)$ for $p\in\pi(G)\setminus\pi(T)$, a contradiction. Thus, any triangle in $\pgc(G)$ is contained in $\pgc(G)[\pi(T)]$. Because $\pgc(G)[\pi(T)]\subseteq\pgc(T)$ and $\{2,a,d\}$ and $\{a, c,d\}$ are the only triangles contained in $\pgc(T)$, if a triangle exists in $\pgc(G)[\pi(T)]$, it must be one of $\{2,a,d\}$ or $\{a, c,d\}$.
\end{proof}

\begin{lemma}\thlabel{13sub2}
    Let $G$ be a $T$-solvable group such that $\pgc(G)$ contains at least one triangle, and there exists an edge $2-q$ for some $q \in \pi(G) \setminus \pi(T)$. Then for all $p \in \pi(G) \setminus \pi(T)$, if $d-p \in \pgc(G)$ then $2-p \in \pgc(G)$. 
\end{lemma}
\begin{proof}
    This follows similarly to \cite[Lemma 3.1.4]{2023REU}. First, notice that since there is a triangle contained in $\pgc(G)$ and all triangles are contained in $\pgc(G)[\pi(T)]$ by \thref{psl225triangles}, we may apply \cite[Lemma 2.1.1]{2023REU} to get $G \cong N.M$ for some solvable group $N$ and $\operatorname{Inn}(T) \leq M \leq \Aut(T)$. 
    %Notice that there are two possible cases for $M$: either $M\cong T$ or $M\cong S$ where $T\cong \operatorname{Inn}(T) < S\leq\Aut(T)$
    % where $S$ is a subgroup of $\Aut(T)$ containing $\operatorname{Inn}(T)$ where every Sylow $b$-subgroup of $S$ is dihedral; if $M$ was not either of these two groups, the Sylow $b$-subgroups of $M$ would not satisfy the Frobenius criterion by Criterion 4, which would contradict $2-q\in\pgc(G)$ by \thref{generalize222} and \thref{lemma212inPaper}. 
    $G\cong N.M$ and $T\cong\operatorname{Inn}(T)$, so take the subgroup $K\cong N.T$. Notice $\pi(K) = \pi(G)$. We may apply \thref{noworry2} to see that $K \cong L.E$ where $L$ is some solvable $2'$-group and $E \cong 2.T$, the perfect central extension.
    
    Notice from \thref{noworry2} that the 2 vertex is isolated from all other primes in $\pgc(K)[\pi(T)]$ therefore, from \thref{psl225triangles} the triangle that must exist is the $\{c,a,d\}$ triangle. Furthermore, $L$ must also be a $d'$-group because we would reach a contradiction from \thref{psl225135edge} on the existence of the $\{a, c,d\}$ triangle if $d$ divides $|L|$. We now take a Hall $\{2, a, c, d,p\}$ subgroup of $G$, which is possible by \cite[Theorem 4.4]{Suz}, and we can use the remark accompanying \cite[Lemma 6.4]{Florez} to say that this subgroup must be of the form $H \cong L_{\{a, c,p\}}.C_2.T$, where $L_{\{a, c,p\}}$ is a Hall $\{a,c,p\}$-group of $L$. Since the edges $d-p, d-c, d-a \in \pgc(G)\subseteq\pgc(K)$, $H$ must contain a subgroup of order $d$ acting Frobeniusly on $L_{\{a,c,p\}}$, therefore $L_{\{a,c,p\}}$ must be nilpotent. Thus $L_{\{a,c,p\}} \cong C\times A \times P$ for Sylow $c$-group $C$, Sylow $a$-group $A$, and Sylow $p$-group $P$. Because $C\times A$ is characteristic $L_{\{a,c,p\}}$ it is normal in $H$. Notice $H/(C \times A) \cong P.E$ and by the Schur-Zassenhaus theorem $P.E \cong P \rtimes E$.
    
    We now prove our result by contrapositive. Assume $2-p \notin \pgc(G)$, then there exists some element $x \in H$ with order $2p$ because $H$ is a Hall $\{2,a, c, d, p\}$-subgroup of $G$. Replacing $x$ with a conjugate of itself if necessary we may assume that $x^p$ modulo $C \times A$ is contained within $E$. Let $1 = K_0\unlhd K_1\unlhd\cdots\unlhd P.E$ be a Chief series of $P.E$. Let $n$ be the least index such that $x^2\not\in K_{n-1}, x^2\in K_n$. Call $K_n = P_0$. It is known that $x^2\in P\unlhd P.E$, so we may assume $P_0\leq P$. Notice $P_0/(K_{n-1})$ is a minimal normal subgroup of $(H/(A\times B))/K_{n-1}$. Thus, it suffices to assume that $P_0$ is a minimal normal subgroup of $H/(A\times B)$ and is thus elementary abelian. 
    % Now let $P_0 \leq P$ be a minimal normal subgroup of $H/(C \times A)$ such that $x^b \in P_0$. Moving $H/(C \times A)$ to a smaller quotient if needed, we may assume that $P_0$ is a minimal normal subgroup of $H/(C \times A)$ and thus is elementary abelian.
    Now $E$ acts on $P_0$ in such a way that $x^p$ fixes $x^2$ (modulo $C \times A$). In other words, we have an order $2$ element of $E$ fixing an order $p$ element of $P_0$. Therefore we may apply \cite[Lemma 2.1.7]{2023REU} and \thref{psl225possible}, to see that if an element of order $2$ has fixed points, an element of order $d$ will have fixed points, thus implying $d-p \notin \pgc(K)$.
    
    Now, we show that if $2-p \in \pgc(K)$ then $2-p \in \pgc(G)$. Recall $G\cong N.M$. Notice that $M\cong T.F$ for some $F\leq \Aut(T)/\operatorname{Inn}(T)$. $G\cong K.F\cong L.2.T.F$ for some $F\leq\Aut(T)/\operatorname{Inn}(T)$. $L$ is a 2'-group, so the Sylow 2-subgroups of $G$ are isomorphic to the Sylow 2-subgroups of $2.T.F$. We claim that these subgroups are generalized quaternion. To see this, let $Q$ be a Sylow 2-subgroup of $2.T.F$. Consider $2.T\cap Q$, a Sylow 2-subgroup of $2.T$. It follows from \thref{structurepslsubgroups} that $2.T\cap Q\leq Q$ is isomorphic to a generalized quaternion group, $Q_{2j}$. By the fact that $2-q\in\pgc(G)$ for $q\in\pi(G)\setminus\pi(T)$ and \thref{generalize222}, $Q$ must satisfy the Frobenius Criterion. By \cite[Theorem 3.1]{konrad}, the only groups satisfying the Frobenius Criterion which have generalized quaternion groups as subgroups are generalized quaternion groups, so $Q\cong Q_{2k}$ for some $k\geq j$.\\
    Suppose to the contrary that $2-p \in \pgc(K)$ but $2-p \notin \pgc(G)$ for some $p \notin \pi(T)$. Then, there exists an element $x$ of order $2p$ in $G$. Take a Hall $\{2,p\}$-subgroup of $G$, $H_G \cong P.Q_{2k}$ where $P$ is some $p$-group. Furthermore, notice a Hall $\{2,p\}$-subgroup of $K$, $H_K \cong P.Q_{2j}$ and $H_K \leq H_G$. Notice via methods similar to the above that $x^p \in Q_{2k}$ fixes $x^2\in P$, in other words we have an order $2$ element of $Q_{2k}$ that fixes an order $p$ element of $P$. But the order $2$ element in $Q_{2k}$ is unique and is contained in the center of this group, so $x^p \in Q_{2j}$ as well. Thus, there exists an element of order $2$ in $Q_{2j}$ fixing an element of order $p$ in $P$. This means that there exists an element of order $2p \in H_K$, so $2-p\not\in\pgc(K)$, a contradiction. Therefore we have shown that if $2-p \in \pgc(K)$ then $2-p \in \pgc(G)$, and thus the previous case is enough to prove the claim.
    \end{proof}
\begin{lemma}\thlabel{psl2253color}
    For any $T$-solvable group $G$, one of the two conditions applies:
    \begin{itemize}
        \item $\pgc(G)$ is triangle free and 3-colorable.
        \item There exists a 3-coloring of $\pgc(G)$ such that all neighbors of $2$ and $d$ in $\pi(G)\setminus\pi(T)$ are the same color.
    \end{itemize}
\end{lemma}
\begin{proof}
    Let $G$ be a $T$-solvable group. If $G$ is solvable then $G$ immediately satisfies the triangle free and 3-colorable condition by \cite{2015REU}. So instead assume that $G$ is strictly $T$-solvable. Recall that only $2$ and $d$ can connect to primes in $\pi(G) \setminus \pi(T)$ via \thref{35disconnectpsl225}, therefore, we break into cases based on the edges between $\pi(T)$ and $\pi(G) \setminus \pi(T)$. 
    
    \begin{itemize}
        \item Assume there exist no edges between 2 or $d$ and an element of $\pi(G) \setminus \pi(T)$. Recall that $G$ is strictly $T$-solvable. Thus, by \cite[Lemma 2.1.1]{2023REU}, there exists a solvable group $N$ and a group $S$ with $\operatorname{Inn}(T)\leq S\leq \Aut(T)$ such that $G\cong N.S$. Notice that $\pgc(G)[\pi(T)]\subseteq\pgc(T)$ which is three-colorable, $\pgc(G)[\pi(G)\setminus\pi(T)]\subseteq\pgc(N)$ which is three-colorable and triangle free by \cite{2015REU}, and $N[\pi(T)] = \pi(T)$, so the second statement holds.
        \item Assume that there exists an edge between exactly one of $2$ or $d$ and a prime in $\pi(G) \setminus \pi(T)$. Notice that $\pgc(T)$ is 3-colorable, and thus all possible subgraphs of $\pgc(T)$ are 3-colorable. Therefore, we may apply \cite[Lemma 2.3.5]{2023REU} to see that the second condition of the statement is satisfied.
        \item Finally, assume that there are edges between both $2$ and $d$ and primes in $\pi(G) \setminus \pi(T)$. If there exist $p, q\in \pi(T)$ such that $2-p, d-q\in\pgc(G)$ then by \thref{13sub2} $d-q, 2-q\in\pgc(G)$. Because of this, we may apply the proof of \cite[Lemma 2.3.5]{2023REU} to color $\pgc(G)$ such that all neighbors of $\pi(T)$ are colored with one color.
    \end{itemize}
    We have checked all possible cases, therefore we have shown that an arbitrary $T$-solvable group $G$ satisfies the claim.

\end{proof}

\begin{theorem}\thlabel{PSL225Classification}
Given a graph $\Xi$, we have that $\Xi$ is isomorphic to the prime graph complement of some $T$-solvable group if and only if one of the following is true: 
\begin{enumerate}
    \item $\Xi$ is triangle-free and 3-colorable.
    \item $\Xi$ contains a subset $X = \{w, x, y, z\} \subseteq V(\Xi)$ such that $N(X) = X$, and $\Xi[X]\cong\twoIsolatedpsl$, $\Xi[X]\cong\hamSandwich$, or $\Xi[X] \cong \szthirtytwoTrickyGraph$. Furthermore, $\Xi\setminus X$ is triangle-free and three colorable.
    \item $\Xi$ contains a subset $X = \{w, x, y, z\} \subseteq V(\Xi)$ such that in some 3-coloring of $\Xi$ the closed neighborhood $N(X) \setminus X$ share one color. Also, $\Xi$ contains exactly one triangle $\{x,y,z\}$ and $N(w) \cap \{x,y,z\} = \varnothing$. All edges incident to $x \text{ and } y$ are within the triangle and $N(z)\setminus\{x,y\} \subseteq N(w)$.
    \item $\Xi$ contains a subset $X = \{w, x, y, z\} \subseteq V(\Xi)$ such that in some 3-coloring of $\Xi$ the closed neighborhood $N(X) \setminus X$ share one color. Also, $\Xi$ contains exactly one triangle $\{x,y,z\}$ and $N(w) \cap \{x,y,z\} = \varnothing$. All edges incident to $w,x,y$ are contained in $\Xi[X]$.
\end{enumerate}
\end{theorem}
\begin{proof}
    We will start by proving the forward direction. Let $G$ be a $T$-solvable group. First, from \thref{psl2253color} $\pgc(G)$ must be 3-colorable. Therefore, if there are no triangles $\pgc(G)$ will satisfy condition 1. Assume $\pgc(G)$ contains at least one triangle. We consider three cases: The case where there exist no $r\in\pi(T)$ and $p\in\pi(G)\setminus\pi(T)$ with $r-p\in\pgc(G)$, the case where there exists $p\in\pi(G)\setminus\pi(T)$ with $2-p\in\pi(G)$, and the case where there exist $p\in\pi(G)\setminus\pi(T)$ with $d-p\in\pgc(G)$ but no $q\in\pi(G)\setminus\pi(T)$ with $2-q\in\pgc(G)$.
    
    Case 1: Choosing $X = \pi(T)$, (2) is satisfied by \thref{psl225triangles}, \thref{psl2253color} and the fact that $\pgc(G)$ contains a triangle.

    Case 2: Notice that $\{a, c, d\}$ forms a triangle such that $\{a, c, d\}\cap N(2) = \varnothing$ because $\pgc(G)[\pi(T)]$ contains a triangle by \thref{psl225triangles} and \thref{noworry2}. Let $w = 2, x = a, y = c, d = z$. $N(z)\setminus\{x, y\}\subseteq N(w)$ by \thref{35disconnectpsl225}. By \thref{psl2253color}, $N(X)\setminus X$ shares one color in some 3-coloring of $\pgc(G)$. Thus, (3) is satisfied.
    
    Case 3: Let $w = 2, x = a, y = c, d = z$. By assumption, all edges incident to $\{w, x, y\}$ are contained in $\pgc(G)[\pi(T)]$. By \thref{psl2253color}, $N(X)\setminus X$ shares one color in some 3-coloring of $\pgc(G)$. It only remains to show that $\{x, y, z\}$ is a triangle and $N(w)\cap\{x, y, z\}=\varnothing$. By assumption, there exists an odd prime $q\in\pi(G)\setminus\pi(T)$ such that $d-q\in\pgc(G)$. Consider a Hall $\{2,a, c, d, q\}$-subgroup of $G$, call it $H$ such that $\pgc(H)\cong\pgc(G)[\pi(T)\cup\{q\}]$ by \cite[Theorem 4.4]{Suz}. By \cite[Proposition 2.2.5]{2023REU}, there exists a section of $H$ of the form $V.E$ where $E$ is a perfect central extension of $T$ and $V$ is a nontrivial elementary abelian $q$-group. Then, we may write $E\cong B.T$. If $E\cong T$, we reach a contradiction by \thref{fpInfoGenPSL225}, \cite[Lemma 2.1.7]{2023REU}, and the fact that $d-q\in\pgc(G)$ for some $q\in\pi(G)\setminus\pi(T)$. We first consider the case where $B.T$ splits, and show that in all cases, we either derive a contradiction or satisfy a case. If $|B|$ is not coprime to $|T|$, notice $B\leq Z(E)$, so $B$ is abelian and thus nilpotent. Then, we may consider a section $P.T$ of $B.T$ for some nontrivial elementary abelian $p$-group $P$ where $p\in\{2,a, c, d\}$. If $p\in\{a, c, d\}$, notice that $P.T$ is still a split central extension of $T$, so because there exists an element of order $p$ in $Z(P.T)$ and $\pgc(G)[\pi(T)\cup\{q\}]\subseteq\pgc(P.T)$, $p-t\not\in\pgc(G)$ for any $t\in\pi(T)$, a contradiction to the triangle assumption if $p\neq 2$. If $p = 2$, then because there exists an element of order 2 in the center of $B.T$, $2-r\not\in\pgc(B.T)$ for $r\in\pi(T)$, so $N(2)\cap \{a, c, d\} = \varnothing$. Because there exists a triangle in $\pgc(G)$, by \thref{psl225triangles}, $\{a, c, d\}$ is a triangle. Thus, (4) is satisfied. This concludes the forward direction.
    
    For the backwards direction, if a graph satisfies (1), then by \cite{2015REU} there exists a solvable group that realizes it. Therefore, we only consider graphs that satisfy (2) through (4).
    
    Let $\Xi$ be a graph satisfying (2). By \thref{psl225possible}, there exists a $T$-solvable group, say $E$, such that $\pgc(E) \cong \Xi[X]$. Using methods from \cite{2015REU}, there exists a solvable group $N$ with order coprime to $|E|$ such that $\pgc(N) \cong \Xi \setminus X$. Notice that $\pgc(E \times N) \cong \Xi$ and that $E \times N$ is a $T$-solvable group. Thus $\Xi$ is the prime graph complement of some $T$-solvable group.
    
    Let $\Xi$ be a graph satisfying (3) and let $X = \{w, x, y, z\}$. Let $E \cong 2.T$ the perfect central extension. Define a graph isomorphism between $\pgc(E)$ and $\Xi[X]$ with the assignments $w\rightarrow 2$, $x \rightarrow c$, $y \rightarrow a$, and $z \rightarrow d$. Now, for each $v \in N(X) \setminus X$ one of the following holds: (1) $v$ is adjacent to $w$ but not $x$, $y$, and $z$, (2) $v$ is adjacent to $w$ and $z$ but not $x$ and $y$. Now refer to the table of fixed points of complex irreducible representations of $E$ in \thref{psl225possible}. If $v$ satisfies (1) select the representation where $a, c$ and $d$ have fixed points but $2$ does not. If $v$ satisfies (2) select the representation where $c$ and $a$ have fixed points but $2$ and $d$ do not. These representations satisfy the conditions for \thref{236Improved}, thus $\Xi$ is the prime graph complement of some $T$-solvable group.
    
    Finally consider the case where $\Xi$ satisfies (4). Take $\Xi'$ to be the graph with all the edges of $\Xi$ and such that if $z-v\in \Xi$ for $v \notin X$ then $w-v\in\Xi'$. Now $\Xi'$ satisfies (3). Therefore, there exists some $T$-solvable group $G$ such that $\pgc(G) \cong \Xi'$. Now take $\pgc(G \times C_2)$ which by construction is isomorphic to $\Xi$. Thus $\Xi$ is the prime graph complement of some $T$-solvable group. 
\end{proof}
\subsection{$\PSL(2, 5^2)$}\label{sec:PSL225new}
\begin{figure}[H]
    \centering
    \begin{minipage}{.5\textwidth}
        \centering
        \begin{tikzpicture}
            % Define coordinates for the vertices
            \coordinate (A) at (0, 1);   % Vertex 2
            \coordinate (B) at (1, 1);   % Vertex 7
            \coordinate (C) at (1, 0);   % Vertex 13
            \coordinate (D) at (0, 0);   % Vertex 5
            % Draw the edges of the complete graph
            \draw (B) -- (C);
            \draw (C) -- (D);
            \draw (A) -- (D);
            \draw (A) -- (B);
            \draw (A) -- (C);
            % Draw the vertices (circles only)
            \foreach \point in {A, B, C, D}
                \draw[draw, fill=white] (\point) circle [radius=0.2cm];
            % Draw the text on top of the circles
            \foreach \point/\name in {(A)/5, (B)/2, (C)/13, (D)/3}
                \node at \point {\scriptsize\name};
        \end{tikzpicture}
        \caption*{$\pgc(\PSL(2,5^2))$}
    \end{minipage}%
    \begin{minipage}{.5\textwidth}
        \centering
        \begin{tikzpicture}
            % Define coordinates for the vertices
            \coordinate (A) at (0, 1);   % Vertex 2
            \coordinate (B) at (1, 1);   % Vertex 7
            \coordinate (C) at (1, 0);   % Vertex 13
            \coordinate (D) at (0, 0);
            % Draw the edges of the complete graph
            \draw (C) -- (D);
            \draw (A) -- (D);
            \draw (A) -- (C);
            % Draw the vertices (circles only)
            \foreach \point in {A, B, C, D}
                \draw[draw, fill=white] (\point) circle [radius=0.2cm];
            % Draw the text on top of the circles
            \foreach \point/\name in {(A)/5, (B)/2, (C)/13, (D)/3}
                \node at \point {\scriptsize\name};
        \end{tikzpicture}
        \caption*{$\pgc(\Aut(\PSL(2,5^2)))$}
    \end{minipage}
\end{figure}
$\PSL(2, 5^2)$ is a group of order $7800 = 2^3\cdot3\cdot5^2\cdot13$ and $\left|\Aut\left(\PSL(2, 5^2)\right)\right| = 31200 = 2^5\cdot3\cdot5^2\cdot13$. $\pi\left(\PSL(2, 5^2)\right) = \pi\left(\Aut\left(\PSL(2, 5^2)\right)\right)$, the Schur multiplier of $T$ is 2, the fixed point information for $\PSL(2, 5^2)$ is given in \thref{psl225fpinfonew}, the graph \SuzukiThirtyTwoElusiveGraph is realized via $\PSL(2,5^2) \ltimes_{\varphi} (\F_2)^{12} $, where $\varphi$ can be found in irreducible representations (see Appendix A for link to GitHub), the Sylow 5-subgroups of $T$ do not satisfy the Frobenius criterion, and it can be checked using \cite{GAP} that the Sylow $2$-subgroups of $T$ are dihedral. Making the assignments $a = 5, c = 3$, and $d = 13$, we may apply \thref{PSL225Classification} to get the classification result for $\PSL(2, 5^2)$-solvable groups.
\begin{fact}\thlabel{psl225fpinfonew}
        \begin{center}
    \begin{tabular}[h]{|c|c|}
    \hline
        Fixed Point Information &  \\
        \hline
        $2.\PSL(2,5^2)$ & $[ [ 2, 3, 5, 13 ], [ 3, 5 ], [ 3, 5, 13 ] ]$ \\
        \hline
        $\PSL(2,5^2)$ & $[ [ 2, 3, 5, 13 ]]$ \\
        \hline
    \end{tabular}
     \end{center}
\end{fact}

\subsection{$\PSL(2,3^4)$}\label{sec:PSL281new}
\begin{figure}[H]
    \centering
    \begin{minipage}{.5\textwidth}
        \centering
        \begin{tikzpicture}
            % Define coordinates for the vertices
            \coordinate (A) at (0, 1);   % Vertex 2
            \coordinate (B) at (1, 1);   % Vertex 7
            \coordinate (C) at (1, 0);   % Vertex 13
            \coordinate (D) at (0, 0);   % Vertex 5
            % Draw the edges of the complete graph
            \draw (B) -- (C);
            \draw (C) -- (D);
            \draw (A) -- (D);
            \draw (A) -- (B);
            \draw (A) -- (C);
            % Draw the vertices (circles only)
            \foreach \point in {A, B, C, D}
                \draw[draw, fill=white] (\point) circle [radius=0.2cm];
            % Draw the text on top of the circles
            \foreach \point/\name in {(A)/3, (B)/2, (C)/41, (D)/5}
                \node at \point {\scriptsize\name};
        \end{tikzpicture}
        \caption*{$\pgc(\PSL(2,3^4))$}
    \end{minipage}%
    \begin{minipage}{.5\textwidth}
        \centering
        \begin{tikzpicture}
            % Define coordinates for the vertices
            \coordinate (A) at (0, 1);   % Vertex 2
            \coordinate (B) at (1, 1);   % Vertex 7
            \coordinate (C) at (1, 0);   % Vertex 13
            \coordinate (D) at (0, 0);
            % Draw the edges of the complete graph
            \draw (C) -- (D);
            \draw (A) -- (D);
            \draw (A) -- (C);
            % Draw the vertices (circles only)
            \foreach \point in {A, B, C, D}
                \draw[draw, fill=white] (\point) circle [radius=0.2cm];
            % Draw the text on top of the circles
            \foreach \point/\name in {(A)/3, (B)/2, (C)/41, (D)/5}
                \node at \point {\scriptsize\name};
        \end{tikzpicture}
        \caption*{$\pgc(\Aut(\PSL(2,3^4)))$}
    \end{minipage}
\end{figure}
$\PSL(2, 3^4)$ is a group of order $265680 = 2^4\cdot3^4\cdot5\cdot41$ and $\left|\Aut\left(\PSL(2, 3^4)\right)\right| = 2125440 = 2^7\cdot3^4\cdot5\cdot41$. $\pi(\PSL(2, 3^4)) = \pi(\Aut(\PSL(2, 3^4)))$, the Schur multiplier of $T$ is 2, the fixed point information for $\PSL(2, 3^4)$ is given in \thref{psl281fpinfo}, the graph \szthirtytwoTrickyGraph is realizable by a $\PSL(2, 3^4)$-solvable group by \thref{possibleGraphPSL281}, the Sylow 3-subgroups of $T$ do not satisfy the Frobenius criterion, and it can be checked using \cite{GAP} that the Sylow $2$-subgroups of $T$ are dihedral. Making the assignments $a = 3, c = 5$, and $d = 41$, we may apply \thref{PSL225Classification} to get the classification result for $\PSL(2, 3^4)$-solvable groups.

\begin{fact}\thlabel{psl281fpinfo}
    \begin{tabular}[h]{|c|c|}
    \hline
        Fixed Point Information &  \\
        \hline
        $\PSL(2,3^4)$ & $[ [ 2, 3, 5, 41 ] ]$ \\
        \hline
        $2.\PSL(2,3^4)$ & $[ [ 2, 3, 5, 41 ], [ 3, 5 ], [ 3, 5, 41 ] ]$ \\
        \hline
    \end{tabular}\\
    The list refers to all element orders where elements have fixed points in some representation. Then the multiple lists refers to the fact that different irreducible representations have different fixed points.
\end{fact}

\begin{lemma}\thlabel{possibleGraphPSL281}
    The graph \szthirtytwoTrickyGraph is realizable by a $\PSL(2, 3^4)$-solvable group.
\end{lemma}
\begin{proof}
    $\chi_2\in\operatorname{IBr}_2(\PSL(2, 3^4))$ has the following fixed point information, where ``Yes" denotes that an element of the order listed in the column fixes an element of order 2 in the corresponding representation:
    \begin{tabular}[h]{|c|c|c|c|}
    \hline
         & 3 & 5 & 41 \\
        \hline
        $\chi_2$ & Yes & Yes & No \\
        \hline
    \end{tabular}\\
    By the proof of \cite[Theorem 2.2]{Suz}, there exists a $\PSL(2, 3^4)$-solvable group $G = \PSL(2, 3^4)\ltimes P$ where $P$ is a 2-group such that $G$ realizes \szthirtytwoTrickyGraph.
\end{proof}

\section{$\PSL(2, 2^f)$ for Primes $f\geq 5$}\label{sec:GeneralPSL}
In this section, we prove general results about the family of $\PSL(2, 2^f)$-solvable groups where $f\geq 5$ is prime and $|\PSL(2, 2^f)|$ has exactly four prime divisors. Note $|\PSL(2, 2^f)| = 2^f(2^{2f}-1)$ by \cite{FSGBook}. Throughout this section, we prove results on the edges between $\pi(G)$ and $\pi(G)\setminus\pi(\PSL(2, 2^f))$ for a $\PSL(2, 2^f)$-solvable group $G$, results on the structure of prime graph complements of $\PSL(2, 2^f)$-solvable groups, and results on the realizability of certain four- and five- vertex graphs as the complements of $\PSL(2, 2^f)$-solvable groups.

\subsection{On Edges Between $\pi(T)$ and Other Primes}
In this subsection, we prove some basic number-theoretic results, results on the subgroup structure of $\PSL(2, 2^f)$, and results relating to edges of the form $t-s$ for $t\in\pi(\PSL(2, 2^f))$ and $s\in\pi(G)\setminus\pi(\PSL(2, 2^f))$ for some $\PSL(2, 2^f)$-solvable group $G$.
\begin{lemma}\thlabel{coprimeOfF}
    There exist no primes $f > 3$ such that $f$ divides $|\PSL(2, 2^f)| = 2^f(2^{2f}-1)$.
\end{lemma}
\begin{proof}
    Let $f$ be a prime number such that $f>3$. Since $2^f$ is a power of 2 for all $f$, $2^f$ is not divisible by $f$ for any $f\geq 3$. Furthermore, by Fermat's Little Theorem, since $4 \nmid f$ for all $f >3$, we have $4^f \equiv 4$ mod $f$, hence $2^{2f} - 1 = 4^f - 1 \equiv 3$ mod $f$, so $2^{2f}-1 \not\equiv 0$ mod $f$ for all primes $f > 3$. Hence, there are no primes $f>3$ such that $f$ divides $2^f$ or $2^{2f}-1$, therefore there are no primes $f>3$ such that $f$ divides $2^f(2^{2f}-1) = |\PSL(2, 2^f)|$.
\end{proof}
\begin{fact}\thlabel{primesOfPSL}
    $|\PSL(2, 2^f)| = 2^f\cdot3\cdot p\cdot r$ for $p = \frac{2^f + 1}{3}, q= 2^f - 1$, $p$ and $q$ are prime.
\end{fact}
\begin{proof}
    This follows from \cite[Theorem 2]{K4Simple} and the fact that $2^f$ divides $|\PSL(2, 2^f)|$ but $2^{f + 1}$ does not.
\end{proof}
Throughout the rest of this section, we will refer to the prime $p$ such that $3p = 2^f + 1$ as $p$ and the prime $q = 2^f - 1$ as $q$.
\begin{lemma}\thlabel{oneEdgeMissingPSL}
    The edge $3 - p\not\in\pgc(\PSL(2, 2^f))$ for the prime $p$ such that $3p = 2^{f} + 1$.
\end{lemma}
\begin{proof}
    By \thref{primesOfPSL}, $3p = 2^f + 1$. By \cite[Theorem 2.1]{MaximalSubgroupPSL}, $D_{2(2^f + 1)}\leq \PSL(2, 2^f)$. Notice $D_{2(2^f + 1)}$ is a Frobenius group with complement $C_2$ and kernel $K$. Because $2\nmid3$ and $2\nmid p$, we must have $3, p$ divides $|K|$. Because the Frobenius kernel of a Frobenius group is nilpotent, there must exist an element of order $3p$ in $K$, so there exists an element of order $3p$ in $\PSL(2, 2^f)$.
\end{proof}

\begin{lemma}\thlabel{PSLGeneral2Sylow}
    All Sylow $2$-subgroups of $\PSL(2,r)$ for $r= 2^f$ are elementary abelian.
\end{lemma}
\begin{proof}
 Suppose $P = \PSL(2,r)$ where $r = 2^f$ for some arbitrary prime $f \geq 5$. Then, by the order formula for $\PSL$ groups given in \cite{FSGBook} and the order formula for special linear groups given in \cite[Lemma 6.2]{Huppert}, $|\PSL(2, r)| = \frac{\Pi_{i=0}^1 (r^2 - r^i)}{(r - 1)(2, r-1)} = \frac{(r^2 - r)(r^2 - 1)}{r - 1} = \frac{r(r - 1)(r^2 - 1)}{r - 1} = r(r^2 - 1) = 2^f(2^{2f} - 1) = |\operatorname{SL}(2, r)|$. Because $\PSL(2, r)$ is defined as a quotient of $\operatorname{SL}(2, r)$ and $|\PSL(2, r)| = |\operatorname{SL}(2, r)|$, $\PSL(2, r)\cong\operatorname{SL}(2, r)$. Then, by \cite[Theorem 8.10]{Huppert}, the Sylow 2-subgroups of $\operatorname{SL}(2, r)$ (and thus $\PSL(2, r)$) are elementary abelian.

\end{proof}
\begin{theorem}\thlabel{generalDisconnect}
    Let $G$ be a strictly $\PSL(2, 2^f)$-solvable group. Then, there exist no edges between $z\in\pi(\PSL(2, 2^f))$ and $n\in\pi(G)\setminus\pi(\Aut(\PSL(2, 2^f)))$.
\end{theorem}
\begin{proof}
    Let the two odd primes dividing the order of $\PSL(2, 2^f)$ that are not 3 be called $p$ and $q$. 
    % Note that our notation here deviates from the one in Fact \ref{primesOfPSL} for reasons that become obvious below. 
    By \cite[Theorem 2]{K4Simple}, we have $2^f - 1 = q$ and $2^f+1 = 3p$. The Schur multiplier of $\PSL(2, 2^f)$ is trivial by \cite{Huppert}. In this proof, we will reference notation and variables defined in \cite{PSLGenericTables}, namely, $S, R, r$ and $t$.
    
    For a generator of a cyclic group of order $q=2^f - 1$, call it $R$, there are $q - 1$ elements of order $q$. By \cite{PSLGenericTables}, there are $\frac{q - 1}{2}$ conjugacy classes of elements of order $q$. Note that the Sylow $q$-groups of $G$ are cyclic 
    %by \cite[Theorem 8.10, Section II]{Huppert}, 
    of order $q$. Let $\chi_i$ refer to the characters given in the $i$th column of the character table on \cite[Page 403]{PSLGenericTables} for $i=1,2,3,4$.
    Then, let some $g\in G$ be of order $q$. We may apply \thref{ConjugacyClassLanding} and the formula $m_i = \frac{1}{\operatorname{o}(g)}\sum_{k=1}^{\operatorname{o}(g)} \chi_i(g^k)$ for the dimension of the fixed point space of $g$ and get the following.
    \begin{align*}
        m_1 =  1\\
        m_2 =  \frac{1}{q}(2^f + q - 1) = 2\\
        m_3 =  \frac{1}{q}(2^f + 1 + 2r + r^2 + ... + r^{q - 1}).
    \end{align*}
    Because $r$ is a non-1 $q$th root of unity, we get:
    \begin{align*}
        m_3 = \frac{1}{q}(2^f + 1 - 2) = \frac{q}{q} = 1.\\
        m_4 =  \frac{1}{q}(2^f - 1) = q.
    \end{align*}
    %Each $m_i$ corresponds to the formula $ \frac{1}{\operatorname{o}(g)}\sum_{k=1}^{\operatorname{o}(g)} \chi(g^k)$ applied with some character $\chi$ in 
    
    Note that $m_1, m_2, m_3, m_4 > 0$. By \thref{corollaryLemma}, this means that for a $\PSL(2, 2^f)$-solvable group $H$ and $x\in\pi(H)\setminus\pi(\Aut(\PSL(2, 2^f)))$, there are no edges of the form $q-x$ in $\pgc(H)$.\\
    We have $2^f + 1 = 3p$. Then, the only powers of $S$ with order 3 are $p$ and $2p$. Because $S^{-p} = S^{2p}$, where $S$ is as described in \cite{PSLGenericTables}, there is one conjugacy class for elements of order 3. We may write, where $t$ is defined in \cite{PSLGenericTables}:
    \begin{align*}
        m_1 =  1\\
        m_2 =  \frac{1}{3}(2^f - 2) = \frac{1}{3}(2^f + 1 - 3) = p - 1\\
        m_3 =  \frac{1}{3}(2^f + 1) = p\\
        m_4 =  \frac{1}{3}(2^f - 1 - 2(t^{p} + (t^{p})^2)) = \frac{1}{3}(2^f + 1) = p
    \end{align*}
    Because $t^{p}$ is a non-1 third root of unity. We can see $m_1, m_2, m_3, m_4 > 0$, so by \thref{corollaryLemma}, for a $\PSL(2, 2^f)$-solvable group $H$ and $x\in\pi(H)\setminus\pi(\Aut(\PSL(2, 2^f)))$ there are no edges of the form $3-x$ in $\pgc(H)$. We now consider $p$. There are $p - 1$ powers of $S$ with order 3, $3p^{d - 1}, 2(3p^{d-1}),..., (p-1)(3p^{d-1})$, so there are $\frac{p - 1}{2}$ conjugacy classes of elements of order $p$ by \cite{PSLGenericTables}. Now let $g\in G$ be of order $p$. By \thref{ConjugacyClassLanding}, there are two non-identity elements of $\langle g\rangle$ in each class. Then, we may write:
   \begin{align*}
       m_1 =  1\\
       m_2 =  \frac{1}{p}(2^f - (p - 1)) = \frac{1}{p}(2^f + 1 - p) = \frac{1}{p}(3p - p) = 3 - 1 = 2\\
       m_3 =  \frac{1}{p}(2^f + 1) = 3\\
       m_4 =  \frac{1}{p}(2^f - 1 - 2(t^{3} + (t^{3})^2 + ... + (t^{3})^{p - 1})) = \frac{1}{p}(2^f - 1 + 2) = 3
   \end{align*}
   Because $t^{3}$ is a non-1 $p$th root of unity. Thus, by \thref{corollaryLemma}, for a $\PSL(2, 2^f)$-solvable group $H$ and $x\in\pi(H)\setminus\pi(\Aut(\PSL(2, 2^f)))$ there are no edges of the form $p-x$ in $\pgc(H)$. By \thref{PSLGeneral2Sylow}, the Sylow 2-subgroups of $\PSL(2, 2^f)$ are elementary abelian, so they do not satisfy the Frobenius criterion. By \cite[Proposition 2.2.2]{2023REU}, for a $\PSL(2, 2^f)$-solvable group $H$ and $x\in\pi(H)\setminus\pi(\PSL(2, 2^f))\subseteq \pi(H)\setminus\pi(\Aut(\PSL(2, 2^f)))$ there are no edges of the form $2-x$ in $\pgc(H)$. 
\end{proof}
\subsection{Graphs on Four Vertices}
In this section we present some results on the structure of $\pgc(\PSL(2, 2^f))$.
\begin{proposition}\thlabel{generalPSLPrimeGraphStructure}
    Let $r$ be such that $\PSL(2, r)$ is a $K_4$-group. Then $\pgc(\PSL(2, r))\cong \hamSandwich$.
\end{proposition}
\begin{proof}
    By \cite[Remark 2.1]{pslGraphStructure}, $\Gamma(\PSL(2, r))$ has three connected components which are cliques. Clearly, $\Gamma(\PSL(2, r))$ does not contain a 3-clique or a 4-clique. We claim $\Gamma(\PSL(2, r))$ cannot contain two 2-cliques. For contradiction, suppose otherwise. Then, because $\PSL(2, r)$ is a $K_4$-group, each vertex of $\Gamma(\PSL(2, r))$ is contained in one of two connected components, a contradiction to the claim that $\Gamma(\PSL(2, r))$ has three connected components. There must exist an edge in $\Gamma(\PSL(2, r))$ or else $\Gamma(\PSL(2, r))$ would have four connected components, so $\Gamma(\PSL(2, r))$ contains exactly one edge. Thus, $\pgc(\PSL(2, r))\cong \hamSandwich$.
\end{proof}
\begin{proposition}\thlabel{PSLHAM}
    For any prime $f\geq 3$, $\pgc(\PSL(2, 2^f))$ is given by:
    \begin{figure}[H]
    \centering
    \begin{minipage}{.5\textwidth}
        \centering
        \begin{tikzpicture}
            % Define coordinates for the vertices
            \coordinate (A) at (0, 1);   % Vertex 2
            \coordinate (D) at (1, 1);   % Vertex 7
            \coordinate (C) at (1, 0);   % Vertex 13
            \coordinate (B) at (0, 0);   % Vertex 5
            % Draw the edges of the complete graph
            \draw (B) -- (C);
            \draw (C) -- (D);
            \draw (A) -- (D);
            \draw (A) -- (B);
            \draw (A) -- (C);
            % Draw the vertices (circles only)
            \foreach \point in {A, B, C, D}
                \draw[draw, fill=white] (\point) circle [radius=0.2cm];
            % Draw the text on top of the circles
            \foreach \point/\name in {(A)/2, (B)/3, (C)/$q$, (D)/$p$}
                \node at \point {\scriptsize\name};
        \end{tikzpicture}
        \caption*{$\pgc(\PSL(2,2^f))$}
    \end{minipage}%
\end{figure}
\end{proposition}
\begin{proof}
By \thref{generalPSLPrimeGraphStructure}, $\pgc(\PSL(2, 2^f))\cong \hamSandwich$. By \thref{oneEdgeMissingPSL}, $3-p\not\in\pgc(\PSL(2, 2^f))$.
\end{proof}
\begin{fact}
    The following graphs containing triangles are realizable by $\PSL(2, 2^f)$-solvable groups:
    \begin{itemize}
        \item[-] $\hamSandwich$ is realized by $\PSL(2, 2^f)$.
        \item[-] $\trianglePlusIsolated$ is realized by $\PSL(2, 2^f)\times C_3$.
    \end{itemize}
\end{fact}

\subsection{Graphs on Five Vertices}
In this subsection, we prove results on the structure of $\pgc(\Aut(\PSL(2, 2^f)))$, results on the structure of the prime graph complements of $\PSL(2, 2^f)$-solvable groups, and results on the realizability of certain rooted five-vertex graphs.
\begin{theorem}\thlabel{Subgraph}
     $\pgc(\Aut(\PSL(2, 2^f)))$ is given by the figure below:
\end{theorem}
\begin{figure}[H]
    \centering
    \begin{minipage}{.5\textwidth}
        \centering
        \begin{tikzpicture}
            % Define coordinates for the vertices
            \coordinate (A) at (0, 1);   % Vertex 2
            \coordinate (D) at (1, 1);   % Vertex 7
            \coordinate (C) at (1, 0);   % Vertex 13
            \coordinate (B) at (0, 0);   % Vertex 5
            \coordinate (E) at (1.5, 0.5);
            % Draw the edges of the complete graph
            \draw (B) -- (C);
            \draw (C) -- (D);
            \draw (A) -- (D);
            \draw (A) -- (B);
            \draw (A) -- (C);
            \draw (C) -- (E);
            \draw (D) -- (E);
            % Draw the vertices (circles only)
            \foreach \point in {A, B, C, D, E}
                \draw[draw, fill=white] (\point) circle [radius=0.2cm];
            % Draw the text on top of the circles
            \foreach \point/\name in {(A)/2, (B)/3, (C)/$q$, (D)/$p$, (E)/$f$}
                \node at \point {\scriptsize\name};
        \end{tikzpicture}
        \caption*{$\pgc(\Aut(\PSL(2,2^f)))$}
    \end{minipage}%
\end{figure}
\begin{proof}
    By \cite[Theorem 3.2]{FSGBook} and \thref{coprimeOfF}, $\Aut(\PSL(2, 2^f))\cong\PSL(2, 2^f)\rtimes C_f$ and $f$ is coprime to $|\PSL(2, 2^f)|$. It is known that the centralizer in $\Aut(\PSL(2, 2^f))$ of an element of order $f$ is isomorphic to $\operatorname{SL}(2, 2)$. $|\operatorname{SL}(2, 2)| = 6$, so $2-f, 3-f\not\in\pgc(\Aut(\PSL(2, 2^f)))$ and $p-f, q-f\in\pgc(\Aut(\PSL(2, 2^f)))$.
    
    $\PSL(2, 2^f)\cong\operatorname{Inn}(\PSL(2, 2^f))\leq\Aut(\PSL(2, 2^f))$, so $\pgc(\Aut(\PSL(2, 2^f)))[\pi(\PSL(2, 2^f))]\subseteq\pgc(\PSL(2, 2^f))$. Let $a-b\in \pgc(\PSL(2, 2^f))$. Then, there exists no element of order $ab$ in $\PSL(2, 2^f)$, and thus no element of order $ab$ in $\operatorname{Inn}(\PSL(2, 2^f))$. $f$ is prime so $a, b\nmid f$. By \cite[Lemma 1.3]{Suz}, $a-b\in\pgc(\Aut(\PSL(2, 2^f)))$. Thus, $\pgc(\PSL(2, 2^f)) = \pgc(\Aut(\PSL(2, 2^f)))[\pi(\PSL(2, 2^f))]$, so by the fact that $2-f, 3-f\not\in\pgc(\Aut(\PSL(2, 2^f)))$ and $p-f, q-f\in\pgc(\Aut(\PSL(2, 2^f)))$, the claim follows.

\end{proof}
% \begin{corollary}
%     Any rooted graph on 5 vertices with greater than 2 edges connected to the root or with a complete subgraph induced by the non-root vertices is not the prime graph complement of a strictly $\PSL(2, 2^f)$-solvable group, $f\geq 5$ prime.
% \end{corollary}
% \begin{corollary}
%     The subgraphs of $\Aut(\PSL(2, 2^f))$ listed in the are realizable by strictly $\PSL(2, 2^f)$-solvable groups:
% \end{corollary}
We now introduce some definitions:
\begin{notation} \thlabel{rootedGraphIsomorphismPSL22f}
    Given a group $G$ with $\pi(G) = \pi(\Aut(\PSL(2, 2^f))$ and a rooted graph $\Lambda$ on five vertices, it is understood that when we say $\pgc(G) \cong \Lambda$, we require the isomorphism to map the vertex $f$ to the root.
\end{notation}
\begin{notation}\thlabel{realizableDefPSL22f}
    Given a set of rooted graphs $\mathcal{H}$ on five vertices, we say that $\mathcal{H}$ is \emph{realizable} if for each $\Lambda \in \mathcal{H}$, there exists a $\PSL(2, 2^f)$-solvable $\pi(\Aut(\PSL(2, 2^f)))$-group $G$ such that $\pgc(G) \cong \Lambda$ in the sense of \thref{rootedGraphIsomorphismPSL22f}.
\end{notation}
\begin{fact}
     We list some subgraphs of $\pgc(\Aut(\PSL(2, 2^f)))$ containing a triangle and strictly $\PSL(2, 2^f)$-solvable groups that realize them:
    \begin{itemize}
        \item[-] \PSLhouse is realized by $\Aut(\PSL(2, 2^f))$.
        \item[-] \northStar is realized by $\Aut(\PSL(2, 2^f))\times C_3$.
        \item[-] \triangleWithShortTail is realized by $\Aut(\PSL(2, 2^f))\times C_2$.
        \item[-] \codart is realized by $\Aut(\PSL(2, 2^f))\times C_r$.
        \item[-] \happyFace is realized by $\Aut(\PSL(2, 2^f))\times C_6$.
        \item[-] \balloon is realized by $(\PSL(2,2^f) \times C_{p}) \rtimes C_f$ where $C_f$ acts on $\PSL(2,2^f)$ as in $\Aut(\PSL(2,2^f))$ and Frobeniusly on $C_p$.
    \end{itemize}
\end{fact}
\begin{lemma}\thlabel{fDisconnect}
    Let $G$ be a $\PSL(2, 2^f)$-solvable group. If $G\cong N.S$ for some $\operatorname{Inn}(\PSL(2, 2^f))\leq S\leq \Aut(\PSL(2, 2^f))$ and $f$ divides $|N|$, then $f-p\not\in\pgc(G)$ for all $p\in\pi(\PSL(2, 2^f))$.
\end{lemma}
\begin{proof}
 Suppose $f$ divides $|N|$. Then,$G\cong N.\PSL(2, 2^f)$ or $G\cong N.\PSL(2, 2^f).C_f$ by \cite[Lemma 2.1.1]{2023REU} and \cite[Theorem 3.2]{FSGBook}. In either case, take a subgroup $K\leq G$ such that $K\cong N.\PSL(2, 2^f)$. By the proof of \thref{generalDisconnect} and the fact that $f$ is coprime to $|\PSL(2, 2^f)|$ \thref{coprimeOfF}, $f-p\not\in\pgc(G)$ for all $p\in\pi(\PSL(2, 2^f))$.
\end{proof}
\begin{lemma}\thlabel{triangleContainmentGeneral}
    Let $G$ be a $\PSL(2, 2^f)$-solvable group such that $\pgc(G)$ contains a triangle $A$. $A\subseteq\pgc(G)[\pi(\Aut(\PSL(2, 2^f)))]$.
\end{lemma}
\begin{proof}
    By \cite{2015REU}, $G$ is a strictly $\PSL(2, 2^f)$-solvable group. $\pi(G)\setminus\pi(T)$ is triangle-free by \cite[Lemma 5.1]{Suz}, so an edge of the triangle is contained in $\pgc(G)[\pi(T)]$. By \cite[Lemma 2.1.1]{2023REU}, $G\cong N.S$ for a solvable group $N$ and a group $S$ such that $\operatorname{Inn}(\PSL(2, 2^f))\leq S\leq \Aut(\PSL(2, 2^f))$. $G$ satisfies the conditions of \cite[Lemma 5.4]{Suz} by \thref{generalDisconnect} and \thref{PSLGeneral2Sylow}. If $A\subseteq\pgc(G)[\pi(T)]$ the claim follows, so we may assume $|V(A)\cap\pi(T)|\leq 2$. By \cite[Corollary 5.6]{2023REU}, $|V(A)\cap\pi(N)| = \varnothing$. Thus, $V(A)\subseteq\pi(\Aut(\PSL(2, 2^f)))$, so $A\subseteq\pgc(G)[\pi(\Aut(\PSL(2, 2^f)))]$.
\end{proof}
\begin{lemma}\thlabel{oneEdgeMissingGeneralPSL}
    Let $G$ be a $\PSL(2, 2^f)$-solvable group such that $\pgc(G)$ contains a triangle. If there exists an edge between $f$ and some element of $\pi(T)$, $q-f\in\pgc(G)$ for the prime $q$ such that $q = 2^f - 1$.
\end{lemma}
\begin{proof}
    By \cite{2015REU}, $G$ is strictly $\PSL(2, 2^f)$-solvable. Let $A$ be a triangle of $\pgc(G)$. $A$ is contained in $\pgc(G)[\pi(\Aut(\PSL(2, 2^f)))]$ by \thref{triangleContainmentGeneral}. Thus, $G\cong N.S$ for a solvable group $N$ and some $S$ such that $\operatorname{Inn}(\PSL(2, 2^f))\leq S\leq \Aut(\PSL(2, 2^f))$. By \thref{generalDisconnect}, \cite[Theorem 3.2]{FSGBook}, and the fact that $f-t\in\pgc(G)$ for some $t\in\pi(\PSL(2, 2^f)); S = \Aut(\PSL(2, 2^f))$. $\pgc(G)[\pi(\Aut(\PSL(2, 2^f)))]\subseteq\pgc(\Aut(\PSL(2, 2^f)))$, so by \thref{Subgraph}, $q-f\in \pgc(G)$ or $p-f\in\pgc(G)$ for $p$ the prime such that $3p = 2^f + 1$. For contradiction, suppose $q-f\not\in\pgc(G)$. Then, by \thref{generalDisconnect}, $f\not\in\pi(N)$, so by \cite[Lemma 1.2]{Suz} $q\in\pi(N)$. $A\subseteq\pgc(G)[\pi(T)]$ by \thref{Subgraph} and the fact that $q-f\not\in\pgc(G)$. By \cite[Theorem 2.1]{MaximalSubgroupPSL}, $D_{2(2^f + 1)}\leq \Aut(\PSL(2, 2^f))$, so there exists a solvable group $K\leq G$ with $N\unlhd K$ and $K/N\cong D_{2(2^f +1 )}$. Then, $\{2, 3, p\}\subseteq\pi(K)$ and $q\in\pi(K)$, so $\pi(T)\subseteq\pi(K)$. Thus, $A\subseteq\pgc(G)[\pi(T)]\subseteq\pgc(K)[\pi(T)]$, so $\pgc(K)$ contains a triangle, a contradiction to \cite{2015REU}.
\end{proof}

We now prove certain rooted graphs on five vertices containing a triangle are not realizable by $\PSL(2, 2^f)$-solvable groups. We define three sets of rooted graphs on five vertices:
\begin{itemize}
    \item Let $\mathcal{F}$ be the set of rooted graphs on five vertices such that the root has degree greater than 2.
    \item Let $\mathcal{G}$ be the set of rooted graphs on five vertices that contains exactly the following elements:
    \begin{align*}
        \nobutt, \housenobrim\\
        \hamsandwichwitharm, \triangleIsolatedEdge\\
        \misshapenhouse, \badDart\\
        \bullcaseone, \zigZag, \hamSandwichWithOtherArm\\
        \trianglestick
    \end{align*}
\end{itemize}
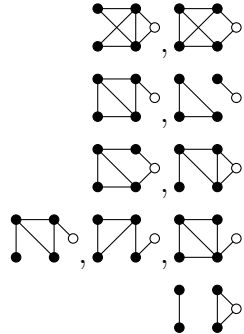
\begin{lemma}\thlabel{FnotRealizable}
    Let $\Xi\in\mathcal{F}$. $\Xi$ is not realizable by a $\PSL(2, 2^f)$-solvable group.
\end{lemma}
\begin{proof}
    Suppose otherwise for contradiction. There exists a $\PSL(2, 2^f)$-solvable group which realizes $\Xi$, so $\pgc(G)\subseteq\pgc(\Aut(\PSL(2, 2^f)))$. By \thref{Subgraph}, $f$ has degree at most two, a contradiction.
\end{proof}
\begin{lemma}\thlabel{notRealizable1}
    Let $\Xi\in\left\{\nobutt, \housenobrim\right\}$. $\Xi$ cannot be realized by a $\PSL(2, 2^f)$-solvable group.
\end{lemma}
\begin{proof}
    For contradiction, suppose otherwise. Then there exists a $\PSL(2, 2^f)$-solvable group $G$ realizing $\Xi$. $G$ must be strictly $\PSL(2, 2^f)$-solvable by \cite{2015REU} because there exist triangles in $\pgc(G)$. Then, the two vertices with edges to the root in $\pgc(G)$ must be $q$ and $p$ by \thref{Subgraph}. As such, the remaining non-root vertices without edges to the root must be 2 and 3, so $3-p\in\pgc(G)$. This is a contradiction because $3-p\notin\pgc(\PSL(2, 2^f))$ and because $\PSL(2, 2^f)$ is a section of $G$, $\pgc(G)[\pi(\PSL(2, 2^f))]\subseteq\pgc(\PSL(2, 2^f))$.
\end{proof}
\begin{lemma}\thlabel{notRealizable2}
    Let $\Xi = \hamsandwichwitharm$ (case (1)) or $\Xi = \triangleIsolatedEdge$ (case (2)). $\Xi$ cannot be realized by a $\PSL(2, 2^f)$-solvable group.
\end{lemma}
\begin{proof}
    Towards a contradiction, suppose otherwise. By \cite{2015REU}, because $\Xi$ is not triangle-free, there exists a strictly $\PSL(2, 2^f)$-solvable group $G$ that realizes $\Xi$. Then by \thref{oneEdgeMissingGeneralPSL}, the neighbor of $f$ in $\pgc(G)$ is $q$. In case (1), note $\pgc(G)[\pi(\PSL(2, 2^f))]\subseteq\pgc(\PSL(2, 2^f))$, so the two degree 3 vertices of $\pgc(G)$ not adjacent to $f$ must be $2$ and $q$. This implies $q$ is both adjacent to $f$ and not adjacent to $f$, a contradiction. In case (2), by \cite[Lemma 2.1.1]{2023REU}, $G\cong N.S$ for a solvable $N$ and some $S$ with $\PSL(2, 2^f)\cong\operatorname{Inn}(\PSL(2, 2^f))\leq S\leq\Aut(\PSL(2, 2^f))$. By \cite[Table 5]{LargeGroupsPSL}, $D_{2(2^f+1)}\leq \PSL(2, 2^f)$, so there exists a solvable $K$ such that $K/N\cong D_{2(2^f+1)}$. Thus, $\pi(K) = \{2, 3, p, q\}$ and $\pgc(G)[\pi(\PSL(2, 2^f))]\subseteq\pgc(K)$. Then, $\pgc(K)$ contains a triangle, a contradiction to \cite{2015REU}.
\end{proof}
\begin{lemma}\thlabel{notRealizable3}
   Let $\Xi = \misshapenhouse$ (case (1)) or $\Xi = \badDart$ (case (2)). $\Xi$ cannot be realized by a $\PSL(2, 2^f)$-solvable group. 
\end{lemma}
\begin{proof}
    Suppose otherwise for contradiction. Then there exists a strictly $\PSL(2, 2^f)$-solvable group $G$ that realizes $\Xi$. $G\cong N.S$ for solvable $N$ and $S$ such that $\operatorname{Inn}(\PSL(2, 2^f))\leq S\leq \Aut(\PSL(2, 2^f))$ by \cite[Lemma 2.1.1]{2023REU}. By \cite[Lemma 1.2]{Suz}, at least one of $p, q\in\pi(N)$ for (1) and at least one of $q, 3\in\pi(N)$ for (2). Let $a\in\pi(N)\cap\{p,q\}$ for (1) and $a\in\pi(N)\cap\{3, q\}$ for $(2)$. By \cite[Lemma 5.9]{Suz}, $2-a\not\in\pgc(G)$ for (1) and (2), contradicting the assumption that $G$ realizes $\Xi$.
\end{proof}
\begin{lemma}\thlabel{notRealizable4}
   For any $\Xi\in\left\{\bullcaseone, \zigZag, \hamSandwichWithOtherArm\right\}$, $\Xi$ cannot be realized with a $\PSL(2, 2^f)$-solvable group.  
\end{lemma}
\begin{proof}
    For contradiction, suppose otherwise. Then there exists a $\PSL(2, 2^f)$-solvable group $G$ such that $G\cong N.S$ for solvable $N$ and $\operatorname{Inn}(\PSL(2, 2^f))\leq S\leq \Aut(\PSL(2, 2^f))$. By \thref{oneEdgeMissingGeneralPSL}, $q-f\in\pgc(G)$. $2$ must be the degree 3 vertex not adjacent to $f$ by \thref{Subgraph}. Because $f-p\not\in\pgc(G)$ and $f$ is not isolated by \thref{fDisconnect}, $p\in\pi(N)$ by \cite[Lemma 1.2]{Suz}, so $2-p\not\in\pgc(G)$ by \cite[Lemma 5.9]{Suz}, a contradiction.
\end{proof}
% \begin{lemma}
%    The graph \zigZag cannot be realized with a $\PSL(2, 2^f)$-solvable group, $f\geq 5$ prime.  
% \end{lemma}
% \begin{proof}
%     For contradiction, suppose otherwise. Then there exists a $\PSL(2, 2^f)$-solvable group $G$ such that $G\cong N.S$ for solvable $N$ and $\operatorname{Inn}(\PSL(2, 2^f))\leq S\leq \Aut(\PSL(2, 2^f))$. By \thref{oneEdgeMissingGeneralPSL}, $q-f\in\pgc(G)$. Then, $2$ must be the 3 vertex by \thref{Subgraph}. Because $f-p\not\in\pgc(G)$ and $f$ is not isolated \thref{fDisconnect}, $p\in\pi(N)$, so $2-p\not\in\pgc(G)$ by \cite[Lemma 5.10]{Suz}, a contradiction.
% \end{proof}
\begin{lemma}\thlabel{notRealizable5}
    The graph $\Xi = \trianglestick$ cannot be realized by a $\PSL(2, 2^f)$-solvable group.
\end{lemma}
\begin{proof}
        Towards a contradiction, suppose otherwise. There exists a strictly $\PSL(2, 2^f)$-solvable group $G$ realizing $\Xi$ such that $G\cong N.S$ for solvable $N$ and $\operatorname{Inn}(\PSL(2, 2^f))\leq S\leq \Aut(\PSL(2, 2^f))$ by \cite[Lemma 2.1.1]{2023REU} and \cite{2015REU}. Notice that $2-3\in\pgc(G)$. At least two of $2, 3, p, q$ must divide $|N|$ by \cite[Lemma 1.2]{Suz}. If $a\in\{p,q\}$ divides $|N|$, there exists a solvable subgroup $K$ such that $K/N$ is isomorphic to a strict $\{f\}\cup\{p,q\}\setminus\{a\}$-subgroup by \cite[Lemma 5.5]{Suz} and we reach a contradiction to the main result of \cite{2015REU}. Then, $3$ divides $|N|$, but $2-3\not\in\pgc(G)$ by \cite[Lemma 5.9]{Suz}, a contradiction.
\end{proof}
Our results can be condensed into the following theorem:
\begin{theorem}
    Let $G$ be a $\PSL(2, 2^f)$-solvable group. Then, there exists no labeled graph isomorphism such that $\pgc(G)[\pi(\Aut(\PSL(2, 2^f)))]$ is isomorphic to $\Xi$ for any $\Xi\in \mathcal{F}\cup\mathcal{G}$.
\end{theorem}
\begin{proof}
    For contradiction, suppose otherwise. Then, there exists a $\PSL(2, 2^f)$-solvable group $G$ such that $\pgc(G)[\pi(\PSL(2, 2^f))]\cong\Xi$ by a labeled graph isomorphism for some $\Xi\in\mathcal{F}\cup\mathcal{G}$. By \cite[Theorem 4.4]{Suz}, there exists a Hall$-\pi(\Aut(\PSL(2, 2^f)))$ subgroup $H\leq G$ such that $\pgc(H)\cong\pgc(G)[\pi(\Aut(\PSL(2, 2^f)))]\cong \Xi$ (note that the proof of \cite[Theorem 4.4]{Suz} actually proves $\pgc(H) = \pgc(G)[\pi(H)])$), so $H$ realizes $\Xi$. By \thref{FnotRealizable}, \thref{notRealizable1}, \thref{notRealizable2}, \thref{notRealizable3}, \thref{notRealizable4}, and \thref{notRealizable5}, this is a contradiction.
\end{proof}
\begin{remark}
    There are a few rooted graphs that must still be either realized or ruled out to completely classify $\PSL(2, 2^f)-$solvable groups for prime $f\geq 5$ that we could not classify as the necessary information on modular representations of arbitrary $\PSL(2, 2^f)$ is unavailable. We have listed the graphs below:
    \begin{center}
    \dart, \cricket, \triangleWithShortTail\\
    \bowtiebruh, \triangleWithTwoTails,     \triangleWithTail, \bull, 
    \emptyhouse
\end{center}
Of these, we conjecture that only the graphs in the first row are realizable by $\PSL(2, 2^f)-$solvable groups for prime $f\geq 5$. We draw this conjecture from \cite[Theorem 7.21]{Suz} and the results throughout this section that showed $\pgc(\PSL(2, 2^f))$ tends to behave similarly to $\pgc(\PSL(2, 2^5))$ for prime $f\geq 5$.
\end{remark}

% \subsection{Graphs that we suspect are impossible}
% \begin{center}
%     \bowtiebruh, \triangleWithTwoTails,  \triangleWithTail, \bull, \emptyhouse
% \end{center}
% \subsection{Graphs that we suspect are realizable}
% \begin{center}
%     \dart, \balloon, \cricket
% \end{center}

\section{Outlook}
This section outlines some areas for future work.
A natural place to start would be with classifying the prime graph complements of $\PSL(2, 13)$-solvable groups. A more ambitious goal would be to finish classifying the prime graph complements of $\PSL(2, q)$-solvable groups for $q\geq 5$ a prime power of 2. An even more ambitious project would be to finish the classification of prime graph complements of $\PSL(2, q)$-solvable groups for all $q$ for which $\PSL(2, q)$ is a $K_4$ group, which would also complete the classification of single $K_4$-solvable groups. After that, one could begin the classification of $K_5$-solvable groups. A different direction would be classifying the prime graph complements of all strictly $T$-solvable groups where $T$ is a nonabelian simple group that has already been classified. Another possible direction would be to classify the prime graph complements of $\mathcal{T}$-solvable groups where $\mathcal{T}$ is a set of distinct nonabelian simple groups. Finally, one could prove or find a counterexample to the following conjecture, which would be useful for classifying the prime graph complements of $\PSL(2, q)$-solvable groups in conjunction with the main results of this paper:
\begin{conjecture}\label{conjectureLast}
    Suppose $\PSL(2, q)$ and $\PSL(2, p)$ are $K_4-$groups. Suppose one of the following holds:
    \begin{itemize}
        \item $q, p$ are prime and $q, p\equiv 1(\operatorname{mod} 12)$.
        \item $q, p$ are prime and $q, p\not\equiv 1(\operatorname{mod} 12)$.
        \item $p = 2^f, q = 2^g$ for primes $f, g\geq 5$.
        \item $p = 3^f, q = 3^g$ such that $f, g\neq 4$.
    \end{itemize}
    Then $\Xi\cong\pgc(G)$ for some $G$ a $\PSL(2, q)$-solvable group if and only if $\Xi\cong\pgc(H)$ for some $H$ a $\PSL(2, p)$-solvable group.
\end{conjecture}
The groups classified in this paper helped form this conjecture, along with the observations that most of the classification results depended most heavily on the representation information of the groups $T$ and the observation that the groups which have representation information available in \cite{GAP} fall roughly into the four classes outlined above.
\begin{appendices}
\section{} \label{AppendixB}
The code used to find the structure of subgroups of groups and to calculate fixed points of representations can be found at \url{https://github.com/gabriel-roca/2024-K4-Groups/tree/main}. An explanation of the methods used to to calculate fixed points of representations via Brauer tables (and some of the more complicated ordinary character tables) can be found in the appendices of \cite{Suz}.
\end{appendices}
\section{Acknowledgements}
This research was conducted under NSF-REU grant DMS-2150205 and NSA grant H98230-24-1-0042 under the mentorship of the first author. We would like to thank Gavin Pettigrew, Saskia Solotko, and especially Lixin Zheng from the 2023 Texas State REU Team for their help, and the fellow students at the Texas State University 2024 REU for their support. We are also indebted to Texas State University for access to its facilities, and to the NSF for funding this research. We would also like to thank Derek Holt and Caroline Lassueur for answering questions of ours. Finally, we would like to thank Alexander Hulpke for some help with GAP.

\end{document}